\documentclass[12pt]{amsart}

\usepackage{amsmath,amsfonts,amsthm,amssymb}
\usepackage{graphicx}
\usepackage{esint}
\usepackage[shortlabels]{enumitem}
\usepackage{color}
\usepackage[colorlinks=true]{hyperref}
\hypersetup{urlcolor=blue, linkcolor=blue, citecolor=red, anchorcolor=blue]}
\usepackage[margin=1.3in]{geometry}
\usepackage{tikz}
\usetikzlibrary{backgrounds}

\newtheorem{theorem}{Theorem}

\newtheorem{lemma}[theorem]{Lemma}
\newtheorem{corollary}[theorem]{Corollary}
\newtheorem{problem}[theorem]{Problem}
\newtheorem{remark}[theorem]{Remark}

\newtheorem{remarks}[theorem]{Remarks}

\def\un{\mathbf{1}}

\newcommand{\R}{{\mathbb R}}
\newcommand{\N}{{\mathbb N}}

\newcommand{\E}{{\mathcal E}}
\newcommand{\cI}{{\mathcal I}}

\newcommand{\cJ}{{\mathcal J}}
\newcommand{\Id}{\hbox{Id}}
\newcommand{\U}{\text{{\bf 1}}}
\newcommand{\V}{\text{{\bf 0}}}
\newcommand{\OO}{\mathbb{O}}
\newcommand{\sW}{{\mathsf W}}

\newcommand{\I}{{\text{Id}}}

\renewcommand{\S}{{\mathbb S}}

\renewcommand{\(}{\left(}
\renewcommand{\)}{\right)}
\newcommand{\p}{\partial}
\newcommand{\var}{\varepsilon}
\newcommand{\dd}{\, \mathrm{d}}
\newcommand{\cU}{\mathcal U}
\newcommand{\cT}{\mathcal T}
\newcommand\ds{{\frac{\mathrm d}{\mathrm ds}}}

\definecolor{darkgreen}{rgb}{0,0.4,0}

\makeatletter
\@namedef{subjclassname@2020}{%
 \textup{2020} Mathematics Subject Classification}
\makeatother

\begin{document}

\title[Quantitative De Giorgi methods]{Introduction to quantitative De Giorgi methods}
\author[G.~Brigati]{Giovanni Brigati}
\address[G.~Brigati]{Institute of Science and Technology Austria (ISTA), Am Campus 1, 3400 Klosterneuburg, Austria}
\email{giovanni.brigati@ist.ac.at}
\author[C.~Mouhot]{Cl\'ement Mouhot}
\address[C.~Mouhot]{Department of Pure Mathematics and Mathematical Statistics, University of Cambridge, Wilberforce Road, Cambridge CB3 0WA, UK}
\email{c.mouhot@dpmms.cam.ac.uk}

\begin{abstract}
  The theory of De Giorgi (1958) and Nash (1959) solves Hilbert's 19th problem and constitutes a major advance in the analysis of PDEs in the 20th century. This theory concerns the Hölder regularity of solutions to elliptic and parabolic equations with non-regular coefficients, and it was extended by Moser (1960) to include the Harnack inequality. This course reviews the classical De Giorgi method in the elliptic and parabolic cases and introduces its recent extension to hypoelliptic equations which appear naturally in kinetic theory. The simplest case is the Kolmogorov equation with a rough diffusion coefficients matrix in the kinetic variable. We present compactness arguments but emphasize the recently developed quantitative methods based on the construction of trajectories. These lecture notes are self-contained and can be used as a general introduction to the topic.
\end{abstract}


\maketitle

\setcounter{tocdepth}{2}
\tableofcontents

\section{Introduction}
\label{Sec:Intro}

This section presents the motivation and scientific context. The new theory we will present in the last section can be called \textbf{regularity theory of kinetic equations with rough coefficients}. But it is a new extension of a classical one. It has three independent origins which we briefly summarise in the three next subsections. 

\subsection{The solution to Hilbert's $19$th Problem}
At the ICM (International Congress of Mathematicians) of 1900, \cite{Hilbert1900}, D.~Hilbert posed twenty-three mathematical problems to the mathematical community for the incoming century. The $19$th of them asked \textit{which are the regular variational problems}.

Hilbert defined a \textbf{regular variational problem} as  a variational problem whose Euler-Lagrange equation is an elliptic partial differential equation with analytic coefficients. Therefore, concretely Hilbert's $19$th problem asks whether, for the class of nonlinear partial differential equations arising as Euler-Lagrange equations of such variational problem, solutions are analytic; this implies that the local minimisers of the original variational problem exists and are analytic. The underlying motivation is that Laplace's equation admits analytic solutions only, and, at the same time, it is the Euler--Lagrange equation associated with the minimisation of the Dirichlet energy over a bounded domain. Hilbert's $19$th problem consists therefore in determining how general this observation is.

Hilbert's question can be reduced to investigating the smoothness of local minimisers, as analyticity follows by studying the convergence of the Taylor series, as observed by Bernstein~\cite{bernstein1904nature}. Let $\cU$ be a bounded domain of $\mathbb R^n$. Let $\mathrm{L}: \mathbb R^n \to \mathbb R$ be a \textbf{Lagrangian} function, assumed to be convex. We then consider the variational problem 
\begin{equation}
  \label{eq:pbvar}
    \min_{u} \, \E(u) := \int_{\cU} \mathrm L(\nabla_x u) \dd x, 
\end{equation}
where $u: \cU \to \mathbb R$ varies in an appropriate functional space (which can and often does include suitable boundary conditions). If the Lagrangian is given by
\begin{equation}
  \label{eq:dirichlet-lag}
  \mathrm{L}(p) = \frac{1}{2} \,|p|^2,
\end{equation}
then $\E$ is the Dirichlet energy, and solutions to \eqref{eq:pbvar} satisfy the linear Laplace equation
\begin{equation*}
  \Delta u = 0 \ \text{ on } \ \cU
\end{equation*}
with appropriate boundary conditions. If
\begin{equation}
  \label{eq:minimal-lag}
  \mathrm{L}(p) = \sqrt{1+|p|^2}
\end{equation}
is the area element of the surface defined by the height function $u$, then $\E$ is the area functional, and the variational problem corresponds to finding \textbf{minimal surfaces}. The nonlinear Euler-Lagrange equation is then
\begin{equation}
  \label{eq:min-surf}
  \nabla_x \cdot \( \frac{\nabla_x u}{\sqrt{1+|\nabla_x u|^2}} \) = 0 \ \text{ on } \ \cU
\end{equation}
with, again, appropriate boundary conditions. 

In general, we derive the Euler-Lagrange equations associated with \eqref{eq:pbvar}. Let $u$ be a solution to \eqref{eq:pbvar}, then for all $w \in \mathrm{C}^\infty_c(\cU)$ we have 
\begin{equation*}
  \forall \,  \var>0, \int_{\cU} \mathrm L\( \nabla_x u+\var \nabla_x w\) \dd x \geq \int_{\cU} \mathrm L(\nabla_x u) \dd x.
\end{equation*}
By Taylor-expanding the left-hand side of the last display, and dividing by $\var$, we get 
\begin{equation*}
  \int_{\cU} \Big[ \nabla \mathrm L(\nabla_x u) \, \cdot \, \nabla_x w \Big] \dd x = 0,
\end{equation*}
which is the weak formulation for (the notation ``$:$'' stands for contracting tensors)
\begin{equation}
  \label{eq:EL1}
  \nabla_x \, \cdot \, \big[ \nabla \mathrm L (\nabla_x u) \big] = \left[\nabla^2 L(\nabla_x u)\right] : \nabla^2_x u = 0
\end{equation}
in the space of distributions $\mathcal D'(\cU)$. This is an elliptic PDE in non-divergence form in general. Let us reason a priori and consider a partial derivative $f = \partial_{x_k} u$ for $k=1,\dots,n$. Then, if $u$ solves \eqref{eq:EL1}, the new unknown $f$ solves (repeated indices are summed)
\begin{equation}
  \label{eq:EL2}
    \partial_{x_i} \cdot (a_{ij}(x) \, \partial_{x_j} f) =0, \quad \text{ with } \quad a_{ij}(x) = \left[\nabla^2 \mathrm{L}(\nabla_x u (x)) \right]_{ij},
\end{equation}
which is a second-order partial differential equation in divergence form. Note that when the matrix of coefficients $A = (a_{ij})$ is definite positive---which corresponds to the strict convexity of $L$---then $u$ satisfies the maximum principle and Bernstein's method allows estimating the gradient of $u$ in a pointwise manner (see for instance~\cite{zbMATH00947827}). Therefore, given suitable boundary conditions, we can assume $\nabla u \in L^\infty$.

In these lecture notes, we focus on the \textbf{interior regularity} of solutions to equations of the form~\eqref{eq:EL2}, and of their kinetic counterparts. This means working away from the boundary, thus we will not discuss further boundary conditions. The regularity at the boundary is an interesting problem, still currently mostly open in kinetic theory. Solving the $19$th Hilbert problem consists in establishing the analytic regularity of solutions to~\eqref{eq:EL1}-\eqref{eq:EL2}. The difficulty is that the matrix of coefficients $A = (a_{i,j})$ depends on the gradient of the solution $\nabla u$, resulting in a non-linear bootstrap. 

We make the following hypotheses on $\mathrm{L}$.
\begin{align}
  \label{eq:H1}\tag{H1}
  &\mathrm{L} \in \mathrm{C}^{\infty}(\mathbb R^n), \\
  \label{eq:H2}\tag{H2}
  &\exists \, \Lambda >0 \, : \, \forall \, p \in \mathbb R^n, \ \Lambda^{-1} \le \nabla^2 \mathrm{L}(p) \le \Lambda.
\end{align}
The first hypothesis can be relaxed appropriately when one is interested in establishing only finite regularity of the solution $u$. The second hypothesis is a uniform coercivity/ellipticity condition on the matrix $A$, combined with an upper bound. In terms of $L$, it corresponds to a uniform convexity combined with a subquadratic control on the growth, together with a pointwise bound on the second derivative. 

Going back to our two examples, the Dirichlet Lagrangian~\eqref{eq:dirichlet-lag} yields $A = \nabla^2 L= \I$ which satisfies our assumption, and the minimal surface Lagrangian~\eqref{eq:minimal-lag} yields
\begin{equation*}
  \nabla^2 _p L(p) = \frac{\I + \( |p|^2 \I - p \otimes p \)}{\(1+|p|^2\)^{\frac32}} 
\end{equation*}
which satisfies a modified form of the hypothesis when we restrict $p$ to a bounded region (this turns out to be sufficient for the local regularity theory).

Here are chronologically the main steps in the resolution of Hilbert's $19$th problem:
\begin{itemize}
\item[1904:] Bernstein \cite{bernstein1904nature} proves in dimension $n=2$ that analyticity follows from $C^3$ regularity for solutions to \eqref{eq:pbvar}-\eqref{eq:EL1}, under our assumption on $L$, see also \cite{petrowsky1939analyticite} for general dimensions.
\item[1934:] Schauder \cite{schauder1934lineare} proves that if $A \in C^\alpha _{\mathrm{loc}}(\cU)$ for $\alpha \in (0,1)$ (Hölder regularity), then the solution $u$ to the non-divergence elliptic equation~\eqref{eq:EL1} is $C^{2+\alpha}_{\mathrm{loc}}(\cU)$. A bootstrap argument can then be implemented by differentiating the equation (the Schauder theory survives when lower order terms are added). It shows that solutions $u$ are $\mathrm{C}^\infty$. Since we know already $\nabla_x u \in L^\infty$, the missing link is the Hölder regularity of $\nabla_x u$ (implying that of $A$), which triggers Schauder's estimates. 
\item[1938:] Morrey \cite{morrey1938solutions} gives a solution which is specific to the two-dimensional case $n=2$.
\item[1957:] De Giorgi and Nash independently \cite{de1957sulla,nash1958continuity} solve the problem. De Giorgi shows that weak $L^2$ solutions to
  \begin{equation}
    \label{eq:EL3}
    \begin{cases}
      & \nabla_x \cdot \(A \, \nabla_x f\) =0 \ \text{on } \ \cU, \\[1mm]
      & A \ \text{ symmetric measurable and } \  \Lambda^{-1} \le A \le \Lambda \ \text{ for some } \ \Lambda >0,
    \end{cases}
  \end{equation} 
  are such that $f \in  \mathrm{C}^{\alpha}_{\mathrm{loc}}(\cU)$ for some $\alpha >0$. Such an elliptic PDE is said to have \textbf{rough coefficients}. Nash proves a similar theorem for linear parabolic equations with rough coefficients. 
\item[1960:] Moser~\cite{moser1960new,MR159138,moser1964harnack} introduces an alternative approach closer to ``energy estimates'' in PDE analysis, which contains many influencing ideas. The related Kruzhkov contribution~\cite{kruzhkov1963priori} simplifies this approach. In particular, these papers contain the first proof of the \textbf{Harnack inequality} for elliptic and parabolic equations with rough coefficients.
\end{itemize}

Together, the last steps of this program are referred to as the \textbf{De Giorgi--Nash--Moser theory} (DGNM). The viewpoint adopted in this theory, which had a huge influence in PDE analysis, is that \emph{the nonlinearity is traded for the understanding of a linear problem with \textbf{low regularity coefficients}}.

Nash performed integral estimates on the fundamental solution to \eqref{eq:EL3} and its logarithm, while De Giorgi studied the level sets, see for instance~\cite{chiarenza1986harnack,de2016john} for modern presentations. De Giorgi's proof is elliptic, while Nash's proof is parabolic, but both approaches can treat both cases. They also both adapt thoroughly to the following \emph{hypoelliptic equations of type I} (see the next subsection)
\begin{equation*}
  \sum_{i,j=1}^m D_i \left(a_{ij} D_j f\right) = 0
\end{equation*}
where $(D_i)_{i=1}^m$ is a collection of first-order differential operators such that they generate a Lie algebra of full rank
\begin{equation*} \mathrm{Lie}\(D_1,\cdots,D_m\)=\mathrm{span}\(\partial_{x_1},\cdots,\partial_{x_n}\),
\end{equation*}
where the Lie bracket of the Lie algebra is the commutator of vector fields. The hypotheses on the matrix $A =(a_{i,j})_{1\le i,j \le m}$ are the same as in~\eqref{eq:EL3}. This case is important in sub-Riemannian geometry, see for instance~\cite{agrachev2019comprehensive}, but will not be discussed in these notes.

The classical DGNM theory had two important restrictions. The first one is that it deals with divergence-form PDEs. This is not merely a technical restriction, as all approaches---by De Giorgi, Nash or Moser---crucially rely on \textbf{energy estimates}. These energy estimates are also called \textbf{Caccioppoli inequalites}: see~\eqref{eq:ee2} in the elliptic case and~\eqref{eq:caccio-para} in the parabolic case and~\eqref{eq:caccio-kin} in the kinetic (hypoelliptic) case.

A similar result of Hölder regularity was later obtained by Krylov and Safonov~\cite{krylov1980certain} for solutions to non-divergence equations of the form 
\begin{equation*}
  a_{ij}(x) \, \partial^2_{x_i x_j} f = 0,
\end{equation*}
with a symmetric matrix $A =(a_{i,j})_{i,j}$ of rough coefficients that satisfies $\Lambda^{-1} \le A \le \Lambda$ for some $\Lambda>0$. We also refer to~\cite{MR164135,MR1487894,MR420016} for important previous contributions. The methods are however different in the non-divergence case.

The second restriction of the classical DGNM theory is that it does not cover the \emph{hypoelliptic case of type II} (see the next subsection) which is ubiquitious in physics, and such an extension is the object of these lecture notes.

We however also focus in these notes on obtaining constructive proofs with quantitative estimates on the constants and exponents. This leads us to revisit certain results even in the classical elliptic and parabolic cases.

There are two motivations behind this. The first motivation is rooted in mathematical physics. The classical DGNM theory was developed to solve Hilbert's $19$th problem about the regularity of energy minimisers. And the extension of this classical theory to the hypoelliptic equations appearing in kinetic theory was developed to obtain the conditional regularity of solutions to the Landau equation (and to the Boltzmann equation for long-range interactions, in the non-local case). In both problems, a mathematical solution is only truly satisfying if the constants of regularity obtained are of orders compatible with the physics itself. And it is difficult to ensure this without developing quantitative methods of proof. The second motivation is internal to mathematics. The search for quantitative methods often deepens the understanding of the mathematical structures and suggests new connexions. In our case, it has lead us to connecting the control of integral oscillation (through inequalities of Poincaré-type) to that of constructing particular trajectories solving a control problem. 

\subsection{Hörmander's hypoellipticity theory}

The DGNM theory revolutionised the study of elliptic and parabolic equations. It was however not covering a fundamental class of related equations, whose importance has been recognised at least since Kolmogorov~\cite{kolmogoroff1934zufallige}. In this note, Kolmogorov investigates a model equation driven by a first-order differential operator combined with diffusion only in certain directions. Let $x,v$ be the variables respectively in $\R^d$. Then, \textbf{Kolmogorov's equation} reads, in its simplest version:
\begin{equation}
    \label{eq:VFP}
    \begin{cases}
    \partial_t f + \, v\cdot \nabla_x f = \Delta_v f, \quad f=f(t,x,v), \\[2mm]
    f(0,\cdot,\cdot) = f_0 \ \text{ which is a probability density on } \R^{2d}.
    \end{cases}
\end{equation}

This equation corresponds to the evolution of the law of the random process 
\begin{equation}\label{eq:langevin}
    \begin{cases}
        \mathrm dX_t  =  V_t \dd t, \\
        \mathrm dV_t  \, =  \mathrm dW_t,
    \end{cases}
\end{equation}
where $(W_t)_t$ is a standard $d-$dimensional Brownian motion. Indeed, \eqref{eq:langevin} is a time-integrated Brownian motion, which gives its name to the note~\cite{kolmogoroff1934zufallige}.

Kolmogorov computes an explicit formula for the fundamental solution to~\eqref{eq:VFP}, by combining Fourier analysis and Duhamel's principle. Given the initial data $f_0 = \delta_{x_0,v_0}$ in~\eqref{eq:VFP}, one gets 
\begin{equation}\label{eq:kol}
  G(t,x,v) = \left( \frac{\sqrt{3}}{2 \, \pi \, t^2} \right)^d \, \mathrm{exp} \left( - 3 \, \frac{\left|x-x_0 - t \, (v+v_0)/2 \right|^2}{t^3} - \frac{|v-v_0|^2}{4t} \right).
\end{equation}

The fundamental solution~\eqref{eq:kol} is smooth in all variables for $t>0$, and therefore all solutions to~\eqref{eq:VFP} are smooth as well for $t>0$. However, ellipticity in $x$ fails for~\eqref{eq:VFP}. Intuitively, the noise in $v$ is ``transferred'' to the $x$ variable by the transport drift operator $v \cdot \nabla_x$, see Figure~\ref{fig:fig2}.
\begin{figure}
  \includegraphics[scale=0.7]{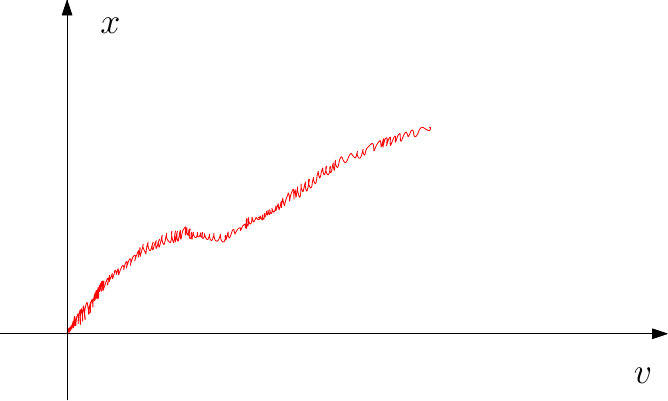}
  \caption{Noise in $v$ propagates to the $x$ variable thanks to the transport drift $v \cdot \nabla_x$.}
  \label{fig:fig2}
\end{figure}

After some seminal works in specific cases such as~\cite{nirenberg1963solvability}, a complete understanding of this phenomenon was developed by Hörmander~\cite{hormander1967hypoelliptic} in 1967. His starting point was Kolmogorov's calculation showing that for second-order operators with smooth coefficients, while ellipticity is a sufficient condition for the smoothness of the fundamental solution (he called the latter property ``hypoellipticity''), it is \emph{not} necessary. For the Kolmogorov equation, regularisation is produced by the interaction between the vector fields  $\nabla_v$ and $v \cdot \nabla_x$.

More precisely, we say that a linear operator $L$ is \textbf{hypoelliptic} if 
\begin{equation*}
  L u \in \mathrm{C}^\infty_{\mathrm{loc}}(\cU) \ \Longrightarrow \ u \in \mathrm{C}^\infty_{\mathrm{loc}}(\cU),
\end{equation*}
in an open subset $\cU \subset \R^n$ (again, we do not address the issue of the boundaries). When restricting to second-order operators with smooth coefficients, the standard \textbf{ellipticity}, i.e. the coercivity of the principal symbol, implies this condition. It is however not necessary. Consider, for $m \in \N^*$, a family $(X_i)_{i=1}^m$ of smooth vector fields and the operator 
\begin{equation}
  \label{eq:hor}
  L = X_0 + \sum_{i=1}^m X_i^2.
\end{equation}
The main theorem of \cite{hormander1967hypoelliptic} is a necessary and sufficient condition for operators of the form~\eqref{eq:hor} to be hypoelliptic: the Lie algebra of $X_0,\dots,X_m$ must satisfy
\begin{equation*}
  \dim \mathrm{Lie}(X_0,\cdots,X_m)= n.
\end{equation*}
When $X_0=0$ (no first-order term) such operators write as a sum of squares and we call them \textbf{hypoelliptic operators of type I}; otherwise we call them \textbf{hypoelliptic operators of type II}. Hypoelliptic operators of type II have also been called \textbf{ultraparabolic} in the literature.

The Kolmogorov operator is the archetypical example of the type II. Let 
\begin{equation}
  \label{eq:Kop}
  \mathcal{K} := \partial_t + v \, \cdot \, \nabla_x + \sum_{i=1}^d \partial^\star_{v_i} \partial_{v_i},
\end{equation}
be the operator driving \eqref{eq:VFP}, with $t \in \R$, $x,v\in \R^d$.
This has the form~\eqref{eq:hor} with
\begin{equation*}
  n=2d+1, \quad X_0 = \partial_t + v \, \cdot \, \nabla_x, \ \text{ and } \ X_i = \partial_{v_i} \ \text{ for } \  i=1,\cdots,d.
\end{equation*}

The fundamental computation is  
\begin{equation*}
  [X_0,X_i] = - \partial_{x_i}, \quad \forall \,  \, i=1,\cdots,d
\end{equation*}
which allows recovering all $\partial_{x_i}$'s, and finally $\partial_t$ is recovered by linear combination. Hence,
\begin{equation*}
  \dim \mathrm{Lie}(X_0,X_1,\cdots,X_d)=n
\end{equation*}
at every point and $\mathcal{K}$ is hypoelliptic. By inspecting the $(x,v)$-Fourier transform of the Kolmogorov fundamental solution, one sees that hypoelliptic operators have in general \emph{Gevrey} rather than \emph{analytic} regularisation (with exponent $2/3$ in the case of the Kolmogorov operator, but this exponent depends on the Lie algebra).

\subsection{Mathematical kinetic theory}

The bedrock of kinetic theory is Newton's classical mechanics. Along Newton's theory, decisive progresses were later made for electrico-magnetic forces (Amp\`ere, Faraday, Maxwell\dots), large velocities and large scales (Lorentz, Poincar\'e, Einstein, Minskowski\dots), small scales and quantum physics (Planck, Einstein, Bohr, Heisenberg, Schr\"odinger, Dirac, Bose\dots). However, all these works concern the study of a single object (a planet, an electron,\dots) or a small number of those, such as the planets in the solar system or the electrons in an atom. Many natural phenomena pertain to systems made of a very large number of objects (or particles, or agents\dots). The observable universe contains $10^9$--$10^{12}$ galaxies, each of them made of billions of stars (about $10^{10}$ in the Milky Way). A gas is composed of about $10^{23}$ molecules per mole (the Avogadro constant). Crowds, flocks of birds, shoals of fishes, are made up of thousands of individuals (sentient or not). In such situations, it is often impossible to track each particle/molecule/indidividual.

In the 1867 paper~\cite{maxwell1867iv} (and in subsequent papers), Maxwell laid the mathematical foundations of statistical mechanics. The statistical approach was developed to understand these large systems through their statistical distributions in a way that is rigorously connected to the microscopic dynamics. This is in contrast with the previously invented fluid mechanics, which combines conservation laws and adhoc closure laws to derive evolution equations at the macroscopic level. The theory invented by Maxwell is called \emph{kinetic} due to the presence of a ``kinetic'' variable, the microsopic velocity, which is inaccessible to observation. We have therefore three scales of description of complex systems, ordered by increasing level of detail: 
\begin{enumerate}
    \item a \textbf{macroscopic} scale (fluid dynamics, observable);
    \item a \textbf{mesoscopic} scale (statistical mechanics, where the behaviour of the unknown in some variables is only observable on average);
    \item a \textbf{microscopic} scale (Newton's mechanics, fully not observable).
\end{enumerate}

At the mesoscopic scale, the fundamental equation---and also the oldest one---is the \textbf{(Maxwell-)Boltzmann equation}, invented in~\cite{maxwell1867iv,boltzmann1872weitere}: \begin{equation}\label{eq:boltz}
  \partial_t f + v\, \cdot \, \nabla_x f = Q(f,f).
\end{equation}
The unknown is a non-negative time-dependent probability distribution function $f=f(t,x,v)$, with $x \in \Omega \subset \mathbb R^d$, $v \in \mathbb R^d$, and $\Omega$ a spatial domain with suitable boundary conditions. The  \textbf{collision operator} $Q$ is given by 
 \begin{multline*}
   Q(f,f) (t,x,v) = \\
   \int_{v_\star \in \mathbb R^d} \, \int_{\sigma \in \S^{d-1}} \mathrm{B}(v-v_\star,\omega) \, \left[ f(t,x,v') \, f(t,x,v'_\star) - f(t,x,v) \, f(t,x,v_\star)  \right] \dd \omega \dd v_\star,
 \end{multline*}
 where the variables $v_\star$, $v'$, $v'_\star$ are related to $v$ as follows (given $\sigma \in \S^{d-1}$)
 \begin{equation*}
   v' = \frac{v+v_\star}{2} + \frac{|v-v_\star|}{2} \sigma, \qquad 
   v'_\star = \frac{v+v_\star}{2} - \frac{|v-v_\star|}{2} \sigma.
 \end{equation*}
 This is a parametrisation of the degrees of freedom left over once the conservation laws of the elastic binary collision between particles with identical mass---see Figure~\ref{fig:fig1}---have been taken into account:
 \begin{equation*}
   \begin{cases}
     v + v_\star = v' + v'_\star, \qquad &\text{(conservation of momentum)} \\[2mm]
     |v|^2+|v_\star|^2 = |v'|^2 + |v'_\star|^2, \qquad &\text{(conservation of kinetic energy)}.
   \end{cases}
\end{equation*}
Observe that $Q(f,f)$ statistically accounts for collisions at time $t$, at a point $x$, between two particles with pre-collisional velocities $v',v'_\star$, and post-collisional velocities $v,v_\star$, see Figure~\ref{fig:fig1}. The collision kernel $\mathrm{B}$ encodes the frequencies at which collisions with different pairs of velocities occur, and $Q(f,f)$ collects the weighted balance-sheet of all possible collisions in the velocity space. $Q$ is a bilinear nonlocal operator acting only in the kinetic variable $v$.
\begin{center}
  \begin{figure}
    \begin{minipage}{6in}
      \centering
      \raisebox{0.65\height}{\includegraphics[height=0.9in]{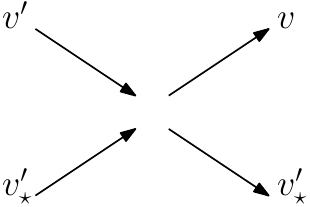}}
      \hspace*{3cm}
      \includegraphics[height=2in]{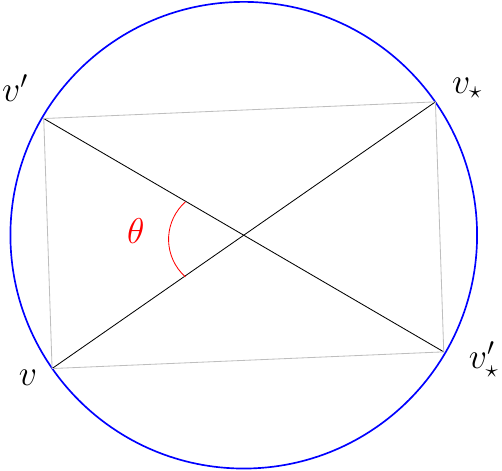}
    \end{minipage}
    \caption{A collision between two particles.}
    \label{fig:fig1}
  \end{figure}
\end{center}

By invariance (see~\cite{bobylev1976fourier}), the collision kernel $B$ is a function of $|v-v_\star|$ and $\cos \theta$, so that $\mathrm{B} = \mathrm{B}(|v-v_\star|,\cos \theta)$, 
where $\theta$ is the angle between $v - v_\star$ and $v'-v'_\star$, i.e. between $v-v_\star$ and $\sigma$. The computation of Maxwell shows that $\mathrm{B} = |v-v_\star|$ for hard spheres interacting by elastic collisions. When particles interact through a repulsive force proportional to $r^{-\alpha}$ with $r$ the distance between particles and $\alpha > 2$, we have in dimension $d=3$
\begin{equation*}
  B = |v-v_\star|^\gamma \, b(\cos \theta), \quad b(\cos \theta) \sim_{\theta \to 0} \theta^{-2-2s}, 
\end{equation*}
with
\begin{equation*}
  \gamma = \frac{\alpha-5}{\alpha -1}, \qquad s = \frac{1}{\alpha-1}.
\end{equation*}
Except in the hard-sphere case, the kernel $B$ is therefore \emph{not} integrable on $\S^{d-1}$:
\begin{align*}
  \int_{\S^2} B
  & = |v-v_\star|^\gamma \int_{\S^2} b(\cos \theta) \dd \sigma \\
  & = 2 \pi |v-v_\star|^\gamma \int_0 ^\pi b(\cos \theta) \sin \theta \dd \theta \\
  & \approx 2\pi |v-v_\star|^\gamma \int_0 ^\pi \theta^{-1-2s} \dd \theta = + \infty.
\end{align*}
This singularity in $\theta$ implies that $Q$ manifests non-local fractional ellipticity. Artificially removing such singularity is called \textbf{Grad's cutoff}, and accordingly the non-modified regime is called \textbf{non-cutoff long-range interaction}. The challenge of extending the De Giorgi--Nash--Moser theory to the Boltzmann equation with non-cutoff long-range interaction was, to our knowledge, first raised in~\cite{peccot-villani}. 

It helps intuition if one draws an analogy with local differential operators. The dependence of $\mathrm B$ on $|v-v_\star|$ is analogous to the growth or decay of the coefficients of a local differential operator, while the singularity of order $2s$ at $\theta \sim 0$ is reminiscent of the order of differentiability for this operator.

It means heuristically that one expects the collision operator $Q$ to share similar features with the following rough approximation
\begin{equation*}
  Q \approx \nabla_v^s \cdot \, \( A \, \nabla_v^s \) \ + \text{  lower order terms,}
\end{equation*}
with $A$ growing like $|v|^\gamma$ at large $v$. This analogy is backed up by the regularity estimates~\cite{MR1765272,MR2679369} and by the spectral gap estimates~\cite{MR2322149,MR2784329}. Moreover, for $\gamma = 0$, Fourier methods apply just like they do for constant coefficients PDEs. And when $\alpha \to 2$ and $s\to1$ the Boltzmann kernel indeed resembles more and more a local diffusion operator (with diffusion coefficients depending on the solution in a non-local manner). The limit case $\alpha = 2$ corresponds to the \textbf{Coulomb repulsive interaction}; the Boltzmann operator is ill-defined in this case, and must be modified, as we discuss below.

Another fundamental observation is that Boltzmann's equation~\eqref{eq:boltz} ``contains'' the compressible hydrodynamic equations on the macroscopic fields (local density, local momentum and local temperatures):
\begin{equation*}
  \rho = \int_{\R^d} f \dd v, \qquad \rho \, u = \int_{\R^d} f\, v \dd v, \qquad \rho \, T = \frac{1}{d} \int_{\R^d} f \, |v-u|^2 \dd v.
\end{equation*}
More precisely: such dynamics can, at least formally, be derived from the Boltzmann equation in certain scaling limits. Viscous incompressible hydrodynamics can also be obtained by a scaling limit. Hence, the Cauchy problem for \eqref{eq:boltz} is expected to be at least difficult as the Cauchy problem for the Euler and Navier-Stokes systems.

Going back to the limit case $\alpha=2$, one sees that the singular integral does not make sense anymore, and the Boltzmann operator must be modified. Instead, one must use the \textbf{Landau equation} (invented by Landau~\cite{Landau1936,LANDAUPAPERS} in 1936)
\begin{equation}
  \label{eq:land}
  \partial_t f + \, v \cdot \nabla_x f = Q_{\mathrm{L}}(f,f),
\end{equation}
with the \textbf{Landau operator} (since it does not act on the variables $t,x$, these variables are omitted from the formula for clarity)
\begin{equation}
  \label{QLC}
    \begin{aligned}
    &Q_{\mathrm{L}}(f,f) (v) = \\
    &\nabla_v \cdot \left( \int_{v_\star \in \R^d} |v-v_\star|^{-1} \, \Pi_{(v -v^\star)^\perp} \Big[ f(v_\star) \, \nabla_v \, f(v) - f(v) \, \nabla_{v_\star} f(v_\star) \Big] \dd v_\star\right).
    \end{aligned}
\end{equation}
On can rewrite $Q_{\mathrm{L}}$ as a non-linear drift-diffusion operator with coefficients that depend in a non-local manner on the solution:
\begin{equation}
  \label{eq:land2}
  Q_{\mathrm{L}}(f,f) = \nabla_v \cdot \left( A[f] \, \nabla_v f + B[f] \, f \right).
\end{equation}

The global Cauchy theory for the Landau--Coulomb equation is open in general, but has recently been solved in the spatially homogeneous case~\cite{guillen2025landauequationdoesblow}, following many previous partial progresses~\cite{MR1646502,MR1737547,MR1737548,MR2502525,fournier2010uniqueness,MR3375485,MR3158719,silvestre2017upper,golse:hal-02145096} (this non-exaustive list is restricted to works dealing with the spatially homogeneous Landau equation). We give two open problems hereby motivating a regularity theory for nonlinear hypoelliptic PDEs.

We denote by \textbf{hydrodynamic fields} the functions
\begin{equation*}
  (t,x) \quad \mapsto \quad \int_{\R^d} f \dd v, \quad 
  \int_{\R^d} v f  \dd v, \quad 
  \int_{\R^d} |v|^2 f \dd v, \quad 
  \int_{\R^d} f \ln f \dd v.
\end{equation*}

\begin{problem}[Conditional regularity for the Boltzmann equation]
  Show that a solution on $[0,T]$ to~\eqref{eq:boltz} with non-cutoff long-range interactions, whose hydrodynamic fields satisfy suitable a priori pointwise bounds on $[0,T]$, is smooth on $(0,T]$. 
\end{problem}

This problem is now solved for $\alpha >3$:
  \begin{itemize}
  \item $L^\infty$ bounds hold on $f$ conditionally to bounds on the hydrodynamic fields by~\cite{silvestre2017upper},
  \item polynomial decay estimates hold on $f$ conditionally to bounds on the hydrodynamic fields by~\cite{imbert2020decay},
  \item finally weak Harnack inequality, H\"older regularity, Schauder estimates, and smoothness are proved, conditionally to bounds on the hydrodynamic fields, in the series of papers~\cite{imbert2019weak,imbert2021schauderii,imbert2022global}.
  \end{itemize}
  For $\alpha \in (2,3)$ the problem is open.

\begin{problem}[Conditional regularity for the Landau equation] 
  Show that a solution on $[0,T]$ to~\eqref{eq:land}, which satisfies suitable a priori pointwise bounds on the hydrodynamic fields, is smooth on $(0,T]$.
\end{problem}

This second problem in roughly the limit case $\alpha=2$ of the first problem. It is still open, but a partial answer has been obtained for less singular, mollified versions, of the Landau operator. Consider for $\gamma \ge -3$, a modification of the operator~\eqref{QLC} where the singular term $|v-v_\star|^{-1}$ in the integrand is replaced by $|v-v_\star|^{\gamma+2}$. The original case of Coulombian interactions corresponds to $\gamma = -3$. For $\gamma \geq -2$, we have the following results:
\begin{itemize}
    \item $L^\infty$ bounds hold on $f$ conditionally to bounds on the hydrodynamic fields by~\cite{silvestre2017upper},
    \item Harnack inequality and Hölder regularity hold on $f$ conditionally to bounds on the hydrodynamic fields by~\cite{golse2019harnack}, 
    \item Smoothness is proved conditionally to bounds on the hydrodynamic fields by~\cite{henderson2020c,imbert2021schauder}.
\end{itemize}
For $\gamma \in [-3,2)$ the problem is open. 

To finish this brief overview, let us add that many other problems in mathematical physics would certainly benefit from a robust and quantitative theory combining the approaches of De Giorgi--Nash--Moser and Hörmander. We will present some steps towards such theory, but much remains to be explored; in particular the boundary regularity and non-divergence equations. 
\medskip

\subsection{Notation} We denote by $n$ or $d$ the dimension. We denote by $\lesssim$ and $\gtrsim$ inequalities with a constant independent of the parameters at hand, or $\lesssim_k$ for an inequality with a constant depending on a parameter $k$. Given $\cU \subset \R^n$ open set, we write $\tilde \cU \subset \subset \cU$ when $\tilde \cU$ is an open subset of $\cU$ so that $\tilde \cU$ has compact closure included in $\cU$. We finally denote $f_+ := \max(f,0)$ for a scalar function $f$, and $B_r:=B(0,r)$ the ball of $\R^n$ with radius $r>0$.

\subsection{Acknowledgements}

These lecture notes are an expanded version of the material that the second author taught in various contexts: at Université Paris-Dauphine when he was a ``PSL Visiting fellow'' in 2023-2024, at the summer school ``Festum Pi'' in Crete in 2024, and at IHES in March 2025 as a ``Cours de l'IHES''. The support of these institutions, as well as the feedback from students and colleagues who attended these lectures are gratefully acknowledged. In particular, the first author typed a partial first draft of these notes based on the lectures at University Paris-Dauphine. Helpful extensive comments on a preliminary version of these notes were provided by Amélie Loher, Lukas Niebel, and an anonymous referee. The second author also greatly benefited from the discussions over the years on the De Giorgi theory with Francesca Anceschi, Helge Dietert, François Golse, Jessica Guerand, Cyril Imbert, Amélie Loher, Lukas Niebel, Sergio Polidoro, Annalaura Rebucci, Luis Silvestre, Alexis Vasseur and Rico Zacher. Several parts of the material presented in these lecture notes originate from collaborations with these colleagues. The first author acknowledges the support of a Maria Skłodowska Curie Postdoctoral Fellowship.


\section{The elliptic case}

This section is partly inspired by the lecture notes~\cite{vasseur2016giorgi,zbMATH07604914}. This is a concise self-contained account of the regularity theory of De Giorgi in~\cite{de1957sulla}, together with more recent quantitative alternative arguments for the control of oscillations. We also include brief presentations of the alternative approaches of Moser and Kruzhkov~\cite{moser1960new,kruzhkov1963priori}, highlighting the similarities and differences with that of De Giorgi. 

\subsection{The result of De Giorgi}

Let us go back to Hilbert's $19^{th}$ problem. We consider $\cU$ an open subset of $\mathbb R^n$, a Lagrangian $L : \R^n \to \R$ and the minimisation problem~\eqref{eq:pbvar}. When $L(p)=|p|^2/2$, then $\E$ 
is the standard Dirichlet form and the Euler-Lagrange equation of the minimisation problem is simply the Laplace equation. Weak solutions to the Laplace equation are harmonic in $\cU$, from which one deduces the  smoothness and real analyticity in $\cU$ (interior regularity). Hilbert's question asked whether this interior regularity hold for more general Lagrangians $\mathrm{L}$ that satisfy~\eqref{eq:H1}-\eqref{eq:H2}. The original motivation of De Giorgi was to answer this question.

\begin{lemma}
  The following two steps reduce Hilbert's $19$th problem to the regularity theory of elliptic PDEs with rough coefficients:
  \begin{enumerate}[(i)]
  \item Let $\mathrm{L} \in C^2(\R^n)$ so that $\nabla^2 \mathrm{L} \le \Lambda$ on $\R^n$ for some $\Lambda >0$. Then all minimisers of~\eqref{eq:pbvar} in $\mathrm{H}^1(\cU)$ are also weak solutions to the Euler--Lagrange equation~\eqref{eq:EL1}.
  \item Let $\mathrm{L} \in C^2(\R^n)$ so that $\Lambda^{-1} \le \nabla^2 \mathrm{L} \le \Lambda$ on $\R^n$ for some $\Lambda >0$. Let $u \in \mathrm{H}^1(\cU)$ be a weak solution to~\eqref{eq:EL1}, $\tilde{\cU} \, \subset \subset \cU$, and $i=1,\cdots,n$. Then $u \in H^2(\tilde{\cU})$ and $f:=\partial_{x_i} u \in H^1(\tilde \cU)$ is a weak solution to~\eqref{eq:EL2} on $\tilde \cU$.
\end{enumerate}

\end{lemma}

\begin{proof}
  (i) Given $\var>0$ and $w \in C^\infty_c(\cU),$ we have
  \begin{equation*}
    \int_\cU \mathrm{L}(\nabla_x u \pm \var \nabla_x w) \dd x \geq \int_{\cU} \mathrm{L}(\nabla_x u) \dd x.
  \end{equation*}
  Note that integrands are integrable because $\nabla_x u$, $\nabla_x w \in L^2(\cU)$ and $\mathrm{L}$ is sub-quadratic. Then, for all $x \in \cU$, we have by Taylor expansion
  \begin{equation*}
    \mathrm{L}(\nabla_x u(x) \pm \var \, \nabla_x w(x)) \le \mathrm L(\nabla_x u(x)) \pm \var \nabla \mathrm{L}(\nabla_x u(x)) \cdot \nabla_x w(x) + \Lambda \, \frac{\var^2}{2} \, |\nabla_x w(x)|^2.
  \end{equation*}
  The last two equations yield
  \begin{equation*}
    \pm \int_{\cU} \nabla \mathrm{L}(\nabla_x u) \cdot \nabla_x w \dd x \geq - \Lambda \, \frac{\var}{2} \, \int_{\cU}  |\nabla_x w|^2 \dd x.
  \end{equation*}
  By letting $\var \to 0$, we deduce
  \begin{equation*}
    \forall \,  w \in C^\infty_c(\cU), \quad  \int_{\cU} \nabla \mathrm{L}(\nabla_x u) \cdot \nabla_x w \dd x =0,
  \end{equation*}
  which is the weak form of~\eqref{eq:EL1}.
  \medskip

  \noindent
  (ii) By symmetry, a covering argument, translation and dilation, it is enough to prove the claim when $\cU=B_3$, $\tilde \cU=B_1$, $i=1$.
  For $-1 \le h \le 1,$ consider 
  \begin{equation*}
    \tau_hu(x) := \frac{u(x+h\, e_1)-u(x)}{h},
  \end{equation*}
  and $\eta$ a smooth function such that
  \begin{equation*}
    \eta \equiv 1 \text{ on } B_1, \quad \eta \equiv 0 \text{ on } B_2^c, \quad \text{ and } \ |\nabla_x \eta| \le 2.
  \end{equation*}
  The operators $\eta \, \tau_h,$ and $\tau_h \,( \eta \, \cdot )$ map $L^2(B_3)$ to itself and map $H^1(B_3)$ to $H^1_0(B_3)$. We then test~\eqref{eq:EL1} against $\phi := \tau_{-h} \left( \eta^2 \tau_h u \right)$. Note that squaring the localisation function $\eta$ is standard in energy estimates and related to the $L^2$ structure. We obtain:
\begin{align*}
    & 0 = \int_{B_3} \nabla_x \phi \cdot\nabla L(\nabla_x u) \dd x = \\
    &- \int_{B_3} \eta \, \nabla_x (\tau_h u  ) \cdot \eta \, \tau_h (\nabla L(\nabla_x u)) \dd x - 2 \int_{B_3} (\nabla_x \eta) \, (\tau_h u)  \,\eta \tau_h (\nabla L(\nabla_x u)) \dd x =\\ & - \int_{B_3} \eta \, \nabla_x (\tau_h u) \, \left( \int_0^1 \nabla^2 L (\nabla_x u + \theta \, ( \nabla_x u(x +h \, e_i) - \nabla_x u(x) )) \, d \theta \right) \, \eta  \nabla_x (\tau_h u) \dd x + \\
    &-2 \int_{B_3} (\nabla_x \eta) (\tau_h u) \left( \int_0^1 \nabla^2 L (\nabla_x u + \theta \, ( \nabla_x u(x +h \, e_i) - \nabla_x u(x) )) \, d \theta \right) \, \eta \, \nabla_x (\tau_h u) \dd x.
\end{align*}
Using $\Lambda^{-1} \le \nabla^2 L \le \Lambda$, we deduce
\begin{equation*}
  \frac{1}{\Lambda} \, \int_{B_3} \eta^2 |\nabla_x \tau_h u|^2 \dd x \le 2 \Lambda \left( \int_{B_3} |\nabla_x \eta|^2 \, |\tau_h u|^2 \dd x \right)^{1/2} \, \left(\int_{B_3} \eta^2 \, |\nabla_x \tau_h u|^2 \dd x  \right)^{1/2},
\end{equation*}
which shows
\begin{equation*}
  \int_{B_3} \eta^2 |\nabla_x \tau_h u|^2 \dd x \lesssim \, \Lambda^4 \int_{B_3} |\nabla_x \eta|^2 \, |\tau_h u|^2 \dd x \le \Lambda^4 \, \|u\|^2_{H^1(B_3)}.
\end{equation*}
By letting $h \to 0$, we deduce $f = \partial_i u \in H^1(B_1)$ and
\begin{equation*}
  \|f\|_{H^1(B_1)} \lesssim \Lambda^2 \|u\|_{H^1(B_3)}
\end{equation*}
(see~\cite{MR2597943} for results on generalised derivatives). Then given $\phi \in H^1_0(B_1)$, we have $\tau_{-h} \phi \in H^1_0(B_2)$ and 
\begin{equation*}
  0 = - \int_{B_2} \tau_h \phi \, \nabla_x \cdot (\nabla L(\nabla_x u)) \dd x = \int_{B_2} \nabla_x \phi \cdot \tau_h (\nabla L(\nabla_x u)) \dd x.
\end{equation*}
Finally, by letting $h \to 0$, we deduce 
\begin{equation*}
  0 = \int_{B_1} \nabla_x \phi \cdot \partial_{x_1} \nabla L(\nabla_x u) \dd x =  \int_{B_1} \nabla_x \phi \cdot  \nabla^2 L (\nabla_x u) \nabla_x f \dd x =  \int_{B_1} \nabla_x \phi \cdot A \nabla_x f \dd x,
\end{equation*}
which is the weak form of~\eqref{eq:EL4} on $B_1$ for $f$.
\end{proof}

As we have discussed above, the previous works of Bernstein, Schauder and others have showed that if $u \in C^{1+\alpha}_{\mathrm{loc}}(\cU)$ solves~\eqref{eq:EL1}, then $u$ is smooth and real analytic in $\cU$. The solution to Hilbert's $19$th problem is therefore reduced to proving $\nabla_x u \in C^\alpha_{\mathrm{loc}}(\cU)$ for some $\alpha \in (0,1)$, i.e. that any weak $H^1$ solution $f$ to~\eqref{eq:EL2} is Hölder continuous.

Let us recall the definition of Hölder spaces. For any open set $\cU \subset \mathbb R^n$, we define the \textbf{Hölder norm} $C^\alpha(\cU)$ as follows:
\begin{equation}
  \label{eq:holder}
  \|f\|_{C^\alpha(\cU)} := \|f\|_{L^\infty(\cU)} + [f]_{C^\alpha(\cU)}, \qquad [f]_{C^\alpha(\cU)} := \sup_{x,y \in \cU} \, \frac{|f(x)-f(y)|}{|x-y|^\alpha},
\end{equation}
where $[f]_{C^\alpha(\cU)}$ is the \textbf{Hölder seminorm} and $\alpha \in (0,1)$. The standard notation $C^\alpha_{\mathrm{loc}}(\cU)$ denotes the set of functions that are in $C^\alpha_{\mathrm{loc}}(\tilde \cU)$ for every $\tilde \cU \subset \subset \cU$. 

\begin{theorem}[Elliptic De Giorgi--Nash--Moser]
  \label{thm:dgnm2}
  Given $\Lambda>0$, there is $\alpha \in (0,1)$, depending only $\Lambda>0$ and the dimension $n$, and there is $C>0$, depending only on $\Lambda$ and the dimension $n$, so that the following holds:
  \smallskip

  Let $f \in H^1(\cU)$ be a weak solution on the open set $\cU \subset \R^n$ to the equation
  \begin{equation}
    \label{eq:EL4}
    \nabla_x \cdot ( A \, \nabla_x f ) =0,
  \end{equation}
  with $A : \cU \to \mathbb R^{2n}$ a symmetric matrix with measurable entries so that
  \begin{equation*}
     \quad \Lambda^{-1} \le A(x) \le \Lambda \quad \text{for almost every } x \in \cU.
  \end{equation*}

  Then, we have $f \in C^\alpha_{\mathrm{loc}}(\cU)$.
  \smallskip

  Moreover, given $\delta>0$, there is $C>0$, depending only on $\Lambda$, $\delta$ and the dimension $n$, so that for any $\tilde{\cU} \subset \subset \cU$ so that  $\mathrm{dist}(\tilde{\cU},\cU) \ge \delta$, we have
  \begin{equation*}
    \|f\|_{C^\alpha(\tilde{\cU})} \le C \|f\|_{\mathrm{L}^2(\cU)}.
  \end{equation*}
\end{theorem}

\begin{remarks}
  \begin{enumerate}
  \item Notice that the strict convexity
    $\nabla^2 \mathrm{L} \geq \Lambda^{-1}$ is necessary. In dimension
    $n=1$, let $\mathrm L$ be a double-well shaped function
    \begin{equation*}
      \mathrm{L}(p) = |p|^2 - 2 |p|.
    \end{equation*}
    Then, one can check that $u(x)=|x|$ is
    a minimiser of~\eqref{eq:pbvar} in $H^1((-2,2))$, but $u'$ is not Hölder
    continuous.
  \item Nash proved a similar result for parabolic equations with rough coefficients, which includes this theorem as a particular case. We discuss the parabolic case in the next section. Moser and Kruzhkov later also recovered this theorem by other proofs.
  \item We emphasize that this is \textbf{not} a perturbation result around the case of constant coefficients: the equation may not be locally close to an elliptic PDE with constant coefficients, whichever the scale one zooms in. The entries of $A$ being merely measurable, they may oscillate fast at \textbf{every} scale.
  \end{enumerate}
\end{remarks}

\subsection{The first De Giorgi Lemma}

The argument of De Giorgi is divided into two parts, corresponding roughly first to the control of the $L^\infty$ part of the Hölder norm at every scale (first De Giorgi Lemma) and second to the control of the oscillation and thus the Hölder semi-norm (second De Giorgi Lemma). The first lemma plays a key role in the proof of the second lemma.

\begin{lemma}[First De Giorgi Lemma: gain of integrability]
  \label{lem:dg1}
  Given $\Lambda>0$, there is $\delta >0$, depending only on $\Lambda$ and the dimension $n$, so that the following holds:
  \smallskip
  
  Let $f \in H^1(B_1)$ be a weak subsolution to
  \begin{equation*}
    \nabla_x \cdot ( A \nabla_x f ) \ge 0
  \end{equation*}
  with $A$ symmetric measurable so that
  \begin{equation*}
    \forall \, x \in \cU, \quad \Lambda^{-1} \le A(x) \le \Lambda.
  \end{equation*}

  Then, we have
  \begin{equation*}
    \|f_+\|_{\mathrm{L}^2(B_1)} \le \delta \quad \implies \quad  \|f_+\|_{\mathrm{L}^\infty(B_{1/2})} \le \frac{1}{2}.
  \end{equation*}
\end{lemma}

\begin{remarks}
  \begin{enumerate}
  \item This lemma is scale-invariant provided that the integrals are normalised: 
    \begin{equation*}
      \forall \, r \in (0,1], \qquad \left(\int_{B_r} f_+^2 \, \frac{dx}{|B_r|}\right)^{1/2} \le \delta \, \implies \|f_+\|_{\mathrm{L}^\infty(B_{r/2})} \le \frac{1}{2}.
    \end{equation*} 
  \item By covering, translation and dilation, this result implies that, when the Lagrangian $\mathrm{L}$ satisfies~\eqref{eq:H1}-\eqref{eq:H2}, any minimiser $u$ of~\eqref{eq:pbvar} in $H^1(\cU)$ is in $W^{1,\infty}_{\mathrm{loc}}(\cU)$.
  \item Lemma \ref{lem:dg1} implies that $f \in H^1(\cU)$ weak solution to~\eqref{eq:EL4} is $\mathrm{L}^\infty_{\mathrm{loc}}(\cU)$. As soon as $n \ge 1$, this is not implied by Sobolev's embeddings. Let us also emphasize that even in situations where the $L^\infty$ bound on $f$ can be obtained by maximum principle arguments, this lemma proves a stronger result because of the the scale-invariant estimates. 

  \item The super-exponential convergence of the De Giorgi--Moser iteration below is reminiscent of Newton's iterative method, used in the KAM Theorem~\cite{MR68687,MR163025,MR147741} and in the so-called Nash--Moser implicit function theorem~\cite{MR75639,MR199523,MR206461}.
  \end{enumerate}
\end{remarks}

\begin{proof}[Proof of Lemma \ref{lem:dg1} (De Giorgi's version)]
  The proof combines an iterative scheme, a localised energy estimate, the use of the Sobolev embedding and the Chebyshev inequality.
  \medskip
  
  \noindent
  \textbf{A. The iterative scheme.} For all $k\geq0$, we consider balls with radii
  \begin{equation*}
    r_k = \frac12 + \frac{1}{2^{k+1}} \ \text{ so that } \ r_0=1 \text{ and } r_k \downarrow \frac12
  \end{equation*}
  (see Figure~\ref{fig:iteration}), and energy levels
  \begin{equation*}
    e_k := \frac12 - 2^{-k-1} \ \text{ so that } e_0=0 \text{ and } e_k \uparrow \frac12.
  \end{equation*}

  \begin{figure}
    \includegraphics[scale=0.8]{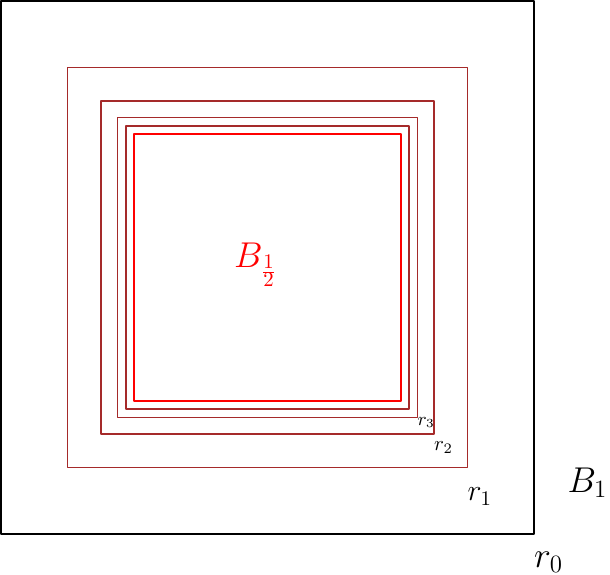}
    \caption{The sequence of balls in the De Giorgi--Moser iteration.}
    \label{fig:iteration}
  \end{figure}
  The goal is to establish estimates by induction on 
  \begin{equation*}
    \forall \,  k \geq 0, \qquad E_k := \int_{B_{r_k}} (f-e_k)_+^2 \dd x.
  \end{equation*}
  Indeed $E_0 = \| f \|_{L^2(B_1)}^2$, and if we prove that there are $C>0$ and $\beta>1$ such that 
  \begin{equation}
    \label{eq:ee1}
    \forall \,  k\geq 1, \quad E_k \le C^k \, E_{k-1}^\beta
  \end{equation}
  then we deduce by induction 
  \begin{equation*}
    E_k \le C^{k + \beta (k-1) + \beta^2 (k-2) + \cdots + \beta^{k-1}} \, E_0^{\beta^k} \le \left(C' \, E_0 \right)^{\beta^k}.
  \end{equation*}
  where we have used 
  \begin{equation*}
    k + \beta (k-1) + \cdots + \beta^{k-1} = \frac{k+1}{1-\beta} + \beta^k \, \frac{\beta}{(1-\beta)^2} \lesssim \beta^k.
  \end{equation*}
  If $E_0 \, C'<1$, which is ensured by $\|f_+\|_{\mathrm{L}^2(B_1)} \le \delta$ with $\delta< 1/C'$, then
  \begin{equation*}
    \int_{B_{\frac12}} \( f - \frac12 \)_+ ^2 \dd x = E_\infty = \lim_{k \to \infty} \, E_k =0,
  \end{equation*}
  which concludes the proof.

  The rest of the proof is devoted to establishing~\eqref{eq:ee1}. This inequality is nonlinear, due to the hidden collision of scalings in the equation which relates the $L^2$ norms of $f$ and of the gradient of $f$: this collision of scalings is revealed by the energy estimate, Sobolev's inequality and Chebyshev's inequality. 
  \medskip
  
  \noindent
  \textbf{B. Localised energy estimate.} The following inequality is a localised version of the by-now standard $L^2 \to H^1$ ellipticity estimate, which still holds for divergence-form elliptic equations with rough coefficients:
  \begin{equation}
    \label{eq:ee2}
    \forall \,  \phi \in C^1_c(B_1), \quad \int_{B_1} |\nabla_x (\phi f_+)|^2 \dd x \le \Lambda^2 \, \|\nabla_x \phi\|^2_{\mathrm{L}^\infty(B_1)} \, \int_{B_1 \cap \mathrm{supp} \,\phi} f_+^2 \dd x.
  \end{equation}
  It is also called \textbf{Caccioppoli inequality}, in memory of~\cite{MR46536} where Caccioppoli introduced such inequalities in his study of elliptic equations. The motivation of Caccioppoli was to provide an $L^2$ analogue of the Cauchy inequalities for holomorphic functions, and mean value type inequalities for harmonic functions. Holomorphic functions indeed satisfy the Cauchy-Riemann equations, an elliptic system, and the Cauchy integral formula implies a pointwise control on the first derivative of a function on a small ball by the values of the function itself on a larger ball, with a constant that involves the difference in the radii of the two balls. Similar estimates can be derived for harmonic functions thanks to the mean value property. Caccioppoli worked on Hilbert's $19$-th problem in the 1930s\footnote{See the biographical note at \href{https://mathshistory.st-andrews.ac.uk/Biographies/Caccioppoli/}{https://mathshistory.st-andrews.ac.uk/Biographies/Caccioppoli/}.}: he extended the method of Bernstein beyond the two-dimensional case, and reduced the assumption from $C^3$ to $C^2$ regularity on the minimiser to~\eqref{eq:pbvar}. 
  
  To prove it, we integrate $\nabla_x \cdot (A \nabla_x f) \ge 0$ against $\phi^2 \, f_+$ to get 
  \begin{align*}
    0 \le & \int_{B_1} \nabla_x \cdot (A \nabla_x f) \, \phi^2 \, f_+ \dd x = - \int_{B_1} (A \nabla_x f_+) \cdot \nabla_x (\phi^2 f_+) \dd x \\
    =  & - \int_{B_1} (A \nabla_x f_+ ) \cdot \nabla_x (\phi f_+) \, \phi \dd x - \int_{B_1} (A \nabla_x f_+ \cdot \nabla_x \phi) \, (f_+ \, \phi) \dd x  \\
    =  & - \int_{B_1} A \nabla_x(\phi f_+) \cdot \nabla_x (\phi f_+) \dd x \\
    & \qquad - \int_{B_1} A \nabla_x \phi \cdot \nabla_x (f_+ \phi) \, f_+ \dd x - \int_{B_1} (A \nabla_x f_+ \cdot \nabla_x \phi) (f_+ \phi) \dd x \\
    =  &- \int_{B_1} A \nabla_x (\phi f_+) \cdot \nabla_x (\phi f_+) \dd x   +  \int_{B_1} (A \nabla_x \phi \cdot \nabla_x \phi) f_+^2\dd x,
  \end{align*}
  (the integration by parts have no boundary terms since $\phi$ being zero on $\partial B_1$), which implies
  \begin{equation*}
    \int_{B_1} A \nabla_x (\phi f_+) \cdot \nabla_x (\phi f_+) \dd x \le \int_{B_1} (A \nabla_x \phi \cdot \nabla_x \phi) f_+^2\dd x.
  \end{equation*}
  We deduce~\eqref{eq:ee2} thanks to the inequality $\Lambda^{-1} \le A \le \Lambda$ almost everywhere on the matrix $A$.
  \medskip

  \noindent \textbf{C. Sobolev's embedding.} Let $(\phi_k)_k \subset C^\infty_c(B_1)$ be a family of cutoff functions such that
  \begin{equation*}
    \forall \,  k\geq1, \quad \mathbf 1_{B_{r_k}} \le \phi_k \le \mathbf 1_{B_{r_{k+1}}}, \qquad |\nabla_x \phi_k| \le c_0 2^k,\end{equation*} 
for some $c_0$ which does not depend on $k$. We bound from above
\begin{equation*}
  E_k = \int_{B_{r_k}} (f-e_k)_+^2 \dd x \le \int_{B_1} \phi_k^2 \, (f-e_k)_+^2 \dd x.
\end{equation*}
For all $k\geq 1$, the function $(f-e_k)$ also satisfies~\eqref{eq:EL4} by linearity and because constants are solutions. We then apply~\eqref{eq:ee2} with $\phi:=\phi_k$ to get
\begin{equation}
    \label{eq:ee3}
    \int_{B_1} |\nabla_x (\phi_k \, (f-e_k)_+) |^2 \dd x \le \Lambda^2 \, \|\nabla_x \phi_k\|^2_{\mathrm{L}^\infty(B_1)} \, \int_{B_1 \cap \mathrm{supp} \, \phi_k} (f-e_k)_+^2 \dd x.
\end{equation}
Let us define the following integrability exponent
\begin{equation*}
  p_* := \frac{2d}{d-2} \text{ for } d \geq 3, \ \text{ or any } \ p_* \in (2,\infty) \text{ for } d=1,2.
\end{equation*}
Then, Sobolev's inequality~\cite{MR2597943} applied to the function $\phi_k \, (f-e_k)_+$ over $ B_1$ combined with the localised energy estimate~\eqref{eq:ee3} yields
\begin{align*}
  \|\phi_k \, (f-e_k)_+ \|^2_{L^{p_*}(B_1)}
  & \le C_{\mathrm{Sob}}^2 \, \int_{B_1} |\nabla_x (\phi_k (f)e_k)_+)|^2 \dd x \\ & \le C_{\mathrm{Sob}}^2 \, \Lambda^2 \, c_0^2 \, 2^{2k} \int_{B_1 \cap \mathrm{supp}\,\phi} (f-e_k)_+^2 \dd x.
\end{align*}
On $\mathrm{supp} \, \phi_k \subset B_{r_{k-1}}$, we have $\phi_{k-1} \equiv 1$ so 
\begin{align*}
  (f-e_k)_+ \le \, (f-e_{k-1})_+ \, \phi_{k-1}
\end{align*}
and we deduce
\begin{align*}
  \|\phi_k \, (f-e_k)_+ \|^2_{L^{p_*}(B_1)}
  & \le C_{\mathrm{Sob}}^2 \, \Lambda^2 \, c_0^2 \, 2^{2k} \int_{B_1 \cap \mathrm{supp}\,\phi} (f-e_k)_+^2 \dd x \\
  & \le C_{\mathrm{Sob}}^2 \, \Lambda^2 \, c_0^2 \, 2^{2k} E_{k-1}.
\end{align*}
\smallskip

\noindent
\textbf{D. Chebyshev's inequality.} We apply H\"older's inequality with $p_*$ and its conjugate exponent. (To fix ideas, we write things for $d \ge 3$ and $p_*=2d/(d-2)$, the case of $d=1$ or $2$ is an easy variation.) 
\begin{equation}
    \label{eq:ee4}
    \forall \,  k\geq 1, \quad E_k = \int_{B_1} (\phi_k \, (f-e_k)_+)^2 \dd x \le \|\phi_k (f-e_k)_+\|^2_{L^{p_*}(B_1)} \, |\left\{ \phi_k (f-e_k)_+ > 0\right\}|^{\frac2n}.
\end{equation}
Equation \eqref{eq:ee4} is the key-step of the proof: together with Chebychev's inequality, it will allow us to exploit the nonlinearity hidden in the relation between $f$ and $\nabla_x f$ (Caccioppoli's inequality) and Sobolev's inequality.

We observe that
\begin{equation*}
  \left\{ \phi_k (f-e_k)_+ > 0\right\} \subset \left\{ \phi_{k-1} \, (f-e_{k-1})_+ \geq 2^{-k-1}\right\}
\end{equation*}
since $\phi_{k-1} \equiv 1$ on the support of $\phi_k$, and  
\begin{equation*}
  f \geq e_k = 1 - 2^{-k-1} \ \Longrightarrow \ f - e_{k-1} = f - \left(\frac{1}{2} - 2^{-k}\right) \geq 2^{-k} - 2^{-k-1} = 2^{-k-1}.
\end{equation*}
Therefore, by Chebyshev's inequality:
\begin{align*}
  |\left\{ \phi_k (f-e_k)_+ > 0\right\}|
  & \le \left|\left\{ \phi_{k-1} \, (f-e_{k-1})_+ \geq 2^{-k-1}\right\}\right| \\
  & \le 2^{2(k+1)} \int_{B_1} \left(\phi_{k-1} (f-e_{k-1})_+\right)^2 \dd x \\
  & =  2^{2(k+1)} \, E_{k-1}.
\end{align*}
The last display, combined with \eqref{eq:ee4}, gives 
\begin{equation*}E_k \le \, C_{\mathrm{Sob}}^2 \, c_0^2 \,\Lambda^2 \, 2^{2k + 2(k+1) \, \frac{2}{d} } \, E_{k-1}^{1+\frac{2}{d}},\end{equation*}
which proves \eqref{eq:ee1} with $\beta \geq 1+\frac{2}{d} > 1$.
\end{proof}

\begin{proof}[Proof of Lemma \ref{lem:dg1} (Moser's version)]
  The proof by Moser~\cite{moser1960new} of the first De Giorgi lemma is also based on an iterative scheme. The proofs of De Giorgi and Moser are conceptually close, so that they are often collectively close the ``De Giorgi--Moser iteration''. The main idea of Moser is to fully exploit the fact that the energy estimate and the Sobolev embedding both continue to yield the same inequality for weak subsolutions, not just for weak solutions, i.e. for functions $f$ so that
  \begin{equation}
    \label{eq:subeq-ell}
    \nabla_x \cdot (A \nabla_x f) \ge 0.
  \end{equation}
  
  Moser's iteration uses the same sequence of radii, but instead of considering energy levels to measure integrability, it considers a \textit{sequence of Lebesgue exponents}
  \begin{equation*}
    p_k := 2 \( \frac{p_*}{2} \)^k
  \end{equation*}
  with
  \begin{equation*}
    p_*:= 2 + \frac{4}{d-2} \text{ for } d \ge 3 \ \text{ or } \ p_* \in (2,+\infty) \text{ in } d=1,2.
  \end{equation*}
  So we have $p_0=2$ and $p_k \uparrow \infty$. Instead of $E_k$, the proof establishes estimates by induction on the following Lebesgue norms:
  \begin{equation*}
    \forall \, k \ge 0, \quad F_k := \| f_+\|_{\mathrm{L}^{p_k}(B_{r_k})}.
  \end{equation*}
  Note that $F_0 = \|f_+\|_{L^2(B_1)}$ and $F_\infty = \|f_+\|_{L^\infty(B_{1/2})}$. 

  The key remark is that if $f$ is a subsolution~\eqref{eq:subeq-ell}  and $\Phi : \R \to \R$ is $C^2$ with $\Phi', \Phi''\ge 0$, then $\Phi(f)$ is also a subsolution:
  \begin{align*}
    \nabla_x \cdot (A \nabla_x \Phi(f)) = \Phi''(f) A \nabla_x f \cdot \nabla_x f + \Phi'(f) \nabla_x \cdot ( A \nabla_x f) \ge 0.
  \end{align*}
  By approximation it implies in particular that $f_+$ is a subsolution, and then that any $f_+^{p_k/2}$ is also a subsolution, since
  \begin{equation*}
    \Phi(z) := z^{\frac{p_k}{2}} 
  \end{equation*}
  satisfies the assumptions above for $p_k/2 >1$. Applying \eqref{eq:ee2}---which is valid for subsolutions---to $f_+^{p_k/2}$, we find
  \begin{equation*}
    \forall \,  k \geq 1, \quad  \left\|\nabla_x \(\phi_k \, f_+^{\frac{p_k}2}\)\right\|_{\mathrm{L}^2(B_1)} \lesssim \Lambda^2 2^k \, F_k^{\frac{p_k}2},
  \end{equation*}
  where $(\phi_k)_k$ is the same cutoff family as in the proof of De Giorgi. Sobolev's inequality then yields
  \begin{equation*}
    \|f_+\|^{\frac{p_k}2}_{\mathrm{L}^{p_k \, \left( 1 + \frac{2}{d-2} \right)}(B_{r_{k+1}})} \lesssim \Lambda^2 2^k \,  F_k^{\frac{p_k}2}.
  \end{equation*}
  This shows
  \begin{equation*}
    F_{k+1} \le \Lambda^2 C^{\frac{2}{p_k}} \, 2^{\frac{2}{p_k}} \, F_k,
  \end{equation*}
  for an explicit $C= C(n,\Lambda)$. Since the Cauchy product $\prod_{k \ge 0} C^{\frac{2}{p_k}} 2^{\frac{2}{p_k}}$ is converging, we deduce $F_\infty \lesssim \Lambda^2 F_0$, which concludes the proof. Note that this approach makes it clear that $\delta \sim \Lambda^{-2}$ in Lemma~\ref{lem:dg1}.
\end{proof}

\subsection{The improved first De Giorgi lemma}

In fact, the first De Giorgi lemma can be improved so that the initial integrability assumption can be a positive power as low as wanted. This only requires one more iteration argument.

\begin{lemma}[Improved first De Giorgi Lemma]
  \label{lem:dg1+}
  Given $\Lambda>0$ and $\zeta_0 \in (0,2]$, there is $C>0$, depending only on $\Lambda$, $\zeta_0$ and the dimension $n$, so that the following holds:
  \smallskip
  
  Let $f$ be a weak subsolution in $H^1(B_R)$ to
  \begin{equation*}
    \nabla_x \cdot ( A \nabla_x f ) \ge 0
  \end{equation*}
  with $A$ symmetric measurable so that
  \begin{equation*}
    \Lambda^{-1} \le A(x) \le \Lambda \ \text{ for almost every } x \in B_R.
  \end{equation*}

  Then, for any $\zeta \in (\zeta_0,2]$ and $0<r<R$, we have
  \begin{equation*}
    \|f\|_{L^\infty(B_r)} \le C (R-r)^{-\frac{n}{\zeta}} \left( \frac{1}{|B_R|} \int_{B_R} f^\zeta \dd x \right)^{\frac{1}{\zeta}}.
  \end{equation*}
\end{lemma}

\begin{proof}
  By using the scaling of the equation in the first De Giorgi lemma, one can prove the following inequality between two general balls $B_r$ and $B_R$:
  \begin{equation*}
     \|f\|_{L^\infty(B_r)} \lesssim (R-r)^{-\frac{n}{2}} \left( \frac{1}{|B_R|} \int_{B_R} f^2 \dd x \right)^{\frac{1}{2}}.
  \end{equation*}
  Together, the  Young inequality and the inequality
  \begin{equation*}
    \left( \frac{1}{|B_R|} \int_{B_R} f^2 \dd x \right)^{\frac{1}{2}} \lesssim \left( \frac{1}{|B_R|} \int_{B_R} f^\zeta \dd x \right)^{\frac12} \|f\|^{1-\frac{\zeta}{2}}_{L^\infty(Q_R)},
  \end{equation*}
  imply that the function $\mathcal N(r):=\|f\|_{L^\infty(Q_r)}$ verifies the following inequality for $0<r<r'<R$ and some $C>0$:
  \begin{equation*}
    \mathcal N(r) \le \frac{1}{2} \mathcal N(r') + C  (r'-r)^{-\frac{n}{\zeta}} \left( \frac{1}{|B_R|} \int_{B_R} f^\zeta \dd x \right)^{\frac1\zeta}.
  \end{equation*}
  
  We define the increasing sequence of radii:
  \begin{equation*}
    r_{k+1} := r_k + \var \(1-\alpha^{-1}\) \alpha^{-k}
  \end{equation*}
  for $k \ge 0$, some $\alpha>1$, $\var := (R-r)$ and $r_0:=r$. The sequence converges to $r_\infty=R$. Note that the scaling ratio $\alpha>1$ between the successive balls is not $2$ as in the proof of Lemma~\ref{lem:dg1}; it will be chosen later so as to make the geometric growth slow enough.

  Then, the previous display implies
  \begin{equation*}
    \mathcal N(r_k) \le \frac{1}{2} \mathcal N(r_{k+1}) + C \var^{-\frac{n}{\zeta}} (\alpha-1)^{-\frac{n}{\zeta}} \alpha^{\frac{kn}{\zeta}} \left( \frac{1}{|B_R|} \int_{B_R} f^\zeta \dd x \right)^{\frac1\zeta},
  \end{equation*} 
  so that for any $\ell \ge 1$,
  \begin{equation*}
    \mathcal N(r) \le \frac{\mathcal N(r_{\ell+1})}{2^\ell} + C  \var^{-\frac{n}{\zeta}} (\alpha-1)^{-\frac{n}{\zeta}} \left( \sum_{k=1}^\ell 2^{-k} \alpha^{\frac{kn}{\zeta}} \right) \left( \frac{1}{|B_R|} \int_{B_R} f^\zeta \dd x \right)^{\frac1\zeta}.
  \end{equation*}

  The series in the right-hand side converges for the choice $\alpha := (3/2)^{\zeta/n}$:
  \begin{equation*}
    \sum_{k=1}^\ell 2^{-k} \alpha^{\frac{kn}{\zeta}} \le \frac{2}{2-\alpha^{n/\zeta}} =4, 
  \end{equation*}
  and we finally deduce by letting $\ell \to \infty$:
  \begin{equation*}
    \mathcal N(r) \lesssim \var^{-\frac{n}{\zeta}} (\alpha-1)^{-\frac{n}{\zeta}} \left( \frac{1}{|B_R|} \int_{B_R} f^\zeta \dd x \right)^{\frac1\zeta},
  \end{equation*}
  from which the conclusion follows (using that $\zeta \ge \zeta_0>0$).
\end{proof}

\begin{remark}
  On can track the dependency in $\zeta$ of the constant in this proof: at first order the constant diverges as $\zeta \to 0$ like $\zeta \exp[ - n (\ln \zeta)/\zeta]$.
\end{remark}

\subsection{The second De Giorgi Lemma}

This second lemma concerns regularity, not just integrability. Since the coefficients of equation \eqref{eq:EL4} have no regularity, it is not possible to access the regularity of $f$ by differentiating the equation. The approach of De Giorgi is to study the \textit{pointwise oscillation} of the solution. The oscillation compares locally the supremum of $f$ with its infimum. De Giorgi proves that such oscillation decreases geometrically with the scaling when one zooms in on a point.

The overall structure of the proof of the second De Giorgi lemma is as follows:
\begin{equation}
  \label{eq:structure-lemmas}
  \boxed{
  \begin{array}{l}
     \text{Intermediate value Lemma } \Longrightarrow \text{ Decrease of supremum Lemma } \\[2mm]
    \Longrightarrow \text{ Reduction of oscillation Lemma } \Longrightarrow \text{ Local Hölder regularity}.
  \end{array}
  }
\end{equation}

\begin{lemma}[Second De Giorgi Lemma: reduction of oscillation]\label{lem:dg2}
  Given $\Lambda>0$, there is $\nu \in (0,1)$, depending only on $\Lambda$ and the dimension $n$, so that the following holds:
  \smallskip
  
  Let $f$ be a weak solution in $H^1(B_3)$ to
  \begin{equation*}
    \nabla_x \cdot ( A \nabla_x f ) =0
  \end{equation*}
  where $A$ is symmetric measurable and so that
  \begin{equation*}
    \Lambda^{-1} \le A(x) \le \Lambda \ \text{ for almost every } x \in B_3.
  \end{equation*}
  Then, we have 
  \begin{equation}
    \label{eq:dg2}
    \mathrm{osc}_{B_{\frac12}} f \le \nu \, \mathrm{osc}_{B_2} f,
  \end{equation}
  where we denote $\mathrm{osc}_B f := \sup_B f - \inf_B f$.
\end{lemma}

Note that this lemma requires $f$ to be \textbf{both} a subsolution and a supersolution, i.e. to be a solution. Intuitively, being a subsolution and taking values far below the supremum imposes decrease of the supremum when zooming in on a point (see Lemma~\ref{lemma4}), while being a supersolution and taking values far above the infimum imposes increase of the infimum when zooming in on a point (which is Lemma~\ref{lemma4} applied to the opposite of the supersolution). Both are combined to obtain the reduction of oscillation.

Let us show that Lemma~\ref{lem:dg2} implies the H\"older regularity, which thus concludes the proof of Theorem~\ref{thm:dgnm2}.
\begin{corollary}
  Given $\Lambda>0$, there is $\alpha \in (0,1)$, depending only on $\Lambda$ and the dimension $n$, so that the following holds:
  \smallskip
  
  Let $f$ be a weak solution in $H^1(\cU)$ with $\cU \subset \R^n$ open set, to
  \begin{equation*}
    \nabla_x \cdot ( A \nabla_x f ) =0
  \end{equation*}
  where $A$ is symmetry measurable so that
  \begin{equation*}
    \Lambda^{-1} \le A(x) \le \Lambda \ \text{ for almost every } x \in \cU.
  \end{equation*}

  Then, we have $f \in C^\alpha_{\mathrm{loc}}(\cU)$.
  \smallskip
  
  Moreover, given any $\delta>0$,  there is $C>0$, depending on $\Lambda$, $\delta$ and the dimension $n$, so that for any $\tilde{\cU} \subset \subset \cU$ with $\text{dist}(\tilde \cU,\p \cU)\ge \delta$, we have 
  \begin{equation}
    \label{eq:gain-reg-loc}
    \|f\|_{C^\alpha(\tilde{\cU})} \le C \|f\|_{L^2(\cU)}.
  \end{equation}
\end{corollary}

\begin{proof}
  We already have that $f \in \mathrm{L}^\infty_{\mathrm{loc}}(\cU)$ by Lemma~\ref{lem:dg1} with the estimate
  \begin{equation}
    \label{eq:gain-localised}
    \forall \, \hat{\cU} \subset \subset \cU, \quad \|f\|_{L^\infty(\hat{\cU})} \le C \|f\|_{L^2(\cU)}
  \end{equation}
  for a constant $C>0$ depending on $\Lambda$, $\text{dist}(\hat \cU,\p \cU)$ and the dimension $n$.
  
  Let $\mathcal{\tilde{U}} \subset\subset \mathcal{\hat U} \subset \subset \cU$ so that
  \begin{equation}
    \label{eq:lien-dist}
    d_0 := \text{dist}(\mathcal{\tilde{U}},\partial\cU) \le \frac12 \text{dist}(\mathcal{\tilde{U}},\partial\mathcal{\hat U}).
  \end{equation}
  Given $x_0 \in \mathcal{\tilde{U}},$ we define the sequence of centred rescaled functions
  \begin{align*}
    & \forall \,  k \geq 1, \ \forall \,  y \in B_1, \quad \tilde g_k(y) := g_{k-1}\(\frac{y}{4}\) \\
    & \forall \, y \in B_1, \quad g_0(y) := f\(x_0 + \frac{d_0y}{4}\).
  \end{align*}

  Note that each function $g_k$ is a rescaling of $f$ and  solves
  \begin{equation*}
    \nabla_x \cdot [A_k \nabla_x g_k] = 0,
  \end{equation*}
  in $B_2$, with the diffusion matrix
  \begin{equation*}
    A_k(y) := A\(x_0 +  d_0 2^{-2k} y\).
  \end{equation*}
  This new matrix is still measurable and satisfies
  \begin{equation*}
    \Lambda^{-1} \le A_k(x) \le \Lambda \ \text{ for almost every } x \in B_2.
  \end{equation*}

  Therefore, we can apply Lemma~\ref{lem:dg2} to each $g_k$ and deduce
  \begin{equation*}
    \forall \,  k \geq 1, \quad \mathrm{osc}_{B_{\frac12}} g_k \, \le \, \nu \, \mathrm{osc}_{B_2} g_k.
  \end{equation*}
  Combined with the observation
  \begin{equation*}
    \forall \,  k \geq 1, \quad \mathrm{osc}_{B_2} \, g_{k+1} = \mathrm{osc}_{B_{\frac12}} g_k,
  \end{equation*}
  we obtain by induction
  \begin{equation*}
    \mathrm{osc}_{B_{4^{-k}}} g_1 \le \nu^{k-1} \mathrm{osc}_{B_2} g_1 \le 2 \, \nu^{k-1} \, \|f\|_{L^\infty(\hat \cU)}.
  \end{equation*}

  Then, given $x \in B(x_0,1/4)$ we pick up $k \ge 0$ so that
  \begin{equation*}
    |x-x_0| \in \left[4^{-k-2},4^{-k-1}\right]
  \end{equation*}
  and write 
  \begin{equation*}
    |f(x) - f(x_0)|\le 2 \, \nu^{k-1} \, \|f\|_{L^\infty(\hat \cU)},
  \end{equation*}
  so that (using $|x-x_0| \ge 4^{-k-2}$)
  \begin{equation*}
    |f(x) - f(x_0)| \lesssim \(4^\alpha \nu\)^k \, |x-x_0|^\alpha \, \|f\|_{L^\infty(\hat \cU)}.
  \end{equation*}
  We then choose
  \begin{equation*}
    \alpha := - \frac{\ln \nu}{2\, \ln 2}
  \end{equation*}
  so that  $4^\alpha \nu =1$, and deduce $f$ is $C^\alpha$ at $x_0$. The $\alpha$ does \emph{not} depend on $x_0$, but only on the constant $\nu$ of Lemma~\ref{lem:dg2}.

  We have therefore proved
  \begin{equation*}
    \|f\|_{C^\alpha(\tilde{\cU})} \le C \|f\|_{L^\infty(\hat \cU)}
  \end{equation*}
  with a constant depending only on $\nu$ and $d_0$. Combined with~\eqref{eq:gain-localised} and~\eqref{eq:lien-dist}, it implies~\eqref{eq:gain-reg-loc} and concludes the proof.
\end{proof}

We shall show that Lemma~\ref{lem:dg2} is implied by the following lemma.
\begin{lemma}[Decrease of supremum---measure-to-pointwise upper bound]\label{lemma4}
  Given $\Lambda>0$ and $\mu>0$, there is $\lambda \in (0,1)$, depending only on $\Lambda$, $\mu$ and the dimension $n$ such that the following holds:
  \smallskip
  
  Let $f \le 1$ be a weak subsolution in $H^1(B_2)$ to
  \begin{equation*}
    \nabla_x \cdot [A \nabla_x f ] \geq 0,
  \end{equation*}
  where $A$ is symmetric  measurable so that
  \begin{equation*}
    \Lambda^{-1} \le A(x) \le \Lambda \ \text{ for almost every } x \in \cU.
  \end{equation*}
  We assume that $f$ satisfies the following condition in measure:
  \begin{equation*}
    |B_1 \cap \{f \le 0\}| \ge \mu |B_1|.
  \end{equation*}

  Then, we have 
  \begin{equation*}
    \sup_{B_{\frac12}} \, f \le 1-\lambda.
  \end{equation*}
\end{lemma}

\begin{remark}
  The setting is represented in Figure~\ref{fig:m2p}. Note that the infima and suprema are in fact essential infima and suprema. A posteriori however, once the Hölder regularity is established, those would be bounds on a continuous function which is defined everywhere.
\end{remark}

\begin{figure}
  \includegraphics[scale=0.8]{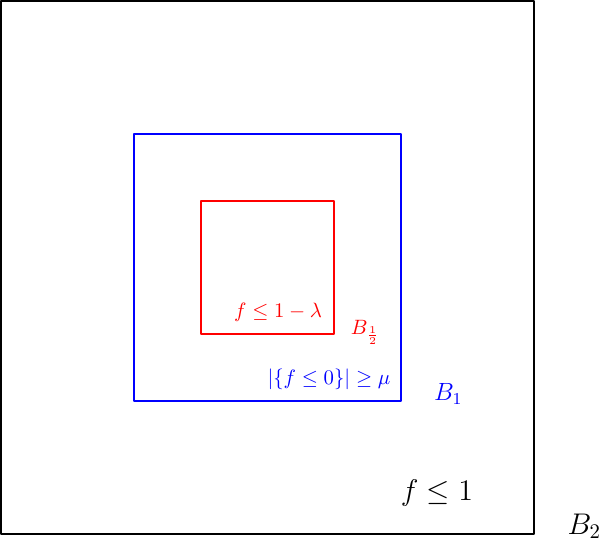}
  \caption{The setting of the decrease of maximum lemma.}
  \label{fig:m2p}
\end{figure}

\begin{proof}[Proof of Lemma~\ref{lem:dg2} when assuming Lemma \ref{lemma4}]
  Take $f$ a solution to~\eqref{eq:EL4} in $H^1(B_3)$. It is then bounded on $B_2$ by the first De Giorgi lemma. Assume $\mathrm{osc}_{B_2} f > 0,$ otherwise the statement is trivial. Then, we define 
  \begin{equation}
    \label{eq:osc-resc}
    g(x) := \frac{2}{\mathrm{osc}_{B_2} f}  \left[ f(x) - \frac{\sup_{B_2} \, f + \inf_{B_2} \, f}{2} \right].
  \end{equation}
  This new unknown $g$ is a solution to~\eqref{eq:EL4}, is valued in $[-1,1]$, and satisfies
  \begin{equation*}
    \mathrm{osc}_{B_2} \, g = 2.
  \end{equation*}

  Then, certainly,
  \begin{equation*}
    \text{ either } \ |B_1 \cap \{g \le 0\}| \ge \frac12 \, |B_1| \quad \text{ or } \quad |B_1 \cap \{g \ge 0\}| \ge \frac12 \, |B_1|.
  \end{equation*}
  
  In the first case, we apply Lemma~\ref{lemma4} to the subsolution $g$, with $\mu := 1/2$:
  \begin{equation*}
    \mathrm{osc}_{B_{\frac12}} g \le (2-\lambda) \le \(1 - \frac{\lambda}{2}\) \, \mathrm{osc}_{B_2} g
  \end{equation*}
  for some $\lambda \in (0,1)$. Going back to $f$, this means 
  \begin{equation*}
    \mathrm{osc}_{B_{\frac12}} f \le \(1- \frac{\lambda}{2}\) \, \mathrm{osc}_{B_2} f,
  \end{equation*}
  yielding Lemma~\ref{lem:dg2} with $\nu = 1-\lambda/2$. In the second case, we perform the same argument on the subsolution $-g$. Note that we are therefore using that both $f$ and $-f$ are subsolutions, i.e. that  $f$ is a solution.
\end{proof}

To show Lemma \ref{lemma4}, the main idea in De Giorgi's approach is to establish an $H^1$ version of the intermediate value lemma. Indeed, if $f$ were \emph{nearly everywhere non-positive} in the sense
\begin{equation*}
  |B_1 \cap \{f \le 0\}| \ge |B_1| - \delta^2,
\end{equation*}
for some small $\delta>0$, we would have 
\begin{equation*}
  \|f_+\|_{\mathrm{L}^2(B_1)} = \left( \int_{B_1} f_+^2 \dd x \right)^{\frac12} \le \left( \int_{B_1 \cap \{f>0\}} 1 \dd x \right)^{\frac12} \le \delta,
\end{equation*}
where we have used $f \le 1.$ Then, provided $\delta$ is small enough to apply Lemma~\ref{lem:dg1}, we would conclude
\begin{equation*}
  \sup_{B_{1/2}} f \le \frac12
\end{equation*}
as required. Our assumption is however merely
\begin{equation*}
  |B_1 \cap \{f\le 0\}| \ge \mu |B_1|.
\end{equation*}
The following $H^1$ intermediate value lemma of De Giorgi allows bridging the gap between these two assumptions by a finite induction.

\begin{lemma}[De Giorgi intermediate value Lemma]
  \label{lemma5}
  There is $C>0$, depending only on the dimension $n$, such that the following holds:
  \smallskip

  Let $g$ be a measurable function on $B_1$ so that  $g_+ \in H^1(B_1)$.

  Then, we have 
  \begin{equation}
    |B_1 \cap \{ 0 < g <1/2 \}| \ge \frac{C}{\|\nabla_x g_+\|^2 _{\mathrm{L}^2(B_1)}} \, |B_1 \cap \{g \ge 1/2\}|^2 \, |B_1 \cap \{ g \le 0\}|^{2-\frac{2}{n}}.  
  \end{equation}
\end{lemma}

\begin{remark}
  \begin{enumerate}
  \item The case $\|\nabla_x g_+\|_{\mathrm{L}^2(B_1)}=0$ is trivial and implicitly omitted.
  \item Inequalities of this form would be straightforward to prove if $g$ were regular and the constant were allowed to depend on a $C^1$ norm. 
  \item A weaker statement where $1/2$ is replaced by some $\theta \in (0,1)$ close to $1$ would be sufficient for the sequel, see later and~\cite{golse2019harnack}. 
  \item Unlike its counterpart in the kinetic setting (see later), this result does not rely on the PDE~\eqref{eq:EL4}, but only on the $H^1$ regularity of $g_+$. The non-constructive proof below shows in fact that the result still holds when replacing the bound on $g_+$ in $H^1(B_1)$ by  the weaker fractional Sobolev norm $H^s(B_1)$ with $s \ge 1/2$. However, the statement becomes false with the fractional Sobolev norm $H^s(B_1)$ for $s \in (0,1/2)$, since indicator functions belong to this Sobolev space.
  \item Nevertheless, importantly, the De Giorgi intermediate value lemma still holds, once properly formulated, for subsolutions to nonlocal elliptic equations with rough coefficients, even for orders of fractional ellipticity $2s$ with $s \in (0,1/2)$ (which corresponds to the $H^s$ regularity in the energy estimate). This is because the non-local version of the energy estimate includes an extra ``cross-term'' specific to non-local operators, see~\cite{zbMATH05913780}. A constructive method based on trajectories is also possible in this case, and it extends to the parabolic and kinetic cases, see~\cite{MR4688651,anceschi2024poincare}.  
  \item This lemma is also sometimes called the ``\textit{isoperimetric lemma of De Giorgi}'' because it gives a control of the measure of interior values $\{ 0< \tilde g < 1/2 \}$ by the measure of boundary values $\{ \tilde g =0 \}$ and $\{ \tilde g =1/2 \}$. 
  \end{enumerate}
\end{remark}

\begin{proof}[Proof of Lemma \ref{lemma5}]
  We present different methods of proof.
  \medskip

  \noindent
  \textbf{A. Non-constructive method by contradiction.} It is possible to give a short proof by contradiction at the cost of constructive estimates. Consider a contradiction sequence $g_k$ so that
  \begin{equation*}
    |B_1 \cap \{ 0 < g_k <1/2 \}| \xrightarrow{k \to \infty} 0,
  \end{equation*}
  and
  \begin{align*}
    & |B_1 \cap \{g_k \ge 1/2\}| \ge a_1, \\[1mm]
    & |B_1 \cap \{ g_k \le 0\}| \ge a_2,
  \end{align*}
  for some $a_1,a_2>0$ independent of $k$. Let the sequence $\tilde g_k$ be defined by
  \begin{equation*}
    \tilde g_k(x) = \begin{cases}
      0, \quad &\text{for } g_k(x) \le 0, \\[1mm]
      g_k(x), \quad &\text{for } 0 < g_k(x) < 1/2, \\[1mm]
      1/2, \quad &\text{for } g_k(x) \ge 1/2.
    \end{cases}
  \end{equation*}
  We can assume by compactness that a subsequence of the sequence $\tilde g_k$ is converging to some $g_\infty$ strongly in $L^2(B_1)$, weakly in $H^1(B_1)$, and almost everywhere. These convergences imply
  \begin{equation*}
    g_\infty \in H^1(B_1)
  \end{equation*}
  and 
  \begin{align*}
    & |B_1 \cap \{ 0 < g_\infty <1/2 \}| = 0,\\[1mm]
    & |B_1 \cap \{g_\infty \ge 1/2\}| \ge a_1>0 \\[1mm]
    & |B_1 \cap \{ g_\infty \le 0\}| \ge a_2 >0.
  \end{align*}
  Therefore $2 g_\infty = \mathbf{1}_{\mathcal B}$ is the indicator function of a Borel set $\mathcal B \subset B_1$ whose measure is positive and strictly less than $|B_1|$.

  Let us prove that $\mathbf{1}_{\mathcal B}$ is not in $H^1(B_1)$. A function $g \in H^1(B_1)$ satisfies, for any $\epsilon \in (0,1)$, the following inequality
  \begin{equation*}
    \forall \, h \in B_\epsilon, \quad \int_{B_{1-\epsilon}} |\mathbf{1}_{\mathcal B}(x+h)-\mathbf{1}_{\mathcal B}(x)|^2 \dd x \lesssim \| g \|_{H^1(B_1)} |h|^2. 
  \end{equation*}
  When we consider $g=\mathbf{1}_{\mathcal B}$, the integrand $|\mathbf{1}_{\mathcal B}(x+h)-\mathbf{1}_{\mathcal B}(x)|^2$ is valued in $\{0,1\}$, so the exponent $2$ can be changed to $1$ indifferently. Given $m \in \N^*$ and some fixed $h_0 \in B_\epsilon$, we have then
  \begin{align*}
    \int_{B_{1-\epsilon}} |\mathbf{1}_{\mathcal B}(x+h_0)-\mathbf{1}_{\mathcal B}(x)|^2 \dd x
    & = \int_{B_{1-\epsilon}} |\mathbf{1}_{\mathcal B}(x+h_0)-\mathbf{1}_{\mathcal B}(x)| \dd x \\
    & \le \sum_{j=0} ^{m-1} \int_{B_{1-\epsilon}} |\mathbf{1}_{\mathcal B-jh_0/m}(x)-\mathbf{1}_{\mathcal B - (j+1)h_0/m}(x)| \dd x \\
    & \le \sum_{j=0} ^{m-1} \int_{B_{1-\epsilon}+jh_0/m} |\mathbf{1}_{\mathcal B}(x)-\mathbf{1}_{\mathcal B - h_0/m}(x)| \dd x \\
    & \le \sum_{j=0} ^{m-1} \int_{B_{1-\epsilon}+jh_0/m} |\mathbf{1}_{\mathcal B}(x)-\mathbf{1}_{\mathcal B - h_0/m}(x)|^2 \dd x \\
    & \lesssim \sum_{j=0} ^{m-1} \left| \frac{h_0}{m} \right|^2 \lesssim m^{-1} |h_0|,
  \end{align*}
  where in the last line, we have used the $H^1$ bound on $\mathbf{1}_{\mathcal B}$. Therefore, by letting $m \to \infty$, we deduce
  \begin{equation*}
    \int_{B_{1-\epsilon}} |\mathbf{1}_{\mathcal B}(x+h_0)-\mathbf{1}_{\mathcal B}(x)|^2 \dd x =0.
  \end{equation*}
  This implies that $\mathbf{1}_{\mathcal B}$ is constant in $B_{1-\epsilon}$, and by letting $\epsilon \to 0$, we deduce that $\mathbf{1}_{\mathcal B}$ is constant in $B_1$, which contradicts the fact that the measure of $\mathcal B$ is non-zero and strictly less than that of $B_1$.

  Note that this argument can be easily generalised to assuming only $H^s(B_1)$ for $s \in (1/2,1)$, rather than $H^1(B_1)$. The fractional Sobolev space $H^s(B_1)$ can indeed by defined by requiring that the following seminorm is finite:
  \begin{equation*}
    \int_{x,y \in B_1} \frac{|g(x)-g(y)|^2}{|x-y|^{n+2s}} \dd x \dd y < \infty
  \end{equation*}
  for $g \in H^s(B_1)$. When we consider $g=\mathbf{1}_{\mathcal B}$, again the integrand $|\mathbf{1}_{\mathcal B}(y)-\mathbf{1}_{\mathcal B}(x)|^2$ is valued in $\{0,1\}$, so the exponent $2$ can be changed to $1$ indifferently. Given $m \in \N^*$, we have then
  \begin{align*}
    & \int_{x,y \in B_1} \frac{|\mathbf{1}_{\mathcal B}(y)-\mathbf{1}_{\mathcal B}(x)|^2}{|y-x|^{n+2s}} \dd y \dd x \\
    & = \int_{x \in B_1} \int_{h \in \R^n} \mathbf{1}_{x+h \in B_1} \frac{|\mathbf{1}_{\mathcal B}(x+h)-\mathbf{1}_{\mathcal B}(x)|^2}{|h|^{n+2s}} \dd h \dd x \\
    & = \int_{x \in B_1} \int_{h \in \R^n} \mathbf{1}_{x+h \in B_1} \frac{|\mathbf{1}_{\mathcal B}(x+h)-\mathbf{1}_{\mathcal B}(x)|}{|h|^{n+2s}} \dd h \dd x \\
    & \le \sum_{j=0} ^{m-1} \int_{x \in B_1} \int_{h \in \R^n} \mathbf{1}_{x+h \in B_1} \frac{|\mathbf{1}_{\mathcal B - jh/m}(x)-\mathbf{1}_{\mathcal B-(j+1)h/m}(x)|}{|h|^{n+2s}} \dd h \dd x  \\
    & \le \sum_{j=0} ^{m-1} \int_{x \in B_1 +jh/m} \int_{h \in \R^n} \mathbf{1}_{x+ \frac{(m-j)h}{m} \in B_1} \frac{|\mathbf{1}_{\mathcal B}(x)-\mathbf{1}_{\mathcal B-h/m}(x)|}{|h|^{n+2s}} \dd h \dd x  \\
    & \le m^{-2s} \sum_{j=0} ^{m-1} \int_{x \in B_1 + jh/m} \int_{h \in \R^n}  \mathbf{1}_{x+\frac{(m-j)h}{m} \in B_1} \frac{|\mathbf{1}_{\mathcal B}(x)-\mathbf{1}_{\mathcal B-h/m}(x)|^2}{|h/m|^{n+2s}} \dd \left( \frac{h}{m} \right) \dd x  \\
    & \lesssim m^{1-2s},
  \end{align*}
  where in the last line, we have used the $H^s(B_1)$ bound on $\mathbf{1}_{\mathcal B}$. Therefore, by letting $m \to \infty$, we deduce
  \begin{equation*}
    \int_{x,y \in B_1} \frac{|\mathbf{1}_{\mathcal B}(x)-\mathbf{1}_{\mathcal B}(y)|^2}{|x-y|^{n+2s}} \dd x \dd y =0.
  \end{equation*}
  This implies that $\mathbf{1}_{\mathcal B}$ is constant in $B_1$, which again contradicts the fact that the measure of $\mathcal B$ is non-zero and strictly less than that of $B_1$.

  In the limit case $s=1/2$, it is still possible to prove that $\mathbf{1}_{\mathcal B}$ is not in $H^{1/2}(B_1)$; the indicator function is however in $H^s(B_1)$ for $s \in [0,1/2)$. We refer for instance to~\cite{MR3032092} for these results. 
  \medskip

  \noindent
  \textbf{B. Constructive method based on trajectories.} This proof was given in~\cite{vasseur2016giorgi}. The idea is to connect the points where $g > 1/2$ with those where $g<0$ via trajectories driven by the vector fields $\partial_{x_i}$ inducing the diffusion; in this elliptic setting these trajectories are simply straight lines, see Figure~\ref{fig:ell-traj}.

  Let $\tilde{g}$ be defined by
  \begin{equation*}
    \tilde g(x) = \begin{cases}
      0, \quad &\text{for } g(x) \le 0, \\[1mm]
      g(x), \quad &\text{for } 0 < g(x) < 1/2, \\[1mm]
      1/2, \quad &\text{for } g(x) \ge 1/2.
    \end{cases}
  \end{equation*}
  \begin{figure}
    \includegraphics[scale=0.8]{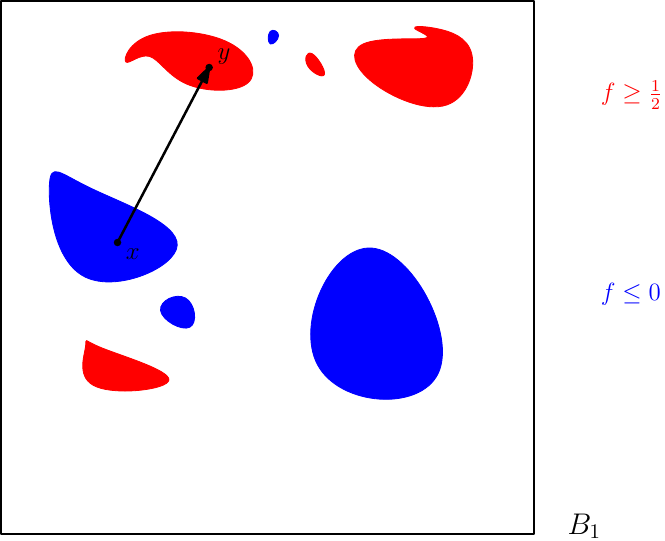}
    \caption{Trajectories in the elliptic case.}
    \label{fig:ell-traj}
  \end{figure}
  
  This modified unknown satisfies $|\nabla_x \tilde g| \le |\nabla_x g_+|$ on $B_1$. Given then any
  \begin{align*}
    & x \in B_1 \cap \{ g \le 0\}, \\[1mm]
    & y \in B_1 \cap \{g \ge 1/2\},
  \end{align*}
  we Taylor-expand between the two points:
  \begin{equation*}
    \frac12 = \tilde g(y) - \tilde g(x) = \int_0^1 (y-x) \cdot \nabla_x \tilde g (x + t (y-x)) \dd t
  \end{equation*}
  (the manipulations can be a posteriori justified by density of smooth functions). Hence, we have
  \begin{align*}
    \frac12
    & \le \int_0^{|y-x|} \left|\nabla_x \tilde g \left(x + s \frac{y-x}{|y-x|} \right) \right| \dd s \\
    & \le \int_0^\infty \left|\nabla_x \tilde g \left( x + s \frac{y-x}{|y-x|} \right) \right| \mathbf 1_{x + s \frac{y-x}{|y-x|} \in B_1} \dd s.
  \end{align*}
  
  Integrating $y$ over the set $B_1 \cap \{g \ge 1/2\}$ yields 
  \begin{align*}
    \frac{1}{2} \, |B_1 \cap \{g \ge 1/2\}| 
    & \le \int_{y \in B_1 \cap \{g\ge 1/2\}} \int_0^\infty \left|\nabla_x \tilde g \left( x + s \frac{y-x}{|y-x|} \right) \right| \mathbf 1_{x + s \frac{y-x}{|y-x|} \in B_1} \dd s \dd y \\
    & \le  \int_{B_1} \int_0^\infty \left|\nabla_x \tilde g \left( x + s \frac{y-x}{|y-x|} \right) \right| \mathbf 1_{x + s \frac{y-x}{|y-x|} \in B_1} \dd s \dd y, 
  \end{align*}
  which, after a spherical change of variables of $y$ around $x$ gives 
  \begin{align*}
    \frac{1}{2} \, |B_1 \cap \{g \ge 1/2\}|
    & \le \left( \int_0^2 r^{n-1} \dd r \right) \, \int_{\sigma \in \mathbb S^{d-1}} \int_0^\infty \left|\nabla_x \tilde g \left( x + s \, \sigma \right) \right| \mathbf 1_{x + s \sigma \in B_1}\dd s \dd \sigma \\
    & \le \frac{2^n}{n} \int_{\R^n} \frac{|\nabla_x \tilde g(z)|}{|x-z|^{n-1}} \mathbf 1_{z \in B_1} \dd z \\
    & \le \frac{2^n}{n} \int_{B_1} \frac{|\nabla_x \tilde g(z)|}{|x-z|^{n-1}} \dd z,
  \end{align*}
  thanks to another spherical change of variables $(s,\sigma) \mapsto z = x + s\, \sigma$. 
  
  Integrating $x$ over $B_1 \cap \{g \le 0\}$ gives 
  \begin{equation*}
    \frac{1}{2} |B_1 \cap \{g\le 0\}| \, |B_1 \cap \{g \ge 1/2\}| \le \frac{2^n}{n} \, \int_{B_1 \cap \{g \le 0\}} \int_{B_1} \frac{|\nabla_x \tilde g(z)|}{|x-z|^{n-1}} \dd z \dd x.
  \end{equation*}
  We now use the standard rearrangement inequality 
  \begin{align*}
    \int_{B_1 \cap \{g \le 0\}} \frac{\mathrm dx}{|z-x|^{n-1}}
    & \le \int_{B_{r_0}} \frac{\mathrm dx}{|x|^{n-1}} \\
    & \le r_0 \\
    & \lesssim |B_1 \cap \{g \le 0\}|^{\frac1n},
  \end{align*}
  with $r_0>0$ is chosen so that
  \begin{equation*}
    |B_{r_0}| = |B_1 \cap \{g\le0\}|.
  \end{equation*}
  
  Altogether, we obtain
  \begin{align*}
    & \frac{1}{2} |B_1 \cap \{g\le 0\}| \, |B_1 \cap \{g \ge 1/2\}| \\
    & \le \frac{2^n}{n} \left( \int_{B_1} |\nabla_x \tilde g| \right) |B_1 \cap \{g \le 0\}|^{\frac1n} \\
    & \le \frac{2^n}{n} \left( \int_{B_1 \cap \{0<g<1/2\}} |\nabla_x \tilde g| \right) |B_1 \cap \{g \le 0\}|^{\frac1n} \\
    & \le \frac{2^n}{n} \|\nabla_x g_+\|_{L^2(B_1)} \, |B_1 \cap \{0<g<1/2\}|^{\frac12} \, |B_1 \cap \{g\le 0\}|^{\frac1n} 
  \end{align*}
  which concludes the proof with $C = n^2 2^{-2-2n}$.
  \medskip

  \noindent
  \textbf{C. Less sharp but simpler constructive method based on the Poincaré inequality.} This argument was proposed in~\cite{MR4398231,guerand2020quantitativeregularityparabolicgiorgi}. It is also---in another respect---close to the method we will deploy in the hypoelliptic setting. We apply the Poincaré inequality in the $L^1$ setting to $\tilde g$ in $B_1$:
  \begin{equation}
    \label{eq:poincare-l1}
    \int_{B_1} \left| \tilde g-\langle \tilde g \rangle_{B_1}\right| \dd x \lesssim \int_{B_1} \left| \nabla_x \tilde g \right| \dd x
  \end{equation}
  where $\langle \tilde g \rangle_{B_1}$ is the average of $\tilde g$ on $B_1$.

  Then, we observe
  \begin{equation*}
    \frac{|B_1 \cap \{ g \ge 1/2\}|}{|B_1|} \le \frac{2}{|B_1|} \int_{B_1 \cap \{g \ge 1/2\}} \tilde g \dd x \le 2 \langle \tilde g \rangle_{B_1}
  \end{equation*}
  so that
  \begin{align*}
    |B_1 \cap \{g\le 0\}| |B_1 \cap \{ g \ge 1/2\}|
    & \le 2 |B_1| \int_{B_1 \cap \{g\le 0\}} \langle \tilde g \rangle_{B_1} \dd x \\
    & \le 2 |B_1| \int_{B_1 \cap \{g\le 0\}} \left| \tilde g(x) - \langle \tilde g \rangle_{B_1} \right| \dd x \\
    & \le 2 |B_1| \int_{B_1} \left| \tilde g(x) - \langle \tilde g \rangle_{B_1} \right| \dd x.
  \end{align*}
  We have added $g(x)$ in the second integrand because it is zero on the domain of integration.

  We then plug the Poincaré inequality~\eqref{eq:poincare-l1} in this last estimate:
  \begin{equation*}
    |B_1 \cap \{g\le 0\}| |B_1 \cap \{ g \ge 1/2\}| \lesssim \int_{B_1} \left| \nabla_x \tilde g \right| \dd x. 
  \end{equation*}
  In this last integrand, we observe that $\nabla_x \tilde g=0$ almost everywhere when restricted to the sets of points where $\tilde g=0$ or $\tilde g=1/2$. These sets are respectively made up of local minima and maxima of $\tilde g$: the gradient would be zero for $g \in C^1$ by standard differential calculus, and the proof in $H^1$ is done by approximation, see~\cite[Section~4.2.2, Theorem~4-(iii), p.~129]{zbMATH08010281}. Therefore
  \begin{align*}
    \int_{B_1} \left| \nabla_x \tilde g \right| \dd x
    & = \int_{B_1 \cap \{0<g<1/2\}} \left| \nabla_x \tilde g \right| \dd x \\
    & \lesssim |B_1 \cap \{0<g<1/2\}|^{\frac12} \left\| \nabla_x \tilde g \right\|_{L^2(B_1)}
      \\
    & \lesssim |B_1 \cap \{0<g<1/2\}|^{\frac12} \left\| \nabla_x g_+ \right\|_{L^2(B_1)}
  \end{align*}
  by Cauchy-Schwarz' inequality and the comparison $|\nabla_x \tilde g| \le |\nabla_x g_+|$. Altogether we deduce the inequality
  \begin{equation*}
    |B_1 \cap \{ 0 < g <1/2 \}| \ge \frac{C}{\|\nabla_x g_+\|^2_{\mathrm{L}^2(B_1)}} \, |B_1 \cap \{g \ge 1/2\}|^2 \, |B_1 \cap \{ g \le 0\}|^{2}
  \end{equation*}
  for some constant $C>0$. This alternative argument proves a slightly weaker inequality where the last exponent is $2$ instead of $2-2/n$, which results in a worse lower bound for sets with small measure. 
\end{proof}

\medskip

\begin{proof}[Proof of the decrease of supremum lemma~\ref{lemma4} from the intermediate value lemma~\ref{lemma5}]
  \text{ } \\
  Let $f \le 1$ subsolution in $H^1(B_2)$ to
  \begin{equation*}
    \nabla_x \cdot ( A \nabla_x f ) \ge 0
  \end{equation*}
  with $A$ symmetric measurable so that $\Lambda^{-1} \le A \le \Lambda$ almost everywhere. We assume that $f$ satisfies the measure condition
  \begin{equation*}
    |B_1 \cap \{f\le 0\}| \ge \mu |B_1|
  \end{equation*}
  for some $\mu>0$.

  Let us define the following sequence of new unknowns on $B_2$:
  \begin{equation*}
    \forall \, k \ge 0, \quad g_k := 2^k \left[ f- \(1-2^{-k}\) \right].
  \end{equation*}
  They satisfy the inductive relation
  \begin{equation*}
    \forall \, k \ge 0, \quad g_{k+1} = 2g_k -1
  \end{equation*}
  and the property $g_k \le 1$ for all $k \ge 0$ (see Figure~\ref{fig:ell-ivl-dg} for the representation of the successive zooms on values of $f$ close to $1$).

  Thus, $(g_k)_+$ is a subsolution on $B_2$ and is always valued in $[0,1]$. By relating the $H^1(B_1)$ norm of $(g_k)_+$ to the $L^2(B_2)$ norm of $(g_k)_+$, the Caccioppoli inequality~\eqref{eq:ee2} (local energy estimate) therefore shows that $(g_k)_+ \in H^1(B_1)$ with a bound independent of $k \ge 0$. 
  \begin{figure}
    \includegraphics[scale=0.8]{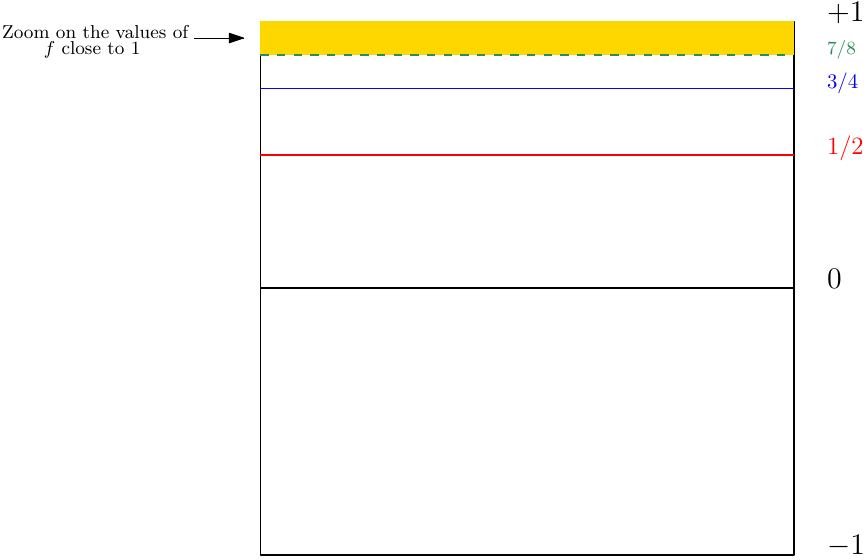}
    \caption{Successive zooms and dilations on the values of $f$ close to $1$.}
    \label{fig:ell-ivl-dg}
  \end{figure}
  
  We now want to apply the intermediate value lemma to each $(g_k)_+$. Since $f(x) \le 0$ implies $g_k(x) \le 0$ for all $k \ge 0$, we have 
  \begin{equation*}
    |B_1 \cap \{g_k \le 0\}| \ge |B_1 \cap \{f \le 0\}| \ge \mu|B_1|.
  \end{equation*}
  If the inequality
  \begin{equation*}
    \int_{B_1} |(g_0)_+|^2 = \int_{B_1} |f_+|^2 \le \delta^2
  \end{equation*}
  holds with $\delta$ as in the first De Giorgi lemma~\ref{lem:dg1}, then this lemma~\ref{lem:dg1} indeed applies and implies $\sup_{B_{1/2}} f \le 1/2$, so we are done. Suppose then, instead, that
  \begin{equation*}
    \int_{B_1} (g_0)_+^2 \dd x > \delta^2,
  \end{equation*}
  and consider any $k_0 \ge 1$ such that
  \begin{equation}
    \label{eq:bad-control}
    \forall \,  0 \le k \le k_0, \quad \int_{B_1} (g_k)_+^2 \dd x > \delta^2.
  \end{equation}
  Then, for $k=0,\cdots,k_0-1$, we have that 
  \begin{align*}
    |B_1 \cap \{g_k \ge 1/2\}|
    & = |B_1 \cap \{g_{k+1} \ge 0\}| \\
    & \ge \int_{B_1} g_{k+1}^2 \dd x > \delta^2,
  \end{align*}
  where we have used $g_k = 1/2 + g_{k+1}/2$ and $g_{k+1}\le1$.

  By applying Lemma~\ref{lemma5} to every $(g_k)_+$ for $k=0,\cdots,k_0-1$, we deduce that there is $\alpha>0$, depending on $n$, $\mu$, $\delta$, and the uniform bound on the $H^1(B_1)$ norm, such that 
  \begin{equation*}
    \forall \,  k = 0,\cdots,k_0-1, \quad |B_1 \cap \{g_k \in (0,1/2)\} \ge \alpha |B_1|.
  \end{equation*}
  Then, 
  \begin{equation*}
    \begin{aligned}
      |B_1 \cap \{g_{k_0} \le 0\}|
      & \ge |B_1 \cap \{ g_{k_0-1} \in (0,1/2) \} | + |B_1 \cap \{g_{k_0-1} \le 0\}| \\
      & \ge \alpha |B_1| + |B_1 \cap \{g_{k_0-1} \le 0\}| \\
      & \ge  \cdots  \\
      & \ge k_0 \alpha |B_1| + |B_1 \cap \{g_0 \le 0\}| \ge ( \mu + \alpha \, k_0 ) |B_1|.
    \end{aligned}
  \end{equation*}
  Therefore $\mu + k_0 \alpha \le 1$ which shows that $k_0$ cannot exceed the value
  \begin{equation*}
    K_0 := \frac{(1-\mu)}{\alpha}.
  \end{equation*}

  We can therefore define $k_0 \ge 1$ as the largest integer so that~\eqref{eq:bad-control} holds. Then, $k_0+1 \le K_0+1$, and we have
  \begin{equation*}
    \int g_{k_0+1}^2 \dd x \le \delta^2.
  \end{equation*}
  This implies, by Lemma~\ref{lem:dg1}, that
  \begin{equation*}
    \sup_{B_{1/2}} g_{k_0+1}\le 1/2 
  \end{equation*}
  which translates into the bound
  \begin{equation*}
    \sup_{B_{1/2}} f\le 1- 2^{-k_0-2}
  \end{equation*}
  on $f$. This concludes the proof with $\lambda := 2^{-k_0-2} \ge 2^{-(1-\mu)/\alpha-2}$.
\end{proof}

\begin{remark}
  By tracking the behaviour of the constant $\alpha$ in terms of $\delta$, and then the behaviour of $\delta$ in terms of $\Lambda$, it is possible to obtain the more explicit formula
  \begin{equation*}
    \lambda \sim C e^{-C' \frac{(1-\mu)}{\mu^{2-2/n}} \Lambda^4}
  \end{equation*}
  for some constant $C,C'>0$ depending only on $n$. The constant $\mu$ is $1/2$ in practice when applying the intermediate value lemma to derive the decrease of maximum, so we deduce
  \begin{equation*}
    \lambda \sim C e^{-C '' \Lambda^4}
  \end{equation*}
  for some new absolute constant $C''>0$. We can also track the Hölder regularity exponent in terms of $\Lambda$, with the help of explicit formulae $\alpha = - (\ln \nu)/(2\, \ln 2)$ and $\nu = 1-\lambda/2$. This leads finally to $\alpha \sim [C/(4\ln 2)] e^{-C '' \Lambda^4}$. These estimates on the constants are obviously not optimal, but the exponential dependency of $\alpha$ in terms of a power of the ellipticity constant $\Lambda$ is likely to be structural.
\end{remark}

\subsection{The method of Moser for controlling the oscillations}

The approach of Moser~\cite{moser1960new,MR159138} regarding the second De Giorgi lemma differs from that of De Giorgi in two crucial aspects:
\begin{enumerate}
\item Moser measures oscillations by \emph{integrals} rather than
  \emph{pointwise}, specifically through a \textbf{Poincaré inequality}.
\item Influenced by Nash's paper~\cite{nash1958continuity} and in particular by the ``G bound'' in~\cite{nash1958continuity}\footnote{See also ``\emph{The masterpieces of John Forbes Nash Jr}'' by De Lellis in~\cite{MR3930576} for a recent detailed rewriting of the proof of Nash.}, Moser derives energy estimates on the logarithm of the solution and exploits the collision of scales at that level.
\end{enumerate}

Moser first ``pushes'' further the first De Giorgi lemma. He denotes
\begin{equation*}
  \Phi(p,r) := \( \frac{1}{|B_r|} \int_{B_r} f_+^p \dd x \)^{\frac1p}
\end{equation*}
which converges to $\sup_{B_r} f_+$ for $p\to +\infty$ and to $\min_{B_r} f_+$ for $p \to -\infty$. He then shows that for $f$ subsolution then
\begin{equation*}
  \forall \, p >0, \quad \Phi(+\infty,1) \lesssim_p \Phi(p,2),
\end{equation*}
and for $f$ supersolution then
\begin{equation*}
  \forall \, p >0, \quad \Phi(-\infty,1) \gtrsim_p \Phi(-p,2).
\end{equation*}
These two estimates are proved by the localised energy estimate, the Sobolev inequality and the De Giorgi--Moser iteration. 

He then invokes Poincaré's inequality to prove a universal bound on the mean oscillation of the logarithm of the solution (see~\eqref{eq:uni-bmo} below), and applies a result by John and Nirenberg~\cite{john1961functions} about functions with bounded mean oscillations (applied to the logarithm of the solution) to deduce
\begin{equation*}
  \Phi(p,2) \lesssim \Phi(-p,2)
\end{equation*}
for $p>0$ small. To satisfy the assumptions of the John-Nirenberg result and prove the latter inequality, Moser must prove that if $f$ is a positive supersolution then $F = \ln f$ satisfies
\begin{equation}
  \label{eq:uni-bmo}
  \int_{B_r(x_0)} \(F-\langle F \rangle_{B_r(x_0)}\)^2 \dd x \lesssim |B_r(x_0)|
\end{equation}
for any ball $B_r(x_0)$ included in the domain, where $\langle F \rangle_{B_r(x_0)}$ is the average on the ball. Thanks to the Poincaré inequality, the inequality~\eqref{eq:uni-bmo} is implied by 
\begin{equation}
  \label{eq:uni-grad}
  \int_{B_r(x_0)} | \nabla_x F |^2 \dd x \lesssim r^{-2} |B_r(x_0)|.
\end{equation}
This universal bound~\eqref{eq:uni-grad} follows from the localised energy estimate on the subequation
\begin{equation*}
  \nabla_x \cdot ( A \nabla_x F) + A \nabla_x F \cdot \nabla_x F \le 0
\end{equation*}
satisfied by $F = \ln f$. The second non-negative term $A \nabla_x F \cdot \nabla_x F$ is produced by having taken logarithm of the solution, and is the key to this universal estimate.

The chain of inequalities then implies $\Phi(+\infty,1) \lesssim \Phi(-\infty,1)$ which finall proves the following \textbf{Harnack inequality}: there is $C_{\mathrm{H}} >0$, depending only on $n$ and $\Lambda$, so that for any $f \ge 0$ non-negative solution to~\eqref{eq:EL4} on $B_2$, then
\begin{equation*}
  \sup_{B_1} f \le C_{\mathrm{H}} \min_{B_1} f.
\end{equation*}

Let us show how this powerful inequality implies the reduction of oscillation~\eqref{eq:dg2}. By scaling the Harnack inequality also holds in $B_{1/2}$. The solution $f$ to~\eqref{eq:EL4} on $B_2$ is bounded on $B_1$ by the first De Giorgi lemma. So we can perform a rescaling similar to~\eqref{eq:osc-resc} and define the new unknown
\begin{equation*}
  g(x) := \frac{2}{\mathrm{osc}_{B_1} f}  \left[ f(x) - \frac{\sup_{B_1} \, f + \inf_{B_1} \, f}{2} \right]
\end{equation*}
which solves~\eqref{eq:EL4} in $B_2$.

This $g$ satisfies
\begin{equation*}
  \min_{B_1} g = -1 \ \text{ and } \ \max_{B_1} g =+1,
\end{equation*}
and we denote
\begin{equation*}
  m_1 := \max_{B_{\frac12}} g \ \text{ and } \ m_0:= \min_{B_{\frac12}} g.
\end{equation*}

We then apply the Harnack inequality in $B_{1/2}$ to the following two non-negative solutions to~\eqref{eq:EL4} in $B_1$: $1-g \ge 0$ and $g+1 \ge 0$. This yields
\begin{align*}
  & 1-m_0 = \max_{B_{1/2}} (1-g) \le C_{\mathrm{H}} \min_{B_{1/2}} (1-g) = C_{\mathrm{H}} (1-m_1), \\
  & 1+ m_1 = \max_{B_{1/2}} (g+1) \le C_{\mathrm{H}} \min_{B_{1/2}} (g+1) = C_{\mathrm{H}} (1+ m_0).
\end{align*}
Summing the two inequalities we get
\begin{equation*}
  m_1-m_0 \le \frac{C_{\mathrm{H}}-1}{C_{\mathrm{H}}+1} 2
\end{equation*}
which is exactly the reduction of oscillation with constant
\begin{equation*}
  \nu := \frac{C_{\mathrm{H}}-1}{C_{\mathrm{H}}+1} \in (0,1).
\end{equation*}
We note that the Moser approach seems to yield better constants in the Harnack inequality and Hölder regularity in terms of the elliptic constant $\Lambda$: we refer to~\cite{dmnz} for a discussion.

\subsection{The method of Kruzhkov for controlling the oscillations}

Kruzhkov's idea~\cite{kruzhkov1963priori} is to push even further the study of the logarithm of the solution that was crucial in both Nash and Moser's arguments. He extracts more information from the ``collision of scales'' on the logarithm of the solution (see below) and simplifies the proof by directly proving the decrease of maximum lemma. (In fact, he proves the equivalent symmetric form: the \emph{increase of the minimum} for supersolutions.) 

Let us consider $f \ge 0$ a non-negative supersolution in $H^1(B_2)$ to
\begin{equation*}
  \nabla_x \cdot ( A \nabla_x f) \le 0
\end{equation*}
with the usual assumptions on $A$, and so that
\begin{equation}
  \label{eq:hyp-mes-kru}
  |B_1 \cap \{f \ge 1\}| \ge \mu |B_1|
\end{equation}
for some $\mu >0$. The following argument of Kruzhkov shows that there is $\lambda \in (0,1)$ so that $f \ge \lambda$ on $B_{1/2}$. This proves Lemma~\ref{lemma4} on $1-f$. 

Let us consider the function
\begin{equation*}
 \forall \, z >0, \quad G(z) = (-\ln z + z -1) \mathbf 1_{z \le 1}.
\end{equation*}
It is twice differentiable, convex, and satisfies
\begin{equation*}
  \forall \, z \in \R_+^*, \quad
  \begin{cases}
    \displaystyle G'(z) = \frac{z-1}{z} \, \mathbf 1_{z \le 1} \le 0, \\[2mm]
    \displaystyle G''(z) = \frac{1}{z^2} \, \mathbf 1_{z \le 1} \ge 0 \\[2mm]
    \displaystyle G''(z) \ge G'(z)^2.
  \end{cases}
\end{equation*}
With the help of this function, we define the new unknown
\begin{equation*}
  g := G \left( f_\delta \right) \quad \text{ with the notation } \quad f_\delta := \frac{f + \delta}{1+\delta},
\end{equation*}
and for $\delta>0$ to be chosen later. By translation and dilation, $f_\delta$ is also a supersolution, and $g$ is a subsolution which satisfies
\begin{align*}
  0 \le A \nabla_x g \cdot \nabla_x g
  & = G''\( f_\delta \) A \nabla_x f_\delta \cdot \nabla_x f_\delta \\
  & \le G''(g) A \nabla_x f_\delta \cdot \nabla_x f_\delta + \underbrace{G'(g) \nabla_x \cdot \( A \nabla_x f_\delta\)}_{\ge 0} = \nabla_x \cdot \( A \nabla_x g \).
\end{align*}
Note that the term $G'(g) \nabla_x \cdot ( A \nabla_x f_\delta)$ is non-negative because $G' \le 0$ and $f$ is a supersolution, so $\nabla_x \cdot ( A \nabla_x f_\delta) \le 0$.

Integrating this inequation against a smooth localisation function $\phi$ so that
\begin{equation*}
  \mathbf{1}_{B_1} \le \phi \le \mathbf{1}_{B_2},
\end{equation*}
we deduce
\begin{equation*}
  \int_{\R^n} \left( A \nabla_x g \cdot \nabla_x g \right) \phi^2 \dd x \le \int_{\R^n} \left[ \nabla_x \cdot \( A \nabla_x g \) \right] \phi^2 \dd x.
\end{equation*}
By integration by parts, Cauchy-Schwarz' inequality, and the bounds $\Lambda^{-1} \le A \le \Lambda$ on the matrix $A$, we deduce from the last estimate:
\begin{equation*}
  \int_{\R^n} \left| \nabla_x g \right|^2 \phi^2 \dd x \le \Lambda^2 \left( \int_{\R^n} \left| \nabla_x g \right|^2 \phi^2 \dd x\right)^{\frac12} \left( \int_{\R^n} \left| \nabla_x \phi \right|^2 \dd x \right)^{\frac12}.
\end{equation*}
We end up with the universal bound
\begin{equation}
  \label{eq:krukru}
  \|\nabla_x g\|^2_{\mathrm{L}^{2}(B_1)} \lesssim \Lambda^4  \int_{B_2} \left| \nabla_x \phi \right|^2 \lesssim \Lambda^4.
\end{equation}

We then want to apply the Poincaré inequality to $G(f_\delta)$ in $B_1$ in order to control the $L^2$ norm. This requires to control the average. This is where we use the assumption~\eqref{eq:hyp-mes-kru} to control the average by the $L^2$ norm. Indeed $f \ge 1$ is equivalent to $f_\delta \ge 1$ which is itself equivalent to $g=G(f_\delta) =0$. Therefore
\begin{equation*}
  \langle g \rangle_{B_1} = \frac{1}{|B_1|} \int_{\{ f < 1\} \cap B_1} g \dd x \le \sqrt{1-\mu} \( \frac{1}{|B_1|} \int_{B_1} g^2 \dd x \)^{\frac12}. 
\end{equation*}
Since by orthogonality we have 
\begin{equation*}
  \frac{1}{|B_1|} \int_{B_1} g^2 \dd x \le \frac{1}{|B_1|} \int_{B_1} \left[ g - \langle g \rangle_{B_1}\right]^2 \dd x + \langle g \rangle_{B_1}^2,
\end{equation*}
we deduce from the two last inequalities
\begin{equation*}
  \frac{1}{|B_1|} \int_{B_1} g^2 \dd x \le \frac{1}{\mu|B_1|} \int_{B_1} \left[ g - \langle g \rangle_{B_1}\right]^2 \dd x.
\end{equation*}
We then apply the Poincaré inequality to $g$ on $B_1$:
\begin{align*}
  \frac{1}{|B_1|} \int_{B_1} g^2 \dd x \le \frac{1}{\mu|B_1|} \int_{B_1} \left[ g - \langle g \rangle_{B_1}\right]^2 \dd x \lesssim  \frac1\mu \|\nabla_x g\|^2_{\mathrm{L}^{2}(B_1)} \lesssim \frac{\Lambda^4}{\mu}
\end{align*}
We finally use the first De Giorgi lemma~\ref{lem:dg1} on the non-negative subsolution $g$: 
\begin{equation*}
  \|g\|_{L^\infty(B_{1/2})} \lesssim \|g\|_{L^2(B_1)} \lesssim \frac{\Lambda^2}{\sqrt{\mu}}. 
\end{equation*}
But the supremum of $g$ on $B_{1/2}$ can be expressed in terms of the infimum of $f$ since $G$ is decreasing and we obtain
\begin{equation*}
  \left| \ln \left( \frac{\inf_{Q_{1/2}} f + \delta}{1+\delta} \right) \right| \lesssim \frac{\Lambda^2}{\sqrt{\mu}}.
\end{equation*}
This proves the lower bound
\begin{equation*}
  \inf_{Q_{1/2}} f \ge \lambda \quad \text{ with } \quad \lambda = C e^{-C' \frac{\Lambda^2}{\sqrt{\mu}}} >0
\end{equation*}
for some constant $C,C'>0$. This concludes the proof.

\begin{remarks}
  \begin{enumerate}
  \item In fact, Kruzhkov's approach also leads to an alternative proof
    of the Harnack inequality. The increase of the minimum is combined with
    an estimate of propagation (spreading) of the lower bound to obtain a ``weak Harnack inequality'' which, combined with the first De Giorgi lemma (gain of integrability), implies the Harnack inequality (see the next remark). We refer to~\cite{MR171086,imbert2019weak,MR4653756} for more details and for the extension to the kinetic case.
    
  \item The Harnack inequality readily follows from the so-called \textbf{weak Harnack inequality} combined with the first De Giorgi lemma. The weak Harnack inequality states as follows: there are $\zeta>0$ and  $C_{\mathrm{WH}} >0$, depending only on $n$ and $\Lambda$, so that for any $f \ge 0$ non-negative solution to~\eqref{eq:EL4} on $B_2$, then
    \begin{equation*}
      \int_{B_{\frac12}} f^\zeta \dd x \le C_{\mathrm{WH}} \inf_{B_1} f.
    \end{equation*}
    
  \item An alternative route to the weak Harnack inequality is to simply deduce it from the measure-to-pointwise upper bound (decrease of supremum lemma) by a covering argument: we refer to~\cite[Proof of Lemma~5.14]{MR2777537} for a non-constructive argument in the elliptic case, and to~\cite{guerand2022quantitative,MR4688651} for a constructive argument that extends to the kinetic case.
    
  \item The dependency of the constants obtained by the Kruzhkov approach in terms of the elliptic constant $\Lambda$ has not been explored to our knowledge.
\end{enumerate}

\end{remarks}

\section{The parabolic case}

This was the setting first considered by Nash~\cite{nash1958continuity}. However, the approaches of De Giorgi, Moser and Kruzhkov were also extended to this case. The parabolic second De Giorgi lemma was however only proved by non-constructive methods, until the recent works~\cite{MR4398231,guerand2020quantitativeregularityparabolicgiorgi,guerand2022quantitative,MR4653756}. We only sketch the proofs when they resemble the elliptic case. 

We consider $\cU \subset \R^n$ open set, $T>0$, and the equation
\begin{equation}
  \label{eq:parabolic}
  \partial_t f = \nabla_x \cdot \( A \, \nabla_x f \) \quad \text{in} \, \, (0,T) \times \cU
\end{equation}
with a matrix-valued function $A=A(t,x)$. We define the Hölder space $C^\alpha((0,T) \times \cU)$ as in~\eqref{eq:holder} by replacing $\cU$ by $(0,T) \times \cU \subset \R \times \R^n$. The parabolic counterpart of the main theorem~\ref{thm:dgnm2} of the elliptic case is the following statement.
\begin{theorem}[Parabolic De Giorgi--Nash--Moser]
  \label{thm:dgnm3}
  Given $\Lambda>0$, there is $\alpha\in (0,1)$, depending only on $\Lambda$ and the dimension $n$, so that the following holds:
  \smallskip
  
  Let $f$ be a weak solution to~\eqref{eq:parabolic} on $(0,T) \times \cU$ with $\cU \subset \R^n$ open set, in the functional setting
  \begin{equation*}
    L^\infty_t((0,T);L^2_x(\cU)) \cap L^2_t(0,T);H^1_x(\cU)),
  \end{equation*}
  and with $A : [0,T] \times \cU \to \mathbb R^{2n}$ a symmetric measurable  matrix-valued function so that
  \begin{equation*}
    \Lambda^{-1} \le A(t,x) \le \Lambda \ \text{ for almost every } t \in [0,T] \text{ and } x \in \cU.
  \end{equation*}

  Then $f \in C^\alpha_{\mathrm{loc}}((0,T) \times \cU)$.

  Moreover, given $\delta>0$ and $0 < \delta< T_0 < T_1 <T-\delta$, there is $C>0$, depending only $\delta$, $\Lambda$ and the dimension $n$, so that for any $\tilde{\cU} \subset \subset \cU$ with $\text{{\rm dist}}(\tilde{\cU},\p\cU) \ge \delta$, we have 
  \begin{equation*}
    \|f\|_{C^\alpha((T_0,T_1) \times \tilde{\cU})} \le C \|f\|_{L^2((0,T) \times \cU)}.
  \end{equation*}
\end{theorem}

\subsection{Modification of the scaling and cylinders}

The first modification is required by the scaling of the equation. Because one derivative in time is equated with two derivatives in space, the invariant scaling is different on both variables, namely
\begin{equation*}
  (t,x) \mapsto (r^2t,rx).
\end{equation*}
The different powers on the two variables $t$ and $x$ reflect the orders of the partial derivatives in $t$ and $x$ in the equation. 

Combined with the fact that parabolic regularisation is forward in time, the base \textbf{parabolic cylinders} that replace the balls are
\begin{equation}
  \label{eq:base-cyl-para}
  Q_r := (-r^2,0] \times B_r.
\end{equation}

\subsection{The parabolic first De Giorgi lemma}

The statement of the parabolic version of the first De Giorgi lemma remains quite similar to the elliptic case. Note however the additional assumption $\p_t f \in L^2_t((-1,0);H^{-1}_x(B_1))$ made on the weak subsolutions. This ensures that one can easily performs approximation procedures and a priori estimates. Such regularity on $\p_t f$ is always satisfied by weak solutions by combining the equality $\p_t f = \nabla_x \cdot (A \nabla_x f)$ in the sense of distributions with the estimate $A \nabla_x f \in L^2$ provided by the parabolic Caccioppoli inequality~\eqref{eq:caccio-para} (the energy estimate) and the pointwise bounds on $A$. It is however not always satisfied by subsolutions due to the presence of a defect measure. Such a subsolution $f$ indeed merely satisfies
\begin{equation*}
  \p_t f = \nabla_x \cdot (A \nabla_x f) - m
\end{equation*}
for a non-negative defect measure $m$. In general, this measure $m$ will not belong to $L^2_t((-1,0);H^{-1}_x(B_1))$.

\begin{lemma}[Parabolic first De Giorgi Lemma]
  \label{lem:para-first}
  Given $\Lambda>0$, there is $\delta>0$, depending only on $\Lambda$ and the dimension $n$, so that the following holds:
  \smallskip

  Let $f$ be a weak subsolution in $Q_1$ to
  \begin{equation}
    \label{eq:para-sub}
    \p_t f \le \nabla_x \cdot \( A \nabla_x f \)
  \end{equation}
  in the functional setting
  \begin{align*}
    & f \in L^\infty_t((-1,0);L^2_x(B_1)) \cap L^2_t((-1,0);H^1_x(B_1)) \ \text{ and } \\[1mm]
    & \p_t f \in L^2_t((-1,0);H^{-1}_x(B_1)),
  \end{align*}
  with $A$ symmetric measurable so that
  \begin{equation*}
    \Lambda^{-1} \le A(t,x) \le \Lambda \ \text{ for almost every } (t,x) \in Q_1.
  \end{equation*}

  Then, we have 
  \begin{equation*}
    \left\| f_+ \right\|_{L^2(Q_1)} \le \delta \quad \implies \quad \left\| f_+ \right\|_{L^\infty(Q_{1/2})}  \le \frac12.
  \end{equation*}
  This estimate is translation and scaling-invariant, provided one uses the appropriate parabolic scaling $(t,x) \mapsto (r^2t,rx)$.
\end{lemma}

The setting of the parabolic first De Giorgi lemma is represented in Figure~\ref{fig:first-para}.
\begin{figure}
  \includegraphics[scale=0.8]{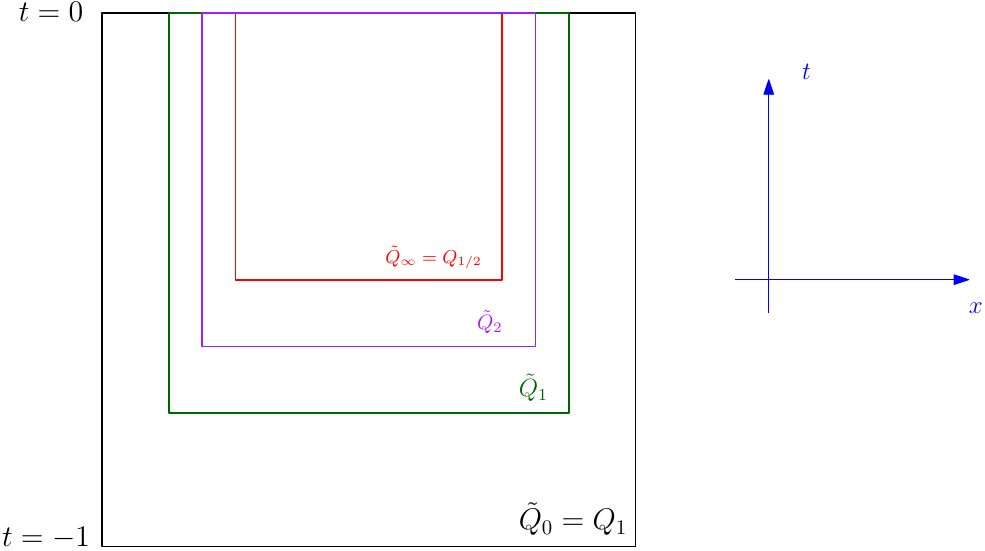}
  \caption{Setting of the first De Giorgi lemma in the parabolic case.}
  \label{fig:first-para}
\end{figure}

\begin{proof}[Proof of Lemma~\ref{lem:para-first}] The De Giorgi--Moser iterative scheme can be easily adapted. For instance, in the De Giorgi viewpoint, the sequence of cylinders and energy levels is 
\begin{equation*}
  \begin{cases}
    &T_k := -\frac{1}{4} - \frac{1}{2^{k+1}}, \quad T_0 =-1, \quad T_k \uparrow -\frac{1}{4}, \\[2mm]
    &r_k = \frac{1}{2} + \frac{1}{2^{k+1}},\quad r_0 =1, \quad r_k \downarrow \frac{1}{2}, \\[2mm]
    &e_k = \frac{1}{2} - \frac{1}{2^{k+1}},  \quad e_0=0, \quad e_k \uparrow \frac{1}{2}, \\[2mm]
    & \tilde Q_k := (T_k,0) \times B_{r_k}, \\[2mm]
    & \displaystyle E_k := \int_{\tilde Q_k} (f-e_k)_+^2 \dd t \dd x.
  \end{cases}
\end{equation*}
Note that the cylinders $\tilde Q_k$ are not exactly of the form~\eqref{eq:base-cyl-para}; instead they are constructed so as to converge to $Q_{1/2}$. 

One establishes again a nonlinear inequality between successive $E_k$ terms, and the proof is based on a localised energy estimate, Sobolev's embedding and Chebyshev's inequality, as in the elliptic case. Only the localised energy estimate and the gain of integrability $L^2 \to L^{p>2}$ in all variables differ from the elliptic proof. We focus only on these two substeps.
\medskip

\noindent
\textbf{A. Localised energy estimate.} We write the energy estimate in a way that is suitable for both the De Giorgi and Moser approaches to the iteration. Consider an energy level $e \ge 0$; it is used in De Giorgi's iteration, but fixed to zero in Moser's iteration). Consider a smooth localisation function $\phi$ with support in $(-1,0] \times B_1$: note that $\phi$ is zero at initial time but not necessarily at final time. We integrate~\eqref{eq:para-sub} against $(f-e)_+ \phi^2$ on $(-1,\tau] \times B_1$ for $\tau \in (-1,0]$:
  \begin{align*}
    & \int_{(-1,\tau] \times B_1} \left( \p_t (f-e)_+ \right) (f-e)_+ \phi^2 \dd t \dd x
    \\  & = - \int_{(-1,\tau] \times B_1} \left(A \nabla_x (f-e)_+ \cdot \nabla_x (f-e)_+ \right) \phi^2 \dd t \dd x\\
    & \, \quad - 2 \int_{(-1,\tau] \times B_1} \left(A \nabla_x (f-e)_+ \cdot \nabla_x \phi \right) \phi (f-e)_+ \dd t \dd x.
  \end{align*}
  The left hand side writes
  \begin{align*}
    & \int_{(-1,\tau] \times B_1} \left( \p_t (f-e)_+ \right) (f-e)_+ \phi^2 \dd t \dd x \\
    & = \frac12 \int_{(-1,\tau] \times B_1} \partial_t (f-e)^2_+  \phi^2 \dd t \dd x \\
    & = \frac12 \left( \int_{B_1} (f-e)^2_+ \phi^2 \dd x \right)\Bigg|_{t=\tau} - \frac12  \int_{(-1,\tau] \times B_1} (f-e)^2_+ \p_t \left( \phi^2 \right) \dd t \dd x,
  \end{align*}
  and we bound the last term as
  \begin{equation*}
    \left| \frac12 \int_{(-1,\tau] \times B_1} (f-e)^2_+ \p_t \left( \phi^2 \right) \dd t \dd x \right|
     \lesssim \| \p_t \phi \|_\infty \|\phi \|_\infty \int_{Q_1 \cap \, \mathrm{supp} \, \phi} (f-e)^2_+ \dd t \dd x.
  \end{equation*}
  Therefore we deduce
  \begin{align*}
    & \frac12 \left( \int_{B_1} (f-e)^2_+ \phi^2 \dd x \right)\Bigg|_{t=\tau} + \int_{(-1,\tau] \times B_1} \left| \nabla_x (f-e)_+ \right|^2 \phi^2 \dd t \dd x \\
    & \lesssim \Lambda \| \p_t \phi \|_\infty \|\phi \|_\infty \int_{Q_1 \cap \, \mathrm{supp} \, \phi} (f-e)^2_+ \dd t \dd x\\
    & \hspace{1cm} + \Lambda \int_{(-1,\tau] \times B_1} \left|A \nabla_x (f-e)_+ \cdot \nabla_x \phi \right| \phi (f-e)_+ \dd t \dd x \\
    & \lesssim \Lambda \| \p_t \phi \|_\infty \|\phi \|_\infty \int_{Q_1 \cap \, \mathrm{supp} \, \phi} (f-e)^2_+  \dd t \dd x \\
    & \hspace{1cm}+ \Lambda^2 \left( \int_{(-1,\tau] \times B_1}  \left|\nabla_x (f-e)_+ \right|^2 \phi \dd t \dd x \right)^{\frac12}  \left( \int_{(-1,\tau] \times B_1} \left|\nabla_x \phi \right|^2 (f-e)_+^2 \dd t \dd x \right)^{\frac12}.
  \end{align*}
  By splitting the last term into two squares, we get
  \begin{multline*}
    \left( \int_{B_1} (f-e)^2_+ \phi^2 \dd x \right)\Bigg|_{t=\tau} + \int_{(-1,\tau] \times B_1} \left| \nabla_x (f-e)_+ \right|^2 \phi^2 \dd t \dd x \\
    \lesssim \Lambda^4 \left( 1+ \|\nabla_x\phi\|_\infty + \|\p_t \phi\|_\infty \right) \, \int_{Q_1 \cap \, \mathrm{supp} \, \phi} (f-e)^2_+ \dd t \dd x.
  \end{multline*}
  By varying the final time $\tau \in (-1,0]$, we finally obtain the following \textbf{parabolic Caccioppoli inequality}, which is the parabolic counterpart to~\eqref{eq:ee2}: 
  \begin{multline}
    \label{eq:caccio-para}
    \frac12 \sup_{-1 < \tau \le 0} \left( \int_{B_1} (f-e)^2_+ \phi^2 \dd x \right)\Bigg|_{t=\tau} + \int_{Q_1} \left| \nabla_x (f-e)_+ \right|^2 \phi^2 \dd t \dd x \\
    \lesssim \Lambda^4 \left( 1+ \|\nabla_x\phi\|_\infty + \|\p_t \phi\|_\infty \right) \, \int_{Q_1 \cap \, \mathrm{supp} \, \phi} (f-e)^2_+ \dd t \dd x.
  \end{multline}
  
  Given $\phi_k$ a sequence of smooth localisation functions such that \begin{equation*}
    \mathbf 1_{\tilde Q_{k+1}} \le \phi_k \le \mathbf 1_{\tilde Q_k} \quad \text{ and } \quad |\nabla_{t,x} \varphi_k| \lesssim 2^k,
  \end{equation*}
  the sequence of localised energy estimate is therefore
  \begin{equation*}
    \sup_{T_{k+1} \le \tau \le 0} \left( \int_{B_1} \phi^2_{k+1} (f-e_{k+1})_+^2 \dd x \right)\Bigg|_{t=\tau} + \int_{Q_1} |\nabla_x (\phi_{k} (f-e_{k+1})_+)|^2 \dd t \dd x \lesssim \Lambda^4 2^k E_k.
  \end{equation*}
  \medskip
  
  \noindent
  \textbf{B. Parabolic Sobolev embedding.} 
  The second term of the left-hand side of the last equation yields some Sobolev regularity in $x$, which leads to a gain of integrability in this variable by the standard Sobolev embedding in $x$ with $t$ fixed. The gain of integrability in $t$ is ensured by the first term on the left-hand side. By combining the two estimates on the $L^2_t L^{p>2}_x$ norm and the $L^\infty_t L^2_x$ norm, one controls the solution in $L^{p_*}_{t,x}$ with $p_*>2$.
\end{proof}

\subsection{The improved parabolic first De Giorgi lemma}

Similarly to the elliptic case, we can improve the first De Giorgi lemma as follows.

\begin{lemma}[Improved parabolic first De Giorgi Lemma]
  \label{lem:dg2+}
  Given $\Lambda>0$ and $\zeta_0 \in (0,2]$, there is $C>0$, depending only on $\Lambda$, $\zeta_0$ and the dimension $n$, so that the following holds:
  \smallskip

  Let $f$ be a weak subsolution in the parabolic cylinder $Q_R$ to
  \begin{equation*}
    \p_t f \le \nabla_x \cdot \( A \nabla_x f \)
  \end{equation*}
  in the functional setting
  \begin{align*}
    & f \in L^\infty_t((-R^2,0);L^2_x(B_R)) \cap L^2_t((-R^2,0);H^1_x(B_R)) \ \text{ and } \\[1mm]
    & \p_t f \in L^2_t((-R^2,0);H^{-1}_x(B_R)),
  \end{align*}
  with $A$ symmetric measurable so that
  \begin{equation*}
    \Lambda^{-1} \le A(t,x) \le \Lambda \ \text{ for almost every } (t,x) \in Q_R.
  \end{equation*}

  Then for any $\zeta \in (\zeta_0,2]$ and $0<r<R$, we have 
  \begin{equation*}
    \|f\|_{L^\infty(Q_r)} \le C (R-r)^{-\frac{2+n}{\zeta}} \left( \frac{1}{|Q_R|} \int_{Q_R} f^\zeta \dd t \dd x \right)^{\frac{1}{\zeta}}.
  \end{equation*}
\end{lemma}

\begin{proof}
  By using the scaling of the equation in the parabolic first De Giorgi lemma, one can prove the following inequality between two general cylinders $Q_r$ and $Q_R$:
  \begin{equation*}
     \|f\|_{L^\infty(B_{r})} \lesssim (R-r)^{-1-\frac{n}{2}} \left( \frac{1}{|Q_R|} \int_{Q_R} f^2 \dd t \dd x \right)^{\frac{1}{2}}.
  \end{equation*}
  Together, the Young inequality and the inequality
  \begin{equation*}
    \left( \frac{1}{|Q_R|} \int_{Q_R} f^2 \dd t \dd x \right)^{\frac{1}{2}} \lesssim \left( \frac{1}{|Q_R|} \int_{Q_R} f^\zeta \dd t \dd x \right)^{\frac12} \|f\|^{1-\frac{\zeta}{2}}_{L^\infty(Q_R)}
  \end{equation*}
  imply that the function $\mathcal N(r):=\|f\|_{L^\infty(Q_r)}$ verifies the following inequality for some $C>0$ and any $0<r<r'<R$:
  \begin{equation*}
    \mathcal N(r) \le \frac{1}{2} \mathcal N(r') + C  (r'-r)^{-\frac{2+n}{\zeta}} \left( \frac{1}{|Q_R|} \int_{Q_R} f^\zeta \dd t \dd x \right)^{\frac1\zeta}.
  \end{equation*}
  
  By the same iteration along the sequence of radii
  \begin{equation*}
    r_{k+1} = r_k + \var \(1-\alpha^{-1}\) \alpha^{-k}
  \end{equation*}
  for $k \ge 1$, some $\alpha>1$, $\var:=R-r$ and $r_0:=r$, we deduce as in the elliptic case that
  \begin{equation*}
    \mathcal N(r) \lesssim_{\zeta_0} \var^{-\frac{2+n}{\zeta}} \left( \frac{1}{|Q_R|} \int_{Q_R} f^\zeta \dd t \dd x \right)^{\frac1\zeta},
  \end{equation*}
  from which the conclusion follows.
\end{proof}

\subsection{The parabolic second De Giorgi lemma}

The parabolic second De Giorgi lemma (reduction of oscillation) resembles the elliptic case, and it implies the Theorem~\ref{thm:dgnm3} by an argument also similar to the elliptic case, but accounting for the slightly different geometry of the parabolic cylinders. 

More generally, the overall structure of the proof of the parabolic second De Giorgi lemma is similar to the elliptic case~\eqref{eq:structure-lemmas}, with parabolic versions of each step. 

\begin{lemma}[Parabolic second De Giorgi Lemma: reduction of oscillation]\label{lem:dg2-para}
  Given $\Lambda>0$, there is $\nu \in (0,1)$, depending only on $\Lambda$ and the dimension $n$, so that the following holds:
  \smallskip
  
  Let $f$ be a weak solution in $Q_3$ to
  \begin{equation}
    \label{eq:EL4-para}
    \p_t f = \nabla_x \cdot ( A \nabla_x f )
  \end{equation}
  in the functional setting $L^2((-9,0];H^1(B_3))$, where $A$ is symmetric measurable and so that
  \begin{equation*}
    \Lambda^{-1} \le A(t,x) \le \Lambda \ \text{ for almost every } (t,x) \in Q_3.
  \end{equation*}
  Then, we have 
  \begin{equation}
    \label{eq:dg2-para}
    \mathrm{osc}_{Q_{\frac12}} f \le \nu \, \mathrm{osc}_{Q_2} f,
  \end{equation}
  where we denote $\mathrm{osc}_B f := \sup_B f - \inf_B f$.
\end{lemma}

Note again that this lemma requires $f$ to be \textbf{both} a subsolution and a supersolution, i.e. to be a solution. Let us show that Lemma~\ref{lem:dg2-para} implies the Hölder regularity, which thus concludes the proof of Theorem~\ref{thm:dgnm3}.
\begin{corollary}
  Given $\Lambda>0$, there is $\alpha \in (0,1)$, depending only on $\Lambda$ and the dimension $n$, so that the following holds:
  \smallskip
  
  Let $f$ be a weak solution in $L^2((0,T) ; H^1(\cU))$, with $\cU \subset \R^n$ an open set, to
  \begin{equation*}
    \p_t f = \nabla_x \cdot \( A \nabla_x f \),
  \end{equation*}
  where $A$ is symmetric measurable and so that
  \begin{equation*}
    \Lambda^{-1} \le A(t,x) \le \Lambda \ \text{ for almost every } (t,x) \in (0,T) \times \cU.
  \end{equation*}

  Then, we have $f \in C^\alpha_{\mathrm{loc}}((0,T) \times \cU)$.
  \smallskip
  
  Moreover, given any $\delta \in (0,T/2)$,  there is $C>0$, depending on $\Lambda$, $\delta$ and the dimension $n$, so that for any $\tilde{\cU} \subset \subset \cU$ with $\text{{\rm dist}}(\tilde \cU,\p \cU)\ge \delta$ and $\delta < T_0 < T_1 < T-\delta$, we have 
  \begin{equation}
    \label{eq:gain-reg-loc-para}
    \|f\|_{C^\alpha((T_0,T_1) \times \tilde{\cU})} \le C \|f\|_{L^2(\cU)}.
  \end{equation}
\end{corollary}

\begin{proof}
  We already have that $f \in \mathrm{L}^\infty_{\mathrm{loc}}(\cU)$ by Lemma~\ref{lem:para-first} with the estimate
  \begin{equation}
    \label{eq:gain-localised-para}
    \forall \, \hat{\cU} \subset \subset \cU, \quad \|f\|_{L^\infty((\delta/2,T-\delta/2) \times \hat{\cU})} \le C \|f\|_{L^2((0,T) \times \cU)}
  \end{equation}
  for a constant $C>0$ depending on $\Lambda$, $\text{dist}(\hat \cU,\p \cU)$, $\delta$ and the dimension $n$.
  
  Let $\mathcal{\tilde{U}} \subset\subset \mathcal{\hat U} \subset \subset \cU$ so that
  \begin{equation}
    \label{eq:lien-dist-para}
    \text{dist}(\mathcal{\tilde{U}},\partial\cU) \le \frac12 \text{dist}(\mathcal{\tilde{U}},\partial\mathcal{\hat U}), \quad \text{dist}(\mathcal{\tilde{U}},\partial\cU) \in \left[\frac{\delta}{4},\frac{\delta}{2}\right], \quad \text{dist}(\mathcal{\tilde{U}},\partial\mathcal{\hat U}) \in \left[ \frac{\delta}{2},\delta\right].
  \end{equation}
  Given $(t_0,x_0) \in (T_0,T_1) \times \mathcal{\tilde{U}}$, we define the following sequence of centred rescaled functions
  \begin{align*}
    & \forall \,  k \geq 1, \ \forall \, (s,y) \in Q_1, \quad \tilde g_k(s,y) := g_{k-1}\(\frac{s}{16},\frac{y}{4}\) \\
    & \forall \, (s,y) \in Q_1, \quad g_0(s,y) := f\(t_0+\frac{\delta}{32} + \frac{\delta s}{16},x_0 +\frac{\delta y}{4}\).
  \end{align*}

  Note that each function $g_k$ is a rescaling of $f$ and solves
  \begin{equation*}
    \p_t g_k = \nabla_x \cdot [A_k \nabla_x g_k]
  \end{equation*}
  in $Q_2$, with the diffusion matrix
  \begin{equation*}
    A_k(s,y) := A\(t_0+\frac{\delta}{32} + \frac{\delta s}{2^{4k}},x_0 + \frac{\delta y}{2^{2k}}\).
  \end{equation*}

  This new diffusion matrix is still measurable and satisfies
  \begin{equation*}
    \Lambda^{-1} \le A_k(s,y) \le \Lambda \ \text{ for almost every } (s,y) \in Q_2.
  \end{equation*}

  Therefore, we can apply Lemma~\ref{lem:dg2} to each $g_k$ and deduce
  \begin{equation*}
    \forall \,  k \geq 1, \quad \mathrm{osc}_{Q_{\frac12}} g_k \, \le \, \nu \, \mathrm{osc}_{Q_2} g_k.
  \end{equation*}
  The scaling relation between successive terms in the sequence $g_k$ yields
  \begin{equation*}
    \forall \,  k \geq 1, \quad \mathrm{osc}_{Q_2} \, g_{k+1} = \mathrm{osc}_{Q_{\frac12}} g_k,
  \end{equation*}
  so we obtain by induction
  \begin{equation*}
    \mathrm{osc}_{Q_{4^{-k}}} g_1 \le \nu^{k-1} \mathrm{osc}_{Q_2} g_1 \le 2 \, \nu^{k-1} \, \|f\|_{L^\infty((T_0/2) \times \hat \cU)}.
  \end{equation*}

  Then, given a point
  \begin{equation*}
    (t,x) \in \(t_0- \frac{\delta}{16},t_0+ \frac{\delta}{16}\) \times B\(x_0, \frac{\delta}{4}\),
  \end{equation*}
  we pick up $k \ge 0$ so that
  \begin{equation*}
    |t-t_0| \in \left[\delta 16^{-k-2},\delta 16^{-k-1}\right], \quad |x-x_0| \in \left[\delta 4^{-k-2},\delta 4^{-k-1}\right],
  \end{equation*}
  and write 
  \begin{equation*}
    |f(t,x) - f(t_0,x_0)|\le 2 \, \nu^{k-1} \, \|f\|_{L^\infty(\hat \cU)},
  \end{equation*}
  so that (using $|t-t_0| \ge \delta 16^{-k-2}$ and $|x-x_0| \ge \delta 4^{-k-2}$)
  \begin{equation*}
    |f(t,x) - f(t_0,x_0)| \lesssim \left[ \( 16^\alpha \nu \)^k |t-t_0|^{\alpha} +\(4^{2\alpha} \nu\)^k \, |x-x_0|^{2\alpha} \right] \, \|f\|_{L^\infty((\delta/2,T-\delta/2) \times \hat \cU)}.
  \end{equation*}
  We then choose
  \begin{equation*}
    \alpha := - \frac{\ln \nu}{4\, \ln 2}
  \end{equation*}
  so that $16^\alpha \nu = 4^{2\alpha} \nu =1$, and deduce $f$ is $C^\alpha$ at $(t_0,x_0)$. (In fact, note that the Hölder regularity exponent is $\alpha$ in $t$ and $2\alpha$ in $x$, which reflects the order of the partial derivatives in each coordinate in the equation). The $\alpha$ does \emph{not} depend on $(t_0,x_0)$, but only on the constant $\nu$ of Lemma~\ref{lem:dg2}.

  We have therefore proved
  \begin{equation*}
    \|f\|_{C^\alpha((\delta,T-\delta) \times \tilde{\cU})} \le C \|f\|_{L^\infty((\delta/2,T-\delta/2) \times \hat \cU)}
  \end{equation*}
  with a constant depending only on $\nu$ and $d_0$. Combined with~\eqref{eq:gain-localised-para} and~\eqref{eq:lien-dist-para}, it implies~\eqref{eq:gain-reg-loc-para} and concludes the proof.
\end{proof}

We shall show that Lemma~\ref{lem:dg2-para} is implied by the following lemma.

\begin{lemma}[Parabolic decrease of supremum Lemma]
  \label{lem:dec-sup-para}
  Given $\Lambda>0$, there is $\lambda \in (0,1)$, depending only on $\Lambda$ and the dimension $n$, so that the following holds:
  \smallskip

  Let $f$ be a weak subsolution in $Q_2$ to
  \begin{equation*}
    \p_t f \le \nabla_x \cdot ( A \nabla_x f)
  \end{equation*}
  in the functional setting
  \begin{align*}
    & f \in L^\infty((-4,0);L^2(B_2)) \cap L^2((-4,0);H^1(B_2)) \\[1mm]
    & \p_t f \in L^2((-4,0);H^{-1}(B_2)),
  \end{align*}
  with $A$ symmetric measurable so that
  \begin{equation*}
    \Lambda^{-1} \le A(t,x) \le \Lambda \ \text{ for almost every } t \in [-4,0] \text{ and } x \in B_2,
  \end{equation*}
  and such that $f \le 1$ and, for some $\mu>0$, 
  \begin{equation*}
    \quad |\tilde Q \cap \{f \le 0\}| \ge \mu |\tilde Q| \quad \text{ with } \quad \tilde Q := (-3/2,-1) \times B_1.
  \end{equation*}

  Then, we have 
  \begin{equation*}
    f \le 1 -\lambda \ \text{ on } \ Q_{1/2} = (-1/2,0) \times B_{1/2}.
\end{equation*}
\end{lemma}

\begin{proof}[Proof of Lemma~\ref{lem:dg2-para} when assuming Lemma \ref{lem:dec-sup-para}]
  Take $f$ a weak solution to~\eqref{eq:EL4-para} in $Q_3$. Assume $\mathrm{osc}_{B_2} f > 0,$ otherwise the statement is trivial. Then, can perform a rescaling similar to~\eqref{eq:osc-resc} and define the new unknown
  \begin{equation*}
    g(x) := \frac{2}{\mathrm{osc}_{B_2} f}  \left[ f(x) - \frac{\sup_{B_2} \, f + \inf_{B_2} \, f}{2} \right].
  \end{equation*}
  This new unknown $g$ is a solution to~\eqref{eq:EL4-para}, is valued in $[-1,1]$, and satisfies
  \begin{equation*}
    \mathrm{osc}_{B_2} \, g = 2.
  \end{equation*}

  Then, certainly,
  \begin{equation*}
    \text{ either } \ |Q_1 \cap \{g \le 0\}| \ge \frac12 \, |Q_1| \quad \text{ or } \quad |Q_1 \cap \{g \ge 0\}| \ge \frac12 \, |Q_1|.
  \end{equation*}
  
  In the first case, we apply Lemma~\ref{lem:dec-sup-para} to the subsolution $g$, with $\mu := 1/2$:
  \begin{equation*}
    \mathrm{osc}_{B_{\frac12}} g \le (2-\lambda) \le \(1 - \frac{\lambda}{2}\) \, \mathrm{osc}_{B_2} g
  \end{equation*}
  for some $\lambda \in (0,1)$. Going back to $f$, this means 
  \begin{equation*}
    \mathrm{osc}_{B_{\frac12}} f \le \(1- \frac{\lambda}{2}\) \, \mathrm{osc}_{B_2} f,
  \end{equation*}
  yielding Lemma~\ref{lem:dg2-para} with $\nu = 1-\lambda/2$. In the second case, we perform the same argument on the subsolution $-g$. Note that we are therefore using that both $f$ and $-f$ are subsolutions, i.e. that  $f$ is a solution.
\end{proof}

In turn, the parabolic decrease of supremum lemma~\ref{lem:dec-sup-para} is itself implied by a parabolic version of the De Giorgi intermediate value lemma, which we now state. The proof is exactly similar to the elliptic case and is not repeated.

\begin{lemma}[Parabolic intermediate value Lemma]
  \label{lem:para-ivl}
  Given $\Lambda>0$, $\mu>0$ and $\delta>0$, there is $\alpha>0$, depending only on $\Lambda$, $\mu$, $\delta$ and the dimension $n$, so that the following holds:
  \smallskip
  
  Let $f$ be a weak subsolution in $Q_2$ to
  \begin{equation*}
    \p_t f \le \nabla_x \cdot ( A \nabla_x f)
  \end{equation*}
  in the functional setting
  \begin{align*}
  & f \in L^\infty((-4,0);L^2(B_2)) \cap L^2((-4,0);H^1(B_2)) \\
  & \p_t f \in L^2((-4,0);H^{-1}(B_2)),
  \end{align*}
  with $A$ symmetric measurable so that
  \begin{equation*}
    \Lambda^{-1} \le A(t,x) \le \Lambda  \ \text{ for almost every } t \in [-4,0] \text{ and } x \in B_2,
  \end{equation*}
  and such that $f \le 1$ and, for some $\mu, \delta >0$,
  \begin{align*}
    & |\{f \ge 1/2\} \cap Q_1| \ge \delta |Q_1|, \\[1mm]
    & |\{f\le 0\} \cap \tilde Q| \ge \mu |\tilde Q| \\[1mm]
    & \text{ where } \tilde Q := (-3/2,-1) \times B_1.
  \end{align*}

  Then, we have
  \begin{equation*}
    |\{0<f<1/2\} \cap (Q_1 \cup \tilde Q)| \ge \alpha |Q_1 \cup \tilde Q|.
  \end{equation*}
\end{lemma}

The setting of the parabolic decrease of maximum lemma and the parabolic intermediate value lemma is represented in Figure~\ref{fig:para-ivl}. 
\begin{figure}
  \includegraphics[scale=0.65]{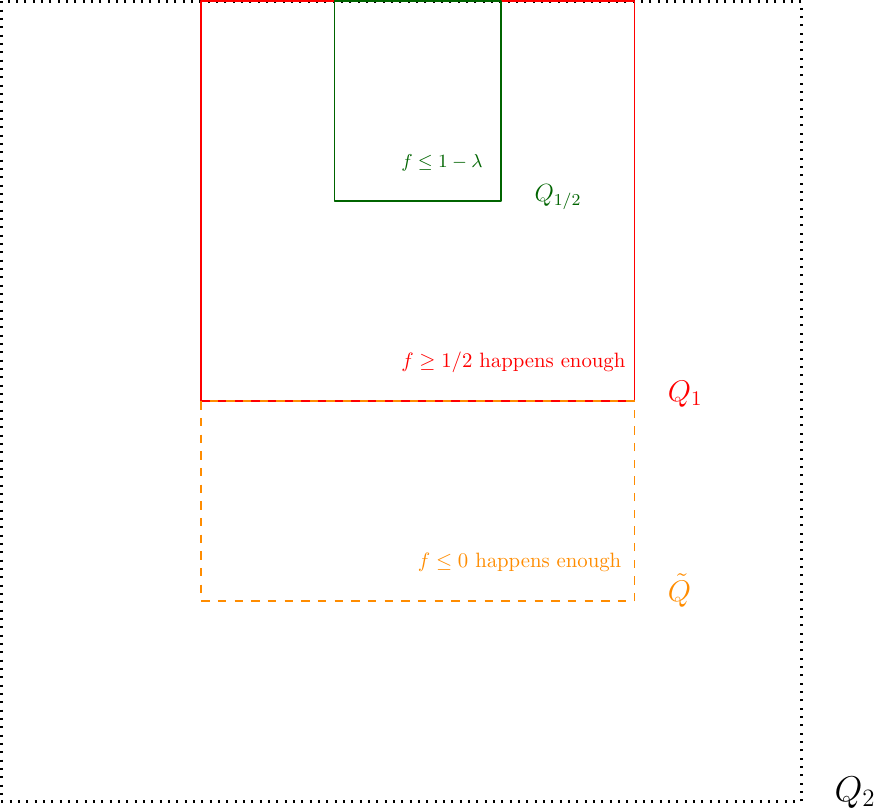}
  \caption{Setting of the second De Giorgi lemma in the parabolic case.}
  \label{fig:para-ivl}
\end{figure}

We first briefly sketch the extension of the non-constructive argument by compactness and contradiction to the parabolic setting, and then turn to the constructive argument based on trajectories. 

\begin{proof}[Proof of Lemma~\ref{lem:para-ivl} by non-constructive compactness method]
  \text{ } \\
  Consider a sequence $f_k$ a of weak subsolutions in $Q_2$ so that $f_k \le 1$ and
  \begin{align*}
    & |\{f_k \ge 1/2\} \cap Q_1| \ge \delta|Q_1| > 0, \\[2mm]
    & |\{f_k\le 0\} \cap (-3/2,-1) \times B_1| \ge \mu |(-3/2,-1) \times B_1|, \\[2mm]
    & |\{0<f_k<1/2\} \cap (Q_1 \cup \tilde Q)| \to 0.
  \end{align*}
  
  Then $(f_k)_+$ is valued in $[0,1]$ so uniformly bounded in $L^2(Q_2)$, and the energy estimate implies it is uniformly bounded in $L^2_tH^1_x (Q_1 \cup \tilde Q)$.
  
  Moreover by integrating the inequation against a localisation function (\emph{without} squaring the unknown) one can establish a uniform bound of $\p_t (f_k)_+$ in
  \begin{equation*}
    M^1(Q_1 \cup \tilde Q) \cap L^\infty_t W^{-1,1}_x(Q_1 \cup \tilde Q),
  \end{equation*}
  where $M^1$ is the space of non-negative measure normed by the total variation. These bounds are sufficient to apply the Aubin-Lions lemma~\cite{MR152860,MR259693} (or equivalent compactness results) and deduce that, for possibly a subsequence that we still denote as the original sequence, $(f_k)_+ \to g$ strongly in $L^2_{t,x}(Q_1 \cup \tilde Q)$, almost everywhere, and weakly in $L^2_tH^1_x (Q_1 \cup \tilde Q)$.

  These convergences are now strong enough to deduce that the limit satisfies
  \begin{align}
    \label{eq:lim-zero}
    & |\{0<g<1/2\} \cap (Q_1 \cup \tilde Q)|=0, \\[2mm]
  \label{eq:lim-pos}
    & |\{g \ge 1/2\} \cap Q_1| \ge \delta|Q_1| > 0, \\[2mm]
    \label{eq:lim-pos2}
    & |\{g\le 0\} \cap \tilde Q| \ge \mu |\tilde Q|.
  \end{align}
  And the weak convergence in $L^2_t H^1_x$ implies that for almost every $t \in (-3/2,0)$:
  \begin{equation*}
    \nabla_x g(t,\cdot) \in L^2(B_1).
  \end{equation*}

  The $H^1(B_1)$ bound in $x$ for almost every $t$, combined with~\eqref{eq:lim-zero}, provide us with the assumptions to apply the elliptic intermediate value lemma~\ref{lemma5} for almost every $t$. It implies, for almost every $t$, that either $g(t,\cdot) = 0$ almost everywhere on $B_1$, or  $g(t,\cdot) \ge 1/2$ almost everywhere on $B_1$. But the lower bound~\eqref{eq:lim-pos2} implies on the one hand that there are $t \in (-3/2,-1)$ such that $g(t,\cdot) = 0$ almost everywhere on $B_1$. On the other hand, the energy estimate shows that the function $t \mapsto \|g(t,\cdot)\|_{\mathrm{L}^2(B_1)}$ cannot jump upwards on $(-3/2,0)$. Therefore, $g(t,x) = 0$ almost everywhere on $(0,1) \times B_1$ which contradicts~\eqref{eq:lim-pos}.
\end{proof}

\begin{remarks}
  \begin{enumerate}
  \item Note that the parabolic intermediate value lemma now assumes $f$ to be a subsolution to the PDE, not simply to be in a certain Sobolev space. It could in fact be traded for the infinite family of energy estimates, which would be sufficient to prove the parabolic first De Giorgi lemma, the parabolic decrease of supremum lemma and the parabolic intermediate value lemma. Functions satisfying such infinite family of energy estimates are said to belong to a \textbf{De Giorgi class}. In the kinetic case, no equivalent of the De Giorgi classes has been identified yet, and the kinetic counterparts of these results are proved only for subsolutions to the PDE itself.
  \item As in the elliptic setting, the cutoff value $1/2$ in the intermediate value lemma could be replaced by $\theta \in (0,1)$ close to $1$, resulting in a slightly weaker statement that is nevertheless sufficient for proving the reduction of oscillation. 
  \item This contradiction argument was extended to the hypoelliptic case in~\cite{golse2019harnack}, as we discuss in the next section.
  \item Note the time delay between $\tilde Q$ and $Q_{1/2}$ in the parabolic decrease of supremum lemma, but the absence of a time delay between $Q_1$ and $\tilde Q$ in the parabolic intermediate value lemma. In the kinetic setting, a gap is needed \textbf{also} in the intermediate value lemma, as first noted in~\cite{guerand2020quantitativeregularityparabolicgiorgi}. Intuitively, this is because when the diffusive term vanishes, the trajectories are simply vertical in the $(t,x)$ representation in the parabolic, whereas they can be tilted in many ways in the kinetic case: we give a counter-example in the next section.
  \item In both the elliptic and parabolic settings (as well as in the kinetic setting in the next section), variants of the previous results hold in the presence of lower order terms, provided they have sufficient integrability. We do not address this issue to keep these notes short, but lower order terms are treated in the references~\cite{MR4398231,guerand2020quantitativeregularityparabolicgiorgi} in the parabolic case for instance. 
  \item The functional setting of the weak subsolutions we consider could be relaxed to
    \begin{equation*}
      f \in L^1_t((0,T);L^2_x(\cU)) \cap L^2_t(0,T);H^1_x(\cU))
    \end{equation*}
    by careful approximation procedures, see~\cite{auscher2024weaksolutionskolmogorovfokkerplanckequations,dmnz}.
  \end{enumerate}
\end{remarks}

\begin{proof}[Proof of Lemma~\ref{lem:para-ivl} by constructive method based on trajectories]
  \text{ } \\
  Lemma~\ref{lem:para-ivl} is in fact implied by the following ``parabolic Poincar\'e inequality'': 
  \begin{theorem}[Parabolic Poincaré inequality]
  \label{thm31para}
  Given $\Lambda>0$, there is $C>0$, depending only on $\Lambda$ and the dimension $d$, so that the following holds:
  \smallskip
  
  Let $r=1/20$ and $f \ge 0$ be a weak subsolution in $Q_1$ to
  \begin{equation*}
    \p_t f \le \nabla_x \cdot \( A \nabla_x f \)
  \end{equation*}
  with $A$ symmetric measurable so that
  \begin{equation*}
    A(t,x) \le \Lambda \ \text{ for almost every } (t,x) \in Q_1.
  \end{equation*}

  Then $f$ satisfies the following inequality
  \begin{equation}
    \label{gmpoi-para}
    \int_{Q_r ^+} \(f- \langle f\rangle_{Q^-_r} \)_+ \dd t \dd x \le C \int_{Q_1} |\nabla_x f| \dd t \dd x,
  \end{equation}
  where $Q_r^-$ is defined by
  \begin{equation*}
    Q_r^- := \left\{ -3 r^2 < t < -2r^2, \quad |x| < r  \right\},
  \end{equation*}
  and the average $\langle f\rangle_{Q^-_r}$ is defined by
   \begin{equation*}
     \langle f\rangle_{Q^-_r} := \frac{1}{|Q^-_r|} \int_{Q^-_r} f \dd t \dd x.
   \end{equation*}
\end{theorem}

\begin{remarks}
  \begin{enumerate}
  \item This parabolic Poincaré inequality is inspired by the kinetic Poincaré inequality in Theorem~\ref{thm31} in the next section, to which we refer for further discussion. A partial version of this parabolic Poincaré inequality was first obtained in~\cite{MR4398231,guerand2020quantitativeregularityparabolicgiorgi}.
  \item Intuitively this theorem measures how much (in $L^1$) the
    solution $f$ can grow above its past average over $Q_r^-$, in terms of a bound on $\nabla_x f$ in $L^1$ in a bigger cube \emph{in the
      present}, provided that $f$ is a subsolution on such a bigger
    cube. This is therefore a \textbf{time-oriented integral control on the oscillation}. For instance if $\nabla_x f \equiv 0$, then $f$ depends only on $t$ and solves $\p_t f \le 0$ so is non-increasing, and the inequality~\eqref{gmpoi-para} is clear.  
  \item This theorem only requires the bound from above on $A$, not the coercivity. However, the coercivity of $A$ is needed to derive the parabolic intermediate value lemma from the parabolic Poincaré inequality. 
\end{enumerate}
\end{remarks}

\begin{proof}[Proof of Theorem~\ref{thm31para} (the parabolic Poincaré inequality)]
  \text{ } \\
  The full, and more powerful, argument will be presented in the next section in the kinetic case. We only stress the modifications and simplifications in the parabolic case with respect to the kinetic case. The setting with the three cylinders $Q_r^-$, $Q_r$ and $Q^0_r$ between $Q_r^-$ and $Q_r$ is similar to the kinetic case and is explained in the next section (see Figure~\ref{fig:trajectories}). The intermediate cylinder is useful for introducing a smooth localisation function.
  \medskip

  \noindent
  \textbf{A. Construction of the trajectories.} 
  We ues the letter ``$z$'' to denote the joint variable $z=(t,x)$.

  Given three points
  $z_+ \in Q^+, z_0 \in Q^0$ and $z_- \in Q^-$, we want to
  construct two paths $s \rightarrow \Gamma_+(s)$ and
  $s \rightarrow \Gamma_-(s)$ for $s \in [0, 1]$ such that
  \begin{equation*}
    \begin{aligned}
      & \Gamma_+(s) = \left(T_+(s),X_+(s)\right), \qquad \Gamma_+(0) = z_+, \qquad \Gamma_+(1) =
      z_0,\\
      & \Gamma_-(s) = \left(T_-(s), X_-(s)\right), \qquad \Gamma_-(0) = z_-, \qquad \Gamma_-(1) =
      z_0,
    \end{aligned}
  \end{equation*}
  where
  \begin{equation}
    \label{eq:control-para}
    \begin{aligned}
      \begin{cases}
        \displaystyle
        \ds X_{\pm}(s)
        = \mathsf m_\pm  s^{-\frac12} \\[3mm]
        \displaystyle 
        \ds T_\pm(s) = \delta_\pm,
      \end{cases}
    \end{aligned}
  \end{equation}
  where $\delta_\pm := t_0 - t_\pm$ and and
  $\mathsf m _\pm \in \R^d$. Solving the differential equations yields
  \begin{equation}
    \label{eq:control-solved-para}
    \begin{aligned}
      \begin{cases}
        X_\pm(s) = x_\pm + 
        2 \delta_\pm s^{\frac12} \mathsf m_\pm,\\[2mm]
        T_\pm(s) = s t_0 + (1-s) t_\pm.
      \end{cases}
    \end{aligned}
  \end{equation}

  The boundary conditions $\Gamma_\pm(1)=z_0$ impose $\mathsf{m}_\pm = (2\delta_\pm)^{-1} (x_0 - x_\pm)$. We deduce $X_\pm(s) = \Phi_\pm^s(x_0)$ with 
  \begin{align*}
    \nabla_{x_0} \left( \Phi_\pm ^s \right)^{-1}
    \sim O\( s^{-\frac12} \)
  \end{align*}
  which remains integrable. Observe also that our choice of control function is such that its derivative
  \begin{equation*}
    X_\pm'(s) \sim O\(s^{-\frac12}\)
  \end{equation*}
  is integrable on $s \in [0,1]$, which implies that the trajectories
  are bounded (with integrable tangent vector field).
  \medskip

  \noindent
  \textbf{B. Proof of the parabolic Poincaré inequality by following the trajectories.}
  We refer again to the next section for the details. Thanks to the estimates in the last two displays, we can argue exactly as in the proof detailed in the kinetic case, to obtain  
  \begin{multline*}
      \int_{Q^+} \Big(f(z_+) - \langle f\rangle_{Q^-}\Big)_+ \dd
      z_+ \lesssim \int_{Q^+} \int_0^1\int_{Q^0}
     \left| \left( \nabla_x f \right)
       (\Gamma_+(s))\right| s^{-\frac12} \dd s \dd z_0 \dd z_+ \\
     + \int_{Q^-} \int_0^1\int_{Q^0} \left| \left( \nabla_x f \right)
       (\Gamma_-(s))\right| s^{-\frac12} \dd s \dd z_0 \dd z_-,
  \end{multline*}
  and the end of the proof is similar: we are now integrating in both $Q^0$ and $Q_\pm$, so we are now in a position to use a change of variable not only with the intermediate variable $z_0$, but also with the future/past variables $z_\pm$. Note that it was not possible to use the integration
  in $Q_\pm$ before because of the positive value around the $Q^0$
  integral. Using these changes of variables, we deduce
  \begin{equation*}
    \int_{Q^\pm} \int_0^1\int_{Q^0}
    \left| \left( \nabla_x f \right)
      (\Gamma_\pm(s))\right| s^{-\frac12} \dd s \dd z_0 \dd z_\pm \lesssim
    \int_{Q_1} \left| \nabla_x f \right| \dd t \dd x,
  \end{equation*}
  which concludes the proof of the parabolic Poincaré inequality~\eqref{gmpoi-para}.
\end{proof}

The parabolic Poincaré inequality~\eqref{gmpoi-para} then implies the parabolic intermediate value Lemma~\ref{lem:para-ivl}. The proof is exactly similar to the argument in the kinetic case in the next section, i.e. the proof of the kinetic intermediate value Lemma~\ref{lem:kin-ivl} from the kinetic Poincaré inequality~\eqref{gmpoi}. One can apply the kinetic proof to a function independent of the space variable to recover the parabolic proof. (In fact, the argument does not depend on the equation but only on the Poincaré-type inequality and the Caccioppoli inequality.) We therefore refer the reader to the next section for the proof.
\end{proof}

\subsection{Parabolic versions of the Moser and Kruzhkov methods}

The Moser approach was extended to the parabolic case in~\cite{moser1964harnack,moser1967correction,moser1971}. The first paper in this series had a mistake (which was corrected a few years later), which illustrates that the John-Nirenberg result is not a convenient tool in the parabolic case. Instead, in the later paper~\cite{moser1971}, Moser  improved and simplified the argument: in order to ``bridge the gap'' between negative and positive exponents and prove
\begin{equation*}
  \Phi(p,2) \lesssim \Phi(-p,2)
\end{equation*}
for $p>0$ small, Moser combines a universal bound on the $L^2$ norm of the gradient of the logarithm of the solution with a lemma by Bombieri and Giusti~\cite{bg1972}. This method provides a sharper dependency of the constant in the Harnack inequality in terms of the ellipticity constant $\Lambda$. We refer to the recent contribution~\cite{nz_trajMBG_2025} for a re-interpretation of this Moser-Bombieri-Giusti approach based on trajectories in the parabolic case, and a discussion of the dependency of the Harnack constant in terms of $\Lambda$. We also refer to~\cite{MR244638,MR226168} for the extension of Moser's method to some quasilinear parabolic equations. 

The Kruzhkov approach was extended to the parabolic case in~\cite{MR171086}. The argument is quite close to the elliptic case. The main difference is that the right-hand side of~\eqref{eq:krukru}, in the parabolic case, also includes the $L^\infty$ norm of $g$, albeit with a power one: 
\begin{equation*}
  \|\nabla_x g\|^2_{L^{2}(Q_1)} \lesssim \Lambda^4  \( \int_{Q_2} \left| \nabla_x \phi \right|^2 + \int_{Q_2} g^2 \left| \p_t \phi \right|^2 \) \lesssim \Lambda^4 \( 1+ \| g\|_{L^\infty(Q_1)} \).
\end{equation*}
This is due to the error terms produced by the time derivative operator.

The different powers between the left-hand and right-hand sides still allows the argument to work. One gets finally
\begin{equation*}
  \|g\|_{L^\infty(Q_{1/2})}^2 \lesssim_\Lambda 1+ \|g\|_{L^\infty(Q_1)}
\end{equation*}
and we then express the norm on the left-hand side in terms of the infimum on the solution $f$ on $Q_{1/2}$, and we bound the norm on the right-hand side by the trivial upper bound on $G(f_\delta)$.

This leads to the inequality
\begin{equation*}
  \left| \ln \left( \frac{\inf_{Q_{1/2}} f + \delta}{1+\delta} \right) \right|^2 \lesssim_\Lambda 1+ \left| \ln \left( \frac{\delta}{1+\delta} \right) \right|
\end{equation*}
which implies a strictly positive lower bound on $\inf_{Q_{1/2}} f$. 

\section{The kinetic case}

We focus in this section on the simplest class of local hypoelliptic equations of type II (in the terminology of Hörmander): the \textbf{Kolmogorov equation}. The underlying motivation is the study of the Landau equation~\eqref{eq:land}-\eqref{eq:land2}, which, we recall, has the following form
\begin{equation*}
  \p_t f + v \cdot \nabla_x f = \nabla_v \cdot \( A[f] \, \nabla_v f \) + B[f] \cdot \nabla_v f + \( \nabla_v \cdot B[f] \) f
\end{equation*}
for an unknown $f =f(t,x,v) \ge 0$, for $t \in \R$, $(x,v) \in \cU \subset \R^d \times \R^d$, where $A[f]$ and $B[f]$ are linear integral operators on the $v$ variable applied to the unknown $f$.

Following the same approach as De Giorgi and Nash in their solution to Hilbert's $19$th problem, i.e. trading off nonlinearity for the study of linear equations with rough coefficients, we consider instead the following linear equation:
\begin{equation}
  \label{kin1}
  \partial_t f + v \cdot \nabla_x f = \nabla_v \cdot [ A \nabla_v f] + B \cdot \nabla_v f + S,
\end{equation}
on $[0,T]$, with $A=A(t,x,v)$ a measurable symmetric matrix-valued function, $B=B(t,x,v)$ a measurable vector-valued function, $S=S(t,x,v)$ a measurable scalar function. We assume that the coefficients and source term satisfy 
\begin{equation}
  \label{kin:h}
  \left\{
    \begin{matrix}
    & \Lambda^{-1} \le A(t,x,v) \le \Lambda \\[1mm]
    & |B(t,x,v)| \le \Lambda \\[1mm]
    & |S(t,x,v)| \le C_S
  \end{matrix}
  \right\} \quad 
  \text{ for almost every } t \in [0,T] \text{ and } (x,v) \in \cU,
\end{equation}
for some $\Lambda>0$ and $C_S>0$.

The main result reviewed this section is the local Hölder regularity of the weak solutions to~\eqref{kin1}, under the assumptions~\eqref{kin:h}. This, in turn, implies the Hölder regularity of weak a priori solutions to the Landau equation~\eqref{eq:land}, provided we assume these solutions to be bounded in $L^\infty$ in all variables and to satisfy uniform pointwise bounds on their hydrodynamic fields. The details of this corollary are included in~\cite{golse2019harnack}. This is one of the steps in the solution the second open problem of the introduction (conditional regularity for the Landau equation). Because of the structure of the Landau equation, this corollary requires the presence of the lower order terms $B \cdot \nabla_v f$ and $S$ in \eqref{kin1}.

In these notes, we however focus only on the simplest form of~\eqref{kin1} when $B=0$ and $S=0$. This simplifies the setting, the calculations and makes the core ideas more visible. The treatment of these lower order terms merely adds technical complication and is not essential to the proofs. Again, the full details are discussed in~\cite{golse2019harnack}, and we also refer to the recent review~\cite{zbMATH07986618} about the work~\cite{golse2019harnack}. We however emphasize that the $L^\infty$ bound on the a priori solutions to the Landau is made because of the assumptions we require on the source term $S$ in the main Hölder regularity result.

The model problem and main statement of the section are therefore as follows. We consider $\cU \subset \R^d \times \R^d$ an open set, $T>0$, and the equation
\begin{equation}
  \label{eq:hypo}
  \partial_t f + v \cdot \nabla_v f = \nabla_v \cdot \( A \, \nabla_v f \) \quad \text{in} \, \, (0,T) \times \cU
\end{equation}
with $A=A(t,x,v)$. We define the Hölder space $C^\alpha((0,T) \times \cU)$ as in~\eqref{eq:holder} by replacing $\cU$ by $(0,T) \times \cU \subset \R \times \R^d \times \R^d$. The kinetic counterpart of the main Hölder regularity result is the following statement. 

\begin{theorem}[Kinetic De Giorgi--Nash--Moser~\cite{wang2009c,golse2019harnack,guerand2022quantitative,MR4653756,niebel2023kinetic,anceschi2024poincare,dmnz}]
  \label{thm:dgnm4}
  Given $\Lambda>0$, there is $\alpha \in (0,1)$, depending only on $\Lambda$ and the dimension $d$, so that the following holds:
  \smallskip
  
  Let $f$ be a weak solution to~\eqref{eq:hypo} on $[0,T] \times \cU$, with $\cU \subset \R^{2d}$ an open set, in the functional setting
  \begin{equation*}
    L^\infty_t((0,T);L^2_{x,v}(\cU)) \cap L^2_t(0,T);L^2_xH^1_v(\cU)),
  \end{equation*}
  with $A : [0,T] \times \cU \to \mathbb R^{n \times n}$ a symmetric measurable  matrix-valued function so that
  \begin{equation}
    \label{eq:hypoA}
    \Lambda^{-1} \le A(t,x,v) \le \Lambda \ \text{ for almost every } t \in [0,T] \text{ and } (x,v) \in \cU.
  \end{equation}

  Then, we have $f \in C^\alpha_{\mathrm{loc}}((0,T) \times \cU)$.

  Moreover, given $\delta>0$ and $\delta < T_0 < T_1 < T-\delta$, there is $C>0$, depending only on $\Lambda$, $\delta$ and the dimension $d$, so that for any $\tilde{\cU} \subset \subset \cU$ with $\text{{\rm dist}}(\tilde{\cU}, \p \cU) \ge \delta$, we have 
  \begin{equation*}
    \|f\|_{C^\alpha((T_0,T_1) \times \tilde{\cU})} \le C \|f\|_{L^2((0,T) \times \cU)}.
  \end{equation*}
\end{theorem}

\subsection{Modification of the scaling and cylinders}

In the sequel, we write $z:=(t,x,v) \in \mathbb R^{1+2d}$ for the full set of variables, and $\cT := \partial_t + v \cdot \nabla_x$ for the transport operator.

The invariant transformations of the class of equations~\eqref{eq:hypo} that satisfy~\eqref{eq:hypoA} are slightly more intricate. The scaling now must take into account the balance between the three differential operators in the equation, which yields the \textbf{kinetic scaling}
\begin{equation}
  \label{eq:scaling-kin}
  (t,x,v) \mapsto (r^2t,r^3x,rv) \ \text{ for } \ r>0.
\end{equation}
Besides we still have invariance by translation in time and in space, and we have an additional invariance involving the velocity variable, albeit modified by the inertia principle. This results in the \textbf{invariance by Galilean translations}
\begin{equation}
  \label{eq:translation-kin}
  (t,x,v) \mapsto (t_0+t,x_0+x + tv_0,v_0+v) \ \text{ for } \ (t_0,x_0,v_0) \in \R^{1+2d}.
\end{equation}

The scaling invariance naturally leads to defining the following base \textbf{kinetic cylinders} for $r>0$
\begin{equation}
  \label{eq:hypo-cyl}
  Q_r := rQ_1 := (-r^2,0] \times B_{r^3} \times B_r,
\end{equation}
where the left multiplication $rQ_1$ corresponds to the scaling law. The invariance by Galilean transformation naturally leads to the following more general translated kinetic cylinders
\begin{align*}
  Q_r(z_0)
  & = z_0 \circ [rQ_1] \\
  & =  \left\{(t,x,r) \, : \, -r^2 < t-t_0 \le 0, \ |x-x_0-(t-t_0)v_0| < r^3, \ |v-v_0|<r \right\},
\end{align*}
where the $\circ$ left operation corresponds to the (non-commutative) Galilean translation. 

\subsection{The kinetic first De Giorgi Lemma}

The main difficulty in adapting the first De Giorgi lemma, by comparison with the parabolic case, lies in the degeneracy of~\eqref{eq:hypo}, which lacks ellipticity in the $x$ variable. We sketch two arguments for the core part of the proof, the gain of integrability from $L^2$ to $L^{p_*}$ with some $p_*>2$ in all variables. We first present the approach of~\cite{golse2019harnack}, based on the so-called ``averaging lemma'', and, second, the strategy introduced in~\cite{pascucci2004moser} and used in~\cite{wang2009c,guerand2022quantitative}. The second strategy leads to sharper results. Its drawback however is that it relies on the fundamental solution of the Kolmogorov equation. So it might not be as robust as the first strategy for future applications.

Overall, the statement remains quite similar to the parabolic case. Note the additional assumption $\cT f \in L^2_t((-1,0);L^2_xH^{-1}_v(B_1 \times B_1))$ made on the weak subsolutions. Just like the assumption made on $\p_t f$ in the parabolic case, it ensures that we can easily perform approximation procedures and a priori estimates. Such regularity on $\cT f$ is always satisfied by weak solutions by combining the equality $\cT f = \nabla_v \cdot (A \nabla_v f)$ in the sense of distributions with the estimate $A \nabla_v f \in L^2$ provided by the kinetic Caccioppoli inequality~\eqref{eq:caccio-kin} (the energy estimate) and the pointwise bounds on $A$. It is however not always satisfied by subsolutions due to the presence of a defect measure, as in the parabolic case.

\begin{lemma}[Kinetic first De Giorgi Lemma]
  \label{lem:hypo-first}
  Given $\Lambda>0$, there is $\delta>0$, depending only on $\Lambda$ and the dimension $d$, so that the following holds:
  \smallskip

  Let $f$ be a weak subsolution in $Q_1$ to
  \begin{equation}
    \label{eq:main-dg1}
    \p_t f + v \cdot \nabla_x f \le \nabla_v \cdot \( A \nabla_v f \)
  \end{equation}
  in the functional setting
  \begin{align*}
    & f \in L^\infty_t((-1,0);L^2_{x,v}(B_1 \times B_1)) \cap L^2_t((-1,0);L^2_xH^1_v(B_1 \times B_1)) \\
    & \cT f \in L^2_t((-1,0);L^2_x H^{-1}_v(B_1 \times B_1)),
  \end{align*}
  with $A$ symmetric measurable so that
  \begin{equation*}
    \Lambda^{-1} \le A(t,x,v) \le \Lambda \ \text{ for almost every } (t,x,v) \in Q_1.
  \end{equation*}

  Then, we have 
  \begin{equation*}
    \left\| f_+ \right\|_{L^2(Q_1)} \le \delta \quad \implies \quad \left\| f_+ \right\|_{L^\infty(Q_{1/2})}  \le \frac12.
  \end{equation*}
  This estimate is translation-invariant under the Galilean transformations~\eqref{eq:translation-kin}, and scaling-invariant under the kinetic scaling~\eqref{eq:scaling-kin}.
\end{lemma}

\begin{proof}[Proof of Lemma~\ref{lem:hypo-first}]
  The setting of the kinetic first De Giorgi lemma is represented in Figure~\ref{fig:first-hypo}: it is schematically very similar to the parabolic case, provided the new definition~\eqref{eq:hypo-cyl} of cylinders is adopted. 
  \begin{figure}
    \includegraphics[scale=0.8]{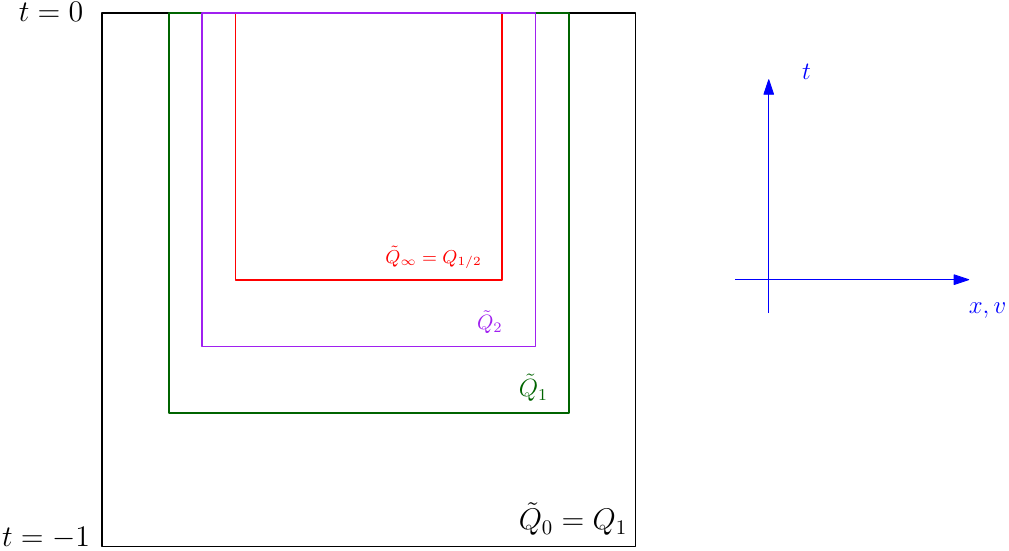}
    \caption{Setting of the first De Giorgi lemma in the kinetic case.}
    \label{fig:first-hypo}
  \end{figure}
  For instance, in the De Giorgi viewpoint, the sequence of cylinders and energy levels in the iteration is 
  \begin{equation*}
    \begin{cases}
      &T_k := -\frac{1}{4} - \frac{1}{2^{k+1}}, \quad T_0 =-1, \quad T_k \uparrow -\frac{1}{4}, \\[2mm]
      &r_k ^x = \frac{1}{8} + \frac{1}{2^{k+1}},\quad r_0 =1, \quad r_k \downarrow \frac{1}{8}, \\[2mm]
      &r_k ^v = \frac{1}{2} + \frac{1}{2^{k+1}},\quad r_0 =1, \quad r_k \downarrow \frac{1}{2}, \\[2mm]
      &e_k = \frac{1}{2} - \frac{1}{2^{k+1}},  \quad e_0=0, \quad e_k \uparrow \frac{1}{2}, \\[2mm]
      & \tilde Q_k := (T_k,0) \times B_{r^x_k} \times B_{r^v_k}, \\[2mm]
      & \displaystyle E_k := \int_{\tilde Q_k} (f-e_k)_+^2 \dd t \dd x \dd v.
    \end{cases}
  \end{equation*}
  Note again that the cylinders $\tilde Q_k$ are not exactly of the form~\eqref{eq:hypo-cyl}; instead they are constructed so as to converge to $Q_{1/2}$.

  One establishes again a nonlinear inequality between successive $E_k$ terms. The proof is based on a \textbf{localised energy estimate}, a small \textbf{gain of integrability in all variables} (playing the role of a ``kinetic Sobolev embedding'') and Chebyshev's inequality. Only the localised energy estimate and the gain of integrability from $L^2$ to $L^{p_*}$ for some $p_*>2$ in all variables---the ``kinetic Sobolev embedding''---differ from the previous proofs. We only discuss these two substeps. The other steps in the De Giorgi--Moser iteration are similar to the elliptic and parabolic cases.
  \medskip

  \noindent
  \textbf{A. Localised energy estimate.} We write the energy estimate in a way that is suitable for both the De Giorgi and Moser approaches to the iteration, as in the previous proofs.

  Consider an energy level $e \ge 0$; it is used in De Giorgi's iteration, but fixed to zero in Moser's iteration. Consider a smooth localisation function $\phi$ with support in $(-1,0] \times B_1 \times B_1$: note that $\phi$ is zero at initial time but not necessarily at final time. We integrate~\eqref{eq:main-dg1} against $(f-e)_+ \phi^2$ on $[-1,\tau]$ for $\tau \in (-1,0]$:
  \begin{align*}
    & \int_{(-1,\tau] \times B_1 \times B_1} \left[ \cT (f-e)_+ \right] (f-e)_+ \phi^2 \dd t \dd x \dd v \\
    & = - \int_{(-1,\tau] \times B_1 \times B_1} \left[ A \nabla_v (f-e)_+ \cdot \nabla_v (f-e)_+ \right] \phi^2 \dd t \dd x \dd v \\
    & \quad - 2 \int_{(-1,\tau] \times B_1 \times B_1} \left[ A \nabla_v (f-e)_+ \cdot \nabla_v \phi \right] \phi (f-e)_+ \dd t \dd x \dd v.
  \end{align*}
  The left-hand side writes
  \begin{align*}
    & \int_{(-1,\tau] \times B_1 \times B_1} \left[ \cT (f-e)_+ \right] (f-e)_+ \phi^2 \dd t \dd x \dd v \\
    & = \frac12 \int_{(-1,\tau] \times B_1 \times B_1} \left[ \partial_t (f-e)^2_+ + v \cdot \nabla_x (f-e)^2_+ \right] \phi^2 \dd t \dd x \dd v \\
    & = \frac12 \left( \int_{B_1 \times B_1} (f-e)^2_+ \phi^2 \dd x \dd v \right)\Bigg|_{t=0} - \frac12  \int_{(-1,\tau] \times B_1 \times B_1} (f-e)^2_+ \cT \left( \phi^2 \right) \dd t \dd x \dd v,
  \end{align*}
  and we bound the last term as
  \begin{multline*}
    \left| \frac12 \int_{(-1,\tau] \times B_1 \times B_1} (f-e)^2_+ \cT \left( \phi^2 \right) \dd t \dd x \dd v \right| \\
    \lesssim \| \cT \phi \|_\infty \|\phi \|_\infty \int_{Q_1 \cap \, \mathrm{supp} \, \phi} (f-e)^2_+ \dd t \dd x \dd v.
  \end{multline*}
  Therefore we deduce
  \begin{align*}
    & \frac12 \left( \int_{B_1 \times B_1} (f-e)^2_+ \phi^2 \dd x \dd v \right)\Bigg|_{t=\tau} + \int_{(-1,\tau] \times B_1 \times B_1} \left| \nabla_v (f-e)_+ \right|^2 \phi^2  \dd t \dd x \dd v \\
    & \lesssim \Lambda \| \cT \phi \|_\infty \|\phi \|_\infty \int_{Q_1 \cap \mathrm{supp} \, \phi} (f-e)^2_+ \dd t \dd x \dd v \\
    & \hspace{1cm} + \Lambda \int_{(-1,\tau] \times B_1 \times B_1} \left|A \nabla_v (f-e)_+ \cdot \nabla_v \phi \right| \phi (f-e)_+ \dd t \dd x \dd v\\
    & \lesssim \Lambda \| \cT \phi \|_\infty \|\phi \|_\infty \int_{Q_1 \cap \mathrm{supp} \, \phi} (f-e)^2_+  \dd t \dd x \dd v \\
    & \hspace{1cm}+ \Lambda^2 \left( \int_{(-1,\tau] \times B_1 \times B_1} \left|\nabla_v (f-e)_+ \right|^2 \phi \dd t \dd x \dd v \right)^{\frac12}  \times \\
    & \hspace{3cm} \left( \int_{(-1,\tau] \times B_1 \times B_1} \left|\nabla_v \phi \right|^2 (f-e)_+^2 \dd t \dd x \dd v \right)^{\frac12}.
  \end{align*}
  By splitting the last term into two squares, we get
  \begin{multline*}
    \left( \int_{B_1 \times B_1} (f-e)^2_+ \phi^2 \dd x \dd v \right)\Bigg|_{t=\tau} + \int_{(-1,\tau] \times B_1 \times B_1} \left| \nabla_v (f-e)_+ \right|^2 \phi^2 \dd t \dd x \dd v \\
    \lesssim \Lambda^4 \left( 1+ \|\nabla_v\phi\|_\infty + \|\mathcal{T} \phi\|_\infty \right) \, \int_{Q_1 \cap \, \mathrm{supp} \, \phi} (f-e)^2_+ \dd t \dd x \dd v.
  \end{multline*}
  Finally we vary the final time $\tau \in (-1,0]$ to deduce the following \textbf{kinetic Caccioppoli inequality}, which is the kinetic counterpart to~\eqref{eq:caccio-para}: 
  \begin{multline}
    \label{eq:caccio-kin}
    \sup_{-1<\tau\le 0} \left( \int_{B_1 \times B_1} (f-e)^2_+ \phi^2 \dd x \dd v \right)\Bigg|_{t=\tau} + \int_{Q_1} \left| \nabla_v (f-e)_+ \right|^2 \phi^2 \dd t \dd x \dd v \\
    \lesssim \Lambda^4 \left( 1+ \|\nabla_v\phi\|_\infty + \|\mathcal{T} \phi\|_\infty \right) \, \int_{Q_1 \cap \, \mathrm{supp} \, \phi} (f-e)^2_+ \dd t \dd x \dd v.
  \end{multline}

  Given $\phi_k$ a sequence of smooth localisation functions such that \begin{equation*}
    \mathbf 1_{\tilde Q_{k+1}} \le \phi_k \le \mathbf 1_{\tilde Q_k} \quad \text{ and } \quad |\nabla_{t,x,v} \phi_k| \lesssim 2^k,
  \end{equation*}
  the sequence of localised energy estimate is therefore
  \begin{multline*}
    \sup_{T_{k+1} \le \tau \le 0} \left( \int_{B_1 \times B_1} \phi^2_{k+1} (f-e_{k+1})_+^2 \dd x \dd v \right)\Bigg|_{t=\tau} \\
    + \int_{Q_1} \left|\nabla_x (\phi_{k} (f-e_{k+1})_+)\right|^2 \dd t \dd x \dd v \lesssim \Lambda^4 2^k E_k.
  \end{multline*}
  
  Note on the one hand that the energy estimate only provides control on the Sobolev regularity in the $v$ variable. It misses part of the variables (the space variable $x$), in contrast to the parabolic case. On the other hand, just like in the parabolic case, the gain of integrability in $t$ is ensured by the first term on the left hand side.   
  \medskip
  
  \noindent \textbf{B. Gain of integrability (proof by averaging lemma).} The goal of this part of the proof is to recover a gain of integrability in all variables.  Once this gain of integrability is obtained, with a different (smaller) exponent $p_*>2$ as compared to the elliptic and the parabolic cases, the structure of the iteration is exactly similar. Let us consider a smooth localisation function $\phi$ so that $\mathbf{1}_{\tilde Q_{k+1}} \le \phi \le \mathbf{1}_{\tilde Q_{k}}$, which satisfies $| \nabla_{t,x,v} \phi | \lesssim 2^{k}$.

  Then, by applying the standard Sobolev inequality in the $v$ variable to~\eqref{eq:caccio-kin} (dropping the first term on the left hand side), we obtain
  \begin{equation*}
    \left\| (f-e)_+ \phi\right\|_{L^2_{t,x}L^p_v(\tilde Q_{k+1})} \lesssim \Lambda^4 2^k \|(f-e)_+\|_{L^2(Q_1 \cap \mathrm{supp} \, \phi)}
  \end{equation*}
  for some $p>2$.
  
  To recover the gain of integrability in $x$ we would like to apply the following standard result in kinetic theory:
  \begin{lemma}[Averaging Lemma~\cite{MR808622,MR923047,zbMATH01975967}]
    Given $f = f(t,x,v) \in L^2(\R^{1+2d})$ which satisfies the equation
    \begin{equation*}
      \cT f = \partial_t f + v \cdot \nabla_x f = \nabla \cdot H_1 + H_0 \quad \text{ with } \quad H_0, H_1 \in L^2(\R^{1+2d}),
    \end{equation*}
    then for any $\varphi =\varphi(v) \in C^\infty_c(\R^d)$, we have
    \begin{equation}
      \label{eq:golse}
      \forall \, s \in \(0,\frac13\), \quad (t,x) \mapsto \int_{\R^d} \varphi \, f \dd v \ \text{ belongs to } \  H^s(\R^{1+d})
    \end{equation}
    with a bound that is controlled linearly by the sum of the $L^2(\R^{1+2d})$ norms of $f$, $H_1$ and $H_2$, and with a constant depending on $\varphi$.

    We also have (Bouchut's version in~\cite[Theorem~1.3]{zbMATH01975967}):
    \begin{multline}
      \label{eq:bouchut}
      \left\| D^{\frac13} _x f \right\|_{L^2(\R^{1+2d})} ^2 + \left\| D^{\frac13} _t f \right\|_{L^2(\R^{1+2d})} ^2 \\
      \lesssim \| f \|_{L^2(\R^{1+2d})} + \left\| \nabla_v f \right\|_{L^2(\R^{1+2d})} ^2 + \left\| H_0 \right\|_{L^2(\R^{1+2d})} ^2 + \left\| (1+|v|^2) H_1 \right\|_{L^2(\R^{1+2d})} ^2
    \end{multline}
    where $D_\cdot^{1/3} = (-\Delta_\cdot)^{1/3}$ is the standard fractional derivative.
  \end{lemma}

  Note that we will review the proof based on~\eqref{eq:bouchut} (as done in~\cite{golse2019harnack}), because it is slightly shorter and cleaner. It is however possible to use~\eqref{eq:golse} instead: for non-negative solutions the estimate gives a control $H^s_{t,x}L^1_v$ provided we localise in velocity, and combined with the energy estimate $L^\infty_t L^2_{x,v} \cap L^2_{t,x} H^1_v$ and some interpolation, this would give the gain of integrability in all variables by a comparison argument as discussed below. The interest of~\eqref{eq:golse}, by comparison with~\eqref{eq:bouchut}, is that such estimate remain true for more general transport operators $\p_t + b(v) \cdot \nabla_x$ where $b$ satisfies weak non-degeneracy assumptions (see~\cite{zbMATH07986618} for a more detailed discussion).
  
  In order to apply~\eqref{eq:bouchut}, we face a difficulty. It does not apply to subsolutions of our problem because of the lack of $L^2_{t,x}H^{-1}_v(\R^{1+2d})$-regularity of the defect measure. Indeed the subsolutions, say $f$, we are dealing with, satisfy merely
  \begin{equation*}
    \cT f \le h 
  \end{equation*}
  with $h \in L^2_{t,x}H^{-1}_v(\R^{1+2d})$. Expressed with a defect measure this means
  \begin{equation*}
    \cT f = h - m 
  \end{equation*}
  with $h \in L^2_{t,x}H^{-1}_v(\R^{1+2d})$ and $m$ a non-negative measure on $\R^{1+2d}$. We will solve this issue by using a comparison principle, since we are \emph{in fine} only interested in integrability properties of $f$, which are preserved by comparison.

  Let us consider a weak subsolution $f$ in $Q_1$, and let us assume without loss of generality that it is non-negative. 
  We define the new unknown $\tilde f := f \phi_k$ on $\R^{1+2d}$, extending it by zero for $t > 0$. It satisfies the inequation
  \begin{equation*}
    \cT \tilde f \le \nabla \cdot ( A \nabla \tilde f ) + \nabla_v \cdot H_1 + H_0 
  \end{equation*}
  on the whole $\R^{1+2d}$, with
  \begin{equation*}
    \begin{cases}
      H_0 := f \( \nabla_v \cdot (A \nabla_v \phi_k) - \cT \phi_k \) \mathbf{1}_{t \le 0}, \\[2mm]
      H_1 :=  A \nabla_v f \cdot \nabla_v \phi_k \mathbf{1}_{t \le 0}.
    \end{cases}
  \end{equation*}
  Note that we have used that $\p_t \mathbf{1}_{t \le 0} \le 0$. 

  The kinetic Caccioppoli inequality~\eqref{eq:caccio-kin} (local energy estimate) on $f$ implies
  \begin{equation}
    \label{eq:energy0-1}
    \| \nabla_v f \|_{L^2(\tilde Q_k)} \lesssim \| f \|_{L^2(\tilde Q_{k-1})}. 
  \end{equation}
  All supports in the equation for $\tilde f$ are included in $\tilde Q_k$ so we use~\eqref{eq:energy0-1} to deduce
  \begin{equation*}
    \| H_0 \|_{L^2(\R^{1+2d})} + \| H_1 \|_{L^2(\R^{1+2d})} \lesssim \| f \|_{L^2(\tilde Q_k)}.
  \end{equation*}
  The unknown $\tilde f$ is zero at $t=-1$ because of the localisation by $\phi_1$. We define $g$ that is zero for $t \le-1$ and satisfies the equation
  \begin{equation}
    \label{eq:comp-g}
    \cT g = \nabla \cdot ( A \nabla g ) + \nabla_v \cdot H_1 + H_0
  \end{equation}
  on $[-1,+\infty) \times \R^{2d}$.
  
  By comparison we have $0 \le \tilde f \le g$ on $Q_1$. Moreover, we integrate~\eqref{eq:comp-g} on $(-\infty,\tau] \times B_1 \times B_1$ for any $\tau > -1$ to get
  \begin{multline}
    \label{eq:energy-sur-g}
    \sup_{\tau > -1} \| g(\tau,\cdot,\cdot) \|_{L^2(\R^{2d})} + \| \nabla_v g \|_{L^2(\R^{1+2d})} \\ \lesssim \Lambda^2 \( \| H_0 \|_{L^2(\R^{1+2d})} + \| H_1 \|_{L^2(\R^{1+2d})} \) \lesssim \| f \|_{L^2(Q_1)}.
  \end{multline}
  Therefore, given a smooth time localisation function $\theta=\theta(t)$ so that
  \begin{equation*}
    \mathbf{1}_{(-1,0]} \le \theta \le \mathbf{1}_{(-1,2]},
  \end{equation*}
  the new unknown $\tilde g(t,x,v) := \theta(t) g(t,x,v)$ satisfies the following modified version of the equation~\eqref{eq:comp-g}:
  \begin{equation*}
    \cT \tilde g = \nabla \cdot ( A \nabla \tilde g ) + \nabla_v \cdot H_1 + \tilde H_0 \quad \text{ where } \quad \tilde H_0 = H_0 + \theta'(t) g.
  \end{equation*}
  And the estimate~\eqref{eq:energy-sur-g} implies
  \begin{equation*}
    \| \tilde g \|_{L^2(\R^{1+2d})} + \| \nabla_v \tilde g \|_{L^2(\R^{1+2d})} + \| \tilde H_0 \|_{L^2(\R^{1+2d})} + \| H_1 \|_{L^2(\R^{1+2d})} \lesssim \| f \|_{L^2(Q_1)}.
  \end{equation*}
  We can therefore apply~\eqref{eq:bouchut} to $\tilde g$ and, together with the last bound, it implies
  \begin{equation}
    \label{eq:estim-g-bis}
    \left\| D^{\frac13} _x \tilde g \right\|_{L^2(\R^{1+2d})}  + \left\| D^{\frac13} _t \tilde g \right\|_{L^2(\R^{1+2d})}  + \| \tilde g \|_{L^2(\R^{1+2d})} + \| \nabla_v \tilde g \|_{L^2(\R^{1+2d})} \lesssim \| f \|_{L^2(Q_1)}.
  \end{equation}
  Therefore by Sobolev's embedding we get $\tilde g \in L^{p_*}(\R^{1+2d})$, for some $p_* \in (2,p)$. This local integrability is finally inherited by $f$ on $\tilde Q_{k+1}$ since $f=\tilde f \le g= \bar g$ on $\tilde Q_{k+1}$. Tracking down all these estimates leads to the inequality
  \begin{equation*}
    \left\|(f-e)_+\right\|_{L^{p_*}(\tilde Q_{k+1})} \lesssim 2^k \|(f-e)_+\|_{L^2(Q_1 \cap \, \mathrm{supp} \, \phi_k)}.
  \end{equation*}
  
  \begin{remark}
    In the final comparison between $f$ and $g$, the Sobolev regularity in $t$ and $x$ seems lost. However, by integrating the inequation on $f$:
    \begin{equation*}
      \cT f = \partial_t f + v \cdot \nabla_x f = \nabla_v \cdot ( A \nabla_x f ) - m
    \end{equation*}
    against a smooth localisation function, and without squaring the unknown, one can deduce a bound on the total variation of the defect measure $m \le 0$ in terms of the $L^2$ norm of $f$. Such a bound in total variation translates into a control on the $W^{-1/2,q}_{t,x,v}$ norm of $m$, for some $q >1$ close to $1$. Then, more sophisticated forms of averaging lemmas can be invoked to obtain a gain of regularity $W^{s,q}$ in $t,x$ for $s>0$ small. It means that even subsolutions have some amount of Sobolev regularity, which can be useful for compactness arguments for instance. This low level of Sobolev regularity for subsolutions is also used quantitatively  in~\cite{guerand2022quantitative}.
  \end{remark}

  \noindent
  \textbf{C. Gain of integrability (proof by fundamental solutions).} 
  This is an alternative proof of the step of ``gain of integrability'' $L^2 \to L^{p_*}$ with $p_*>0$ in all variables. This argument was proposed in~\cite{pascucci2004moser}. The idea is to use the explicit fundamental solutions to the Kolmogorov equation, when $A$ is the identity matrix, and treat the actual operator as a source term. It leads to sharper results but is less robust since it crucially relies on knowning explicitly the fundamental solutions of the equation when the rough coefficients are constant.
  \medskip

  We study the following auxiliary problem:
  \begin{lemma}[Estimates on Kolmogorov's equation]
    \label{prop32}
    Let $f \ge 0$ be a locally integrable function on $\R^{1+2d}$ that satisfies
    \begin{equation}
      \label{kol}
      \mathcal{K} f := (\cT - \Delta_v) f = \nabla_v \cdot F_1 + F_2 -m,
    \end{equation}
    with $F_1,F_2 \in L^1 \cap L^2(\R^{1+2d})$, $m \ge 0$ a measure on $\R^{1+2d}$. Then, we have 
    \begin{align}
        \label{k1}
        &\|f\|_{L^{2 +\frac1d}(\R^{1+2d})} \lesssim \|F_1\|_{L^2(\R^{1+2d})} + \|F_2\|_{L^2(\R^{1+2d})}, \\
        \label{k2} 
        &\|f\|_{L^1_{t,v}W_x^{\sigma_1,1}(\R^{1+2d})} \lesssim \|F_1\|_{L^1(\R^{1+2d})} + \|F_2\|_{L^1(\R^{1+2d})} + \|m\|_{\mathrm{TV}(\R^{1+2d})}, \\
      \label{k3} 
        &\|f\|_{L^1_{x,v}W_t^{\sigma_2,1}(\R^{1+2d})} \lesssim \|F_1\|_{L^1(\R^{1+2d})} + \|F_2\|_{L^1(\R^{1+2d})} + \|m\|_{\mathrm{TV}(\R^{1+2d})},
    \end{align}
    for any $\sigma_1 \in (0,1/3)$ and $\sigma_2 \in (0,1/2)$. In~\eqref{k2} and~\eqref{k3}, the constants are proportional to $(1/3-\sigma_1)^{-1}$ and $(1/2-\sigma_2)^{-1}$, respectively.
  \end{lemma}

  \begin{remark}
    It is an interesting question to know whether the maximal regularity $L^1_{t,v}W_x^{1/3,1}$ is attained in such an $L^1$ setting, in which the Young inequality for weak Lebesgue spaces (see for instance~\cite[Theorem~1.4.25, p.~73]{zbMATH06313565}, and various other standard inequalities, degenerate.
  \end{remark}

\begin{proof}[Proof of Proposition \ref{prop32}]
    We express $f$ in terms of Kolmogorov's fundamental solution:
    \begin{equation}
      \label{eq:convol-kol}
      f(t,x,v) = \int_{\R^{1+2d}} G\(t-t',x-x'-(t-t')v',v-v'\) \mathcal{K}(f) \(t',x',v'\) \dd z',
    \end{equation}
    where the Galilean translation is used to define the modified convolution, and  
    \begin{equation}
      \label{k-fond}
      G(t,x,v) =
      \left\{
        \begin{aligned}
        & \left(\frac{3}{4\pi^2t^4}\right)^{\frac{d}{2}} \, \exp \, \left[ - \frac{3\left|x-\frac{t}{2} v\right|^2}{t^3} - \frac{|v|^2}{4t} \right] , & t>0,\\[3mm]
        & 0, &t \le 0,
        \end{aligned}
        \right.
    \end{equation}
    is the fundamental solution to $\mathcal{K} (G) = \delta_0$. The calculation of $G$ by Kolmogorov is based on the Fourier transform in $(x,v)$ and the characteristic method. This classical result is exposed in~\cite[Chapter~VII, Section~7.6, p.~211]{zbMATH01950198}, and a recent concise re-writing can be found in~\cite[Appendix A]{bouin2019hypocoercivity}.
    
    Since $f,G,m \ge 0$, we deduce that
    \begin{equation*}
      0\le f \le \int_{\R^{1+2d}} G\(t-t',x-x'-(t-t')v',v-v'\) \left[\nabla_v \cdot F_1 +F_2\right]\(t',x',v'\) \dd z'.
    \end{equation*}
    Then, we have, denoting by $L^{p,\infty}$ the weak Lebesgue space,
    \begin{equation*}
      |\nabla_v G(t,x,v)| + t |\nabla_x G(t,x,v)| \le \frac{1}{t^{2d+\frac12}} \, \exp \left[  - \frac{3\left|x-\frac{t}{2} v\right|^2}{t^3} - \frac{|v|^2}{4t}  \right]
    \end{equation*}
    so that
    \begin{equation*}
      G, \ \nabla_v G(t,x,v), \ t \nabla_x G(t,x,v) \in L^{\frac{2d+1}{2d+1/2},\infty}(\R^{1+2d}).
    \end{equation*}
    By integration by parts and Young's inequality (easily extended to our modified convolution product), we deduce~\eqref{k1}.
    
    To prove~\eqref{k2} and~\eqref{k3}, the idea is to use, roughly, that $\nabla_v G$ and $t \nabla_x G$ ``almost'' belong to $L^1_{t,v} W^{1/3,1}_x$ and $L^1_{x,v} W^{1/2,1}_t$, i.e. in the sense of \emph{weak} Lebesgue spaces. We refer to~\cite[Section~1.3, p.~7]{zbMATH03536702} for the definition of weak Lebesgue spaces, as well as the more general family of Lorentz spaces they belong to. We only sketch the argument. Define for $\var >0$ the following decomposition:
    \begin{equation*}
      G := G_\var + G_\var^\perp, \quad \text{ with } \quad
      \begin{cases}
        G_\var(t,x,v) = \chi\(\frac{t}{\var}\) G(t,x,v) \\[2mm]
        G_\var(t,x,v) = \left( 1- \chi\(\frac{t}{\var}\) \right) G(t,x,v)
      \end{cases}
    \end{equation*}
    for some smooth cutoff function $\chi \in C^\infty(\R_+)$ such that $\mathbf 1_{[0,1]} \le \chi \le \mathbf 1_{[0,2]}$. The formula~\eqref{eq:convol-kol} then implies a corresponding decomposition $f = f_\var + f_\var^\perp$ on $f$.
    
    Arguing similarly as above, by integration by parts and Young's inequality, one gets 
    \begin{equation*}
      \begin{cases}
        \|f_\var\|_{L^1_{t,x,v}} \lesssim \var^{\frac12}
        \left( \|F_1\|_{L^1} + \|F_2\|_{L^1} + \|m\|_{\mathrm{TV}} \right) a \\[2mm]
        \|f_\var^\perp\|_{L^1_{t,v}\mathrm{W}^{k,1}_x} \lesssim \var^{-\frac12-\frac{3k}{2}} \left( \|F_1\|_{L^1} + \|F_2\|_{L^1} + \|m\|_{\mathrm{TV}} \right) \\[2mm]
        \|f_\var^\perp\|_{L^1_{x,v}\mathrm{W}^{k,1}_t} \lesssim \var^{-\frac12-k} \left( \|F_1\|_{L^1} + \|F_2\|_{L^1} + \|m\|_{\mathrm{TV}} \right)
    \end{cases}
    \end{equation*}
    for any $k \ge 1$.

    The convergence $O(\var^{1/2})$ to zero in the first line is due to the shrinking domain of integration in time when integrating $|\nabla_v G(t,x,v)| + t |\nabla_x G(t,x,v)| \sim t^{-1/2}$. The blow-up $O(\var^{-1/2-3k/2})$ in the second line follows from integrating
    \begin{equation*}
      \nabla_x ^k \nabla_v G(t,x,v) \ \text{ and } \  t \nabla^{k+1}_x G(t,x,v) \quad \sim \ t^{-\frac12-\frac{3k}{2}}
    \end{equation*}
    away from the singularity $t \ge \var$ (and using the time integrability ensured by $k \ge 1$). The blow-up $O(\var^{-1/2-k})$ in the third line follows from integrating
    \begin{equation*}
      \p_t ^k \nabla_v G(t,x,v) \ \text{ and } \  \p_t^k t \nabla^k_x G(t,x,v) \quad \sim \ t^{-\frac12-k}
    \end{equation*}
    away from the singularity $t \ge \var$ (and using again the time integrability ensured by $k \ge 1$). 

    By the Littlewood-Paley decomposition or any equivalent interpolation argument, one then deduces that $f \in L^1_{t,v} W^{\sigma_1,1}_x$ for $\sigma_1 \in [0,1/3)$, and that $f \in L^1_{x,v} W^{\sigma_2,1}_t$ for $\sigma_2 \in [0,1/2)$.
\end{proof}

We now deduce the gain of integrability on subsolutions, as well as gain of (low level) Sobolev regularity in $x$. We consider $f \ge 0$ a sub-solution to~\eqref{eq:main-dg1} on $\tilde Q_k$, and we claim: 
  \begin{equation}
    \label{k4}
      \|f\|_{L^p(\tilde Q_{k+1})} + \|f\|_{L^1_{t,v} W^{\sigma_1,1}_x(\tilde Q_{k+1}))} + \|f\|_{L^1_{x,v} W^{\sigma_1,1}_t(\tilde Q_{k+1}))} \lesssim 2^k \|f\|_{L^2(\tilde Q_k)},
  \end{equation} 
  with $\sigma_1 \in [0,1/3)$, $\sigma_2 \in [0,1/2)$ and a constant degenerating as $\sigma_1 \to 1/3$ or $\sigma_2 \to 1/2$.

  Given $\phi_k$ the usual smooth localisation function such that $\mathbf 1_{\tilde Q_{k+1}} \le \phi_k \le \mathbf 1_{\tilde Q_k}$, we define $\tilde f := f \phi_k$ for $t \le 0$ and $\tilde f=0$ for $t >0$. Then $\tilde f$ satisfies the equation
  \begin{equation}
    \label{eq:localise-kol}
    \mathcal K \tilde f = \nabla_v \cdot F_1 + F_0 - \tilde m
  \end{equation}
  with 
  \begin{align*}
    & F_1 = (A \nabla_v f) \phi_k - (\nabla_v f) \phi_k -f \nabla_v \phi_k, \\
    & F_0 = - A \nabla_v f \cdot \nabla_v \phi_k + f \cT \phi_k - \nabla_v \phi_k \cdot \nabla_v f,\\
    & \tilde m = \phi_k m.
  \end{align*}
  and $m$ the defect measure of the subsolution $f$.

  The local energy estimate~\eqref{eq:caccio-kin} and the compact support imply
  \begin{equation*}
    \|F_1\|_{L^1(\R^{1+2d})} + \|F_2\|_{L^1(\R^{1+2d})} + \|F_1\|_{L^2(\R^{1+2d})} + \|F_2\|_{L^2(\R^{1+2d})} \lesssim \| f\|_{L^2(\tilde Q_k)}.
  \end{equation*}

  By integrating~\eqref{eq:localise-kol} against $\phi_k$ (without squaring the unknown) we also obtain
  \begin{equation*}
    \| \tilde m \|_{\mathrm{TV}(\R^{1+2d})} \lesssim \| f\|_{L^2(\tilde Q_k)}.
  \end{equation*}

  Therefore we can apply~\eqref{k1}-\eqref{k2}-\eqref{k3} to deduce~\eqref{k4}. 
\end{proof}

\subsection{The improved kinetic first De Giorgi lemma}

Similarly to the elliptic and parabolic cases, we can improve the first De Giorgi lemma as follows.
\begin{lemma}[Improved kinetic first De Giorgi Lemma]
  \label{lem:dg3+}
  Given $\Lambda>0$ and $\zeta_0 \in (0,2]$, there is $C>0$, depending only $\Lambda$, $\zeta_0$ and the dimension $d$, so that the following holds:
  \smallskip

  Let $f$ be a weak subsolution in a kinetic cylinder $Q_R$ to
  \begin{equation*}
    \p_t f + v \cdot \nabla_x f \le \nabla_v \cdot \( A \nabla_v f \)
  \end{equation*}
  in the functional setting
  \begin{align*}
    & f \in L^\infty_t((-R^2,0);L^2_{x,v}(B_{R^3} \times B_R)) \cap L^2_t((-R^2,0);L^2_xH^1_v(B_{R^3} \times B_R)) \\
    & \cT f \in L^2_t((-R^2,0);L^2_x H^{-1}_v(B_{R^3} \times B_R)),
  \end{align*}
  with $A$ symmetric measurable so that
  \begin{equation*}
    \Lambda^{-1} \le A(t,x) \le \Lambda \ \text{ for almost every } (t,x,v) \in Q_R.
  \end{equation*}

  Then, we have, for any $\zeta \in (\zeta_0,2]$ and $0<r<R$, 
  \begin{equation*}
    \|f\|_{L^\infty(Q_r)} \le C (R-r)^{-\frac{2+4d}{\zeta}} \left( \frac{1}{|Q_R|} \int_{Q_R} f^\zeta \dd z \right)^{\frac{1}{\zeta}}.
  \end{equation*}
\end{lemma}

\begin{proof}
  By using the scaling of the equation in the kinetic first De Giorgi lemma, one can prove the following inequality between two general cylinders $Q_r$ and $Q_R$:
  \begin{equation*}
     \|f\|_{L^\infty(B_{r})} \lesssim (R-r)^{-1-2d} \left( \frac{1}{|Q_R|} \int_{Q_R} f^2 \dd z \right)^{\frac{1}{2}}.
  \end{equation*}
  The Young's inequality and the inequality
  \begin{equation*}
    \left( \frac{1}{|Q_R|} \int_{Q_R} f^2 \dd z \right)^{\frac{1}{2}} \lesssim \left( \frac{1}{|Q_R|} \int_{Q_R} f^\zeta \dd z \right)^{\frac12} \|f\|^{1-\frac{\zeta}{2}}_{L^\infty(Q_R)},
  \end{equation*}
  imply that the function $\mathcal N(r):=\|f\|_{L^\infty(Q_r)}$ verifies the following inequality for some $C>0$ and any $0<r<r'<R$:
  \begin{equation*}
    \mathcal N(r) \le \frac{1}{2} \mathcal N(r') + C  (r'-r)^{-\frac{2+4d}{\zeta}} \left( \frac{1}{|Q_R|} \int_{B(0,R)} f^\zeta \dd z \right)^{\frac1\zeta}.
  \end{equation*}
  
  By the same iteration along the sequence of radii
  \begin{equation*}
    r_{k+1} = r_k + \var \(1-\alpha^{-1}\)\alpha^{-k}
  \end{equation*}
  for $k \ge 0$, some $\alpha>1$, $\var:=R-r$ and $r_0:=r$, we deduce as in the elliptic and parabolic cases that
  \begin{equation*}
    \mathcal N(r) \lesssim_{\zeta_0} \var^{-\frac{2+4d}{\zeta}} \left( \frac{1}{|Q_R|} \int_{Q_R} f^\zeta \dd z \right)^{\frac1\zeta},
  \end{equation*}
  from which the conclusion follows.
\end{proof}

\subsection{The kinetic second De Giorgi lemma}

Inspired by the use of Poincar\'e's inequality as an \emph{integral measure of the oscillation} in the works of Moser \cite{moser1960new,moser1964harnack}, Kruzhkov \cite{kruzhkov1963priori}, and later in \cite{wang2009c} in a modified fashion, we introduce the strategy of \cite{guerand2022quantitative}, in the form simplified and optimised in~\cite{niebel2023kinetic,anceschi2024poincare,dmnz}. This method does not involve the equation verified by the logarithm of a solution to~\eqref{eq:main-dg1}, as in Nash's, Moser's and Kruzhkov's methods.

The kinetic second De Giorgi lemma (reduction of oscillation) resembles the elliptic case, and it implies the Theorem~\ref{thm:dgnm4} by an argument also similar to the parabolic case, but accounting for the slightly different geometry of the kinetic cylinders. 

More generally, the overall structure of the proof of the kinetic second De Giorgi lemma is similar to the elliptic and parabolic cases~\eqref{eq:structure-lemmas}, with kinetic versions of each step. 

\begin{lemma}[Kinetic second De Giorgi Lemma: reduction of oscillation]\label{lem:dg2-kin}
  Given $\Lambda>0$, there is $\nu \in (0,1)$, depending only on $\Lambda$ and the dimension $n$, so that the following holds:
  \smallskip
  
  Let $f$ be a weak solution on $Q_3$ to 
  \begin{equation*}
    \cT f = \nabla_v \cdot ( A \nabla_v f )
  \end{equation*}
  in the functional setting $L^2((-9,0) \times B_{27} ; H^1(B_3))$, where $A$ is symmetric measurable and so that
  \begin{equation*}
    \Lambda^{-1} \le A(t,x,v) \le \Lambda \ \text{ for almost every } (t,x,v) \in Q_3.
  \end{equation*}
  Then, we have 
  \begin{equation}
    \label{eq:dg2-kin}
    \mathrm{osc}_{Q_{\frac12}} f \le \nu \, \mathrm{osc}_{Q_2} f,
  \end{equation}
  where we denote $\mathrm{osc}_B f := \sup_B f - \inf_B f$.
\end{lemma}

Note again that this lemma requires $f$ to be \textbf{both} a subsolution and a supersolution, i.e. to be a solution. Let us show that Lemma~\ref{lem:dg2-kin} implies the Hölder regularity, which thus concludes the proof of Theorem~\ref{thm:dgnm4}.
\begin{corollary}
  Given $\Lambda>0$, there is $\alpha \in (0,1)$, depending only on $\Lambda$ and the dimension $n$, so that the following holds:
  \smallskip
  
  Let $f$ be a weak solution in $(0,T) \times \cU$, with $\cU \subset \R^{2d}$ and open set, to
  \begin{equation*}
    \cT f = \nabla_v \cdot \( A \nabla_v f \),
  \end{equation*}
  where $A$ is symmetric measurable and so that
  \begin{equation*}
    \Lambda^{-1} \le A(t,x,v) \le \Lambda \ \text{ for almost every } (t,x,v) \in (0,T) \times \cU.
  \end{equation*}

  Then, we have $f \in C^\alpha_{\mathrm{loc}}((0,T) \times \cU)$.
  \smallskip
  
  Moreover, given any $\delta \in (0,T/2)$,  there is $C>0$, depending on $\Lambda$, $\delta$ and the dimension $n$, so that for any $\tilde{\cU} \subset \subset \cU$ with $\text{dist}(\tilde \cU,\p \cU)\ge \delta$ and $\delta < T_0 < T_1 < T-\delta$, we have 
  \begin{equation}
    \label{eq:gain-reg-loc-kin}
    \|f\|_{C^\alpha((T_0,T_1) \times \tilde{\cU})} \le C \|f\|_{L^2(\cU)}.
  \end{equation}
\end{corollary}

\begin{proof}
  We already have that $f \in \mathrm{L}^\infty_{\mathrm{loc}}(\cU)$ by Lemma~\ref{lem:hypo-first} with the estimate
  \begin{equation}
    \label{eq:gain-localised-kin}
    \forall \, \hat{\cU} \subset \subset \cU, \quad \|f\|_{L^\infty((T_0/2,T) \times \hat{\cU})} \le C \|f\|_{L^2((0,T) \times \cU)}
  \end{equation}
  for a constant $C>0$ depending on $\Lambda$, $\text{dist}(\hat \cU,\p \cU)$, $T_0$ and the dimension $n$.
  
  Let $\mathcal{\tilde{U}} \subset\subset \mathcal{\hat U} \subset \subset \cU$ so that
  \begin{equation}
    \label{eq:lien-dist-kin}
    \text{dist}(\mathcal{\tilde{U}},\partial\cU) \le \frac12 \text{dist}(\mathcal{\tilde{U}},\partial\mathcal{\hat U}), \quad \text{dist}(\mathcal{\tilde{U}},\partial\cU) \in \left[\frac{\delta}{4},\frac{\delta}{2}\right], \quad \text{dist}(\mathcal{\tilde{U}},\partial\mathcal{\hat U}) \in \left[ \frac{\delta}{2},\delta\right].
  \end{equation}
  Given $(t_0,x_0,v_0) \in (T_0,T_1) \times \mathcal{\tilde{U}}$, we define the sequence of centred rescaled functions
  \begin{align*}
    & \forall \,  k \geq 1, \ \forall \, (s,y,w) \in Q_1, \quad \tilde g_k(s,y,w) := g_{k-1}\(\frac{s}{16},\frac{y}{64},\frac{w}{4}\) \\
    & \forall \, (s,y) \in Q_1, \quad g_0(s,y,w) := f\(t_0+\frac{\delta}{32} + \frac{\delta s}{16},x_0+\frac{\delta y}{64},v_0 +\frac{\delta w}{4}\).
  \end{align*}

  Note that each function $g_k$ is a rescaling of $f$ and solves
  \begin{equation*}
    \cT g_k = \nabla_x \cdot [A_k \nabla_x g_k]
  \end{equation*}
  in $Q_2$, with the diffusion matrix
  \begin{equation*}
    A_k(s,y,w) := A\(t_0+\frac{\delta}{32} + \frac{\delta s}{2^{4k}},x_0 + \frac{\delta y}{2^{6k}},v_0 + \frac{\delta w}{2^{2k}}\).
  \end{equation*}
  
  This new diffusion matrix is still measurable and satisfies
  \begin{equation*}
    \Lambda^{-1} \le A_k(s,y,w) \le \Lambda \ \text{ for almost every } (s,y) \in Q_2.
  \end{equation*}

  Therefore, we can apply Lemma~\ref{lem:dg2-kin} to each $g_k$ and deduce
  \begin{equation*}
    \forall \,  k \geq 1, \quad \mathrm{osc}_{Q_{\frac12}} g_k \, \le \, \nu \, \mathrm{osc}_{Q_2} g_k.
  \end{equation*}
  The scaling relation between successive terms in the sequence $g_k$ yields
  \begin{equation*}
    \forall \,  k \geq 1, \quad \mathrm{osc}_{Q_2} \, g_{k+1} = \mathrm{osc}_{Q_{\frac12}} g_k,
  \end{equation*}
  so we obtain by induction
  \begin{equation*}
    \mathrm{osc}_{Q_{4^{-k}}} g_1 \le \nu^{k-1} \mathrm{osc}_{Q_2} g_1 \le 2 \, \nu^{k-1} \, \|f\|_{L^\infty((\delta/2,T-\delta/2) \times \hat \cU)}.
  \end{equation*}

  Then, given a point
  \begin{equation*}
    (t,x,v) \in \(t_0- \frac{\delta}{16},t_0+ \frac{\delta}{16}\) \times B\(x_0,\frac{\delta}{64}\) \times B\(v_0, \frac{\delta}{4}\),
  \end{equation*}
  we pick up $k \ge 0$ so that
  \begin{equation*}
    \begin{cases}
      \displaystyle
      |t-t_0| \in \left[\delta 16^{-k-2},\delta 16^{-k-1}\right], \\[2mm]
      \displaystyle
      |x-x_0| \in \left[\delta 64^{-k-2},\delta 64^{-k-1}\right], \\[2mm]
      \displaystyle
      |v-v_0| \in \left[\delta 4^{-k-2},\delta 4^{-k-1}\right],
  \end{cases}
  \end{equation*}
  and write 
  \begin{equation*}
    |f(t,x,v) - f(t_0,x_0,v_0)|\le 2 \, \nu^{k-1} \, \|f\|_{L^\infty((\delta/2,T-\delta/2) \times \hat \cU)},
  \end{equation*}
  so that (using $|t-t_0| \ge \delta 16^{-k-2}$, $|x-x_0| \ge \delta 64^{-k-2}$ and $|v-v_0| \ge \delta 4^{-k-2}$)
  \begin{multline*}
    |f(t,x,v) - f(t_0,x_0,v_0)| \lesssim \\
    \left[ \( 16^{3\alpha} \nu \)^k |t-t_0|^{3\alpha} +\(64^{2\alpha} \nu\)^k \, |x-x_0|^{2\alpha} +\(4^{6\alpha} \nu\)^k \, |v-v_0|^{6\alpha} \right] \, \|f\|_{L^\infty((\delta/2,T-\delta/2) \times \hat \cU)}.
  \end{multline*}
  We then choose
  \begin{equation*}
    \alpha := - \frac{\ln \nu}{12\, \ln 2}
  \end{equation*}
  so that $16^{3\alpha} \nu = 64^{2\alpha} \nu = 4^{6\alpha} \nu =1$, and deduce $f$ is $C^\alpha$ at $(t_0,x_0,v_0)$. (In fact, note that the Hölder regularity exponent is $3\alpha$ in $t$, $2\alpha$ in $x$ and $6\alpha$ in $v$, which reflects the scaling of the equation). The $\alpha$ does \emph{not} depend on $(t_0,x_0,v_0)$, but only on the constant $\nu$ of Lemma~\ref{lem:dg2-kin}.

  We have therefore proved
  \begin{equation*}
    \|f\|_{C^\alpha((\delta,T-\delta) \times \tilde{\cU})} \le C \|f\|_{L^\infty((\delta/2,T-\delta/2,\hat \cU)}
  \end{equation*}
  with a constant depending only on $\nu$ and $\delta$. Combined with~\eqref{eq:gain-localised-kin} and~\eqref{eq:lien-dist-kin}, it implies~\eqref{eq:gain-reg-loc-kin} and concludes the proof.
\end{proof}

Lemma~\ref{lem:dg2-kin} is implied by the following kinetic decrease of supremum lemma, whose setting is represented in Figure~\ref{fig:kin-m2p}. The proof is exactly similar to the parabolic case and is not repeated.

\begin{figure}
  \includegraphics[scale=0.8]{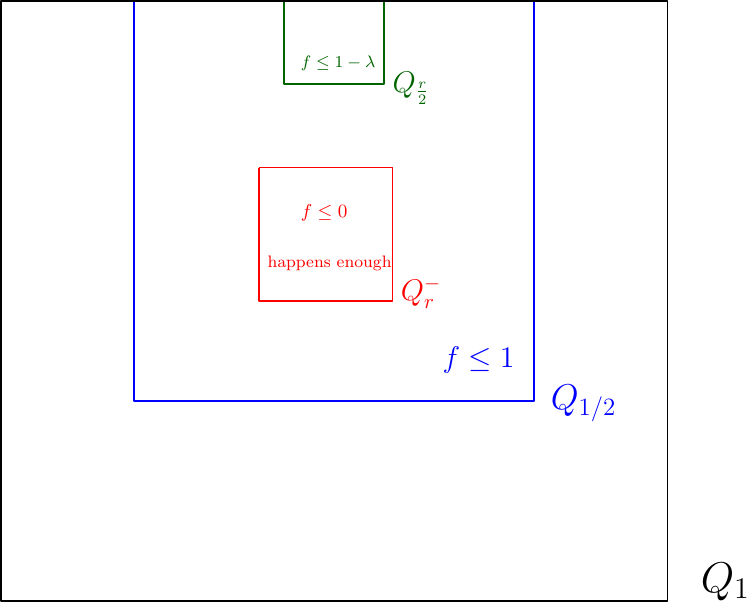}
  \caption{Setting of the decrease of supremum lemma and intermediate value lemma in the kinetic case.}
  \label{fig:kin-m2p}
\end{figure}
\begin{lemma}[Kinetic decrease of supremum lemma]
  \label{lem:kin-dec}
  Given $\Lambda >0$ and $\delta \in (0,1)$, there is $\lambda \in (0,1)$, depending only $\Lambda$, $\delta$ and the dimension $d$, so that the following holds:
  \smallskip

  Let $r =1/20$, and let $f$ be a weak subsolution over $Q_1$ to
  \begin{equation*}
    \p_t f + v \cdot \nabla_x f \le \nabla_v \cdot \( A \nabla_v f \),
  \end{equation*}
  with $A$ symmetric measurable so that
  \begin{equation*}
    \Lambda^{-1} \le A(t,x) \le \Lambda \ \text{ for almost every } (t,x,v) \in Q_1,
  \end{equation*}
  and that satisfies $f\le 1$ on $Q_{1/2}$ and the measure condition
  \begin{equation*}
    |\{f\le0\} \cap Q_{r}^-| \ge \delta |Q_{r}^-|,
  \end{equation*}
  where we have denoted
  \begin{equation}
    \label{eq:Qminus}
    Q_r^- := \left\{ -3 r^2 < t < -2r^2, \quad |x| < r^3, \quad |v|<r \right\}.
  \end{equation}
   
  Then, we have 
  \begin{equation*}
    \sup_{Q_{\frac{r}{2}}^-} f \le 1-\lambda.
  \end{equation*}
\end{lemma}

\begin{remark}
  By tracking the dependency of the constants in the proof, one finds $\lambda \approx \delta^{2(1+\delta^{-2})}$. This (bad) exponential dependency is due to the finite De Giorgi inductive argument used to deduce this lemma from the intermediate value lemma, where the number of steps potentially grows exponential as $\delta \to 0$. This lemma was first proved in~\cite{golse2019harnack} by non-constructive method, and then in~\cite{guerand2022quantitative,MR4653756} by constructive methods.
\end{remark}

Finally the kinetic decrease of supremum lemma is itself implied by a kinetic version of De Giorgi intermediate value lemma, whose setting is again represented in Figure~\ref{fig:kin-m2p}. The proof is exactly similar to the elliptic and parabolic cases and is not repeated.

\begin{lemma}[Kinetic De Giorgi intermediate value Lemma]
  \label{lem:kin-ivl}
  Given $\Lambda>0$ and $\delta_1,\delta_2 \in (0,1)$, there are $\theta, \nu \in (0,1)$, depending only on $\Lambda$, $\delta_1$, $\delta_2$ and the dimension $d$, so that the following holds:
  \smallskip
  
  Let $r = 1/20$, and let $f$ be a subsolution over $Q_2$ to
  \begin{equation*}
    \p_t f + v \cdot \nabla_x f \le \nabla_v \cdot \( A \nabla_v f \),
  \end{equation*}
  with $A$ symmetric measurable so that
  \begin{equation*}
    \Lambda^{-1} \le A(t,x) \le \Lambda \ \text{ for almost every } (t,x,v) \in Q_2,
  \end{equation*}
  and that satisfies $f \le 1$ on $Q_2$, and the measure conditions
  \begin{align}
    \label{eq:kin-mc1}
    & |\{f\le0\} \cap Q^-_r| \ge \delta_1 |Q_r^-|, \\[2mm]
    \label{eq:kin-mc2}
    & |\{f\ge1-\theta\} \cap Q_r| \ge \delta_2 |Q_r|
  \end{align}
  where $Q_r^-$ is defined as before, see~\eqref{eq:Qminus}.
  
  Then, we have
  \begin{equation*}
    |\{0<f<1-\theta\} \cap Q_1| \ge \nu |Q_1|.
  \end{equation*}
\end{lemma}
\begin{remarks}
  \begin{enumerate}
  \item By tracking the constants in the proof one finds
    $\theta \approx \delta_1 \delta_2$ and $\nu \approx \delta_1 \delta_2$
    for small $\delta_1,\delta_2$.
  \item This lemma was first proved
    in~\cite{golse2019harnack} by non-constructive method, and then
    in~\cite{guerand2022quantitative} by a constructive method.
  \item The gap in time is required, as shown by the following counter-example. The function $\mathbf{1}_{x+ct<a}$  is a subsolution on $Q_1$ for any $a \in \R$ and $c > 1$. If $Q^-_r$ and $Q_r$ were too close, a line of discontinuity of the form $x + ct = a$ could cross both and the previous subsolution would contradict the conclusion.
\end{enumerate}
\end{remarks}

We first give a proof by contradiction that extends the contradiction argument in the elliptic and parabolic cases, and then turn to a method based on trajectories and Poincaré-type inequalities.

\begin{proof}[Proof of the kinetic intermediate value lemma~\ref{lem:kin-ivl} by compactness]
  \text{ } \\
  This argument was first proposed in~\cite{golse2019harnack}, but is reminiscent of similar compactness arguments in the De Giorgi theory and in homogenisation theory. In this compactness argument, we can prove a slightly stronger statement where $\theta=1/2$.
  \medskip

  \noindent
  \textbf{A. The setting.} We argue by contradiction. Let us consider a contradiction sequence 
  $(f_k)_{k \ge 0}$ of subsolutions in $Q_2$ to 
  \begin{equation}
    \label{eq:fk}
    \cT f_k = \nabla_v \cdot (A_k \nabla_v f_k) - \mu_k
  \end{equation}
  for some non-negative measures $\mu_k$, such that $f_k \le 1$ for all $k \ge 1$ and, for some fixed $\delta_1, \delta_2 >0$:
  \begin{align*}
    & |\{ f_k \ge 1/2 \} \cap Q_r | \ge  \delta_1 |Q_r| \\
    & |\{ f_k \le 0 \} \cap Q^-_r |  \ge \delta_2 |Q_r ^-|\\
    &   |\{ 0 < f_k < 1/2\} \cap Q_1|  \to 0 \quad \text{ as } \quad k \to +\infty.
  \end{align*}

  The same assumptions still hold for the positive parts $(f_k)_+$ of each function $f_k$ in the sequence. 
  \medskip
  
  \noindent
  \textbf{B. A priori estimates on the contradiction sequence.} First, by using the local energy estimate~\eqref{eq:caccio-kin} (the kinetic Caccioppoli inequality) and the first De Giorgi lemma, we obtain uniform bounds on the $L^\infty_{\mathrm{loc}}(Q_2)$ and $(L^2_t H^1_x)_{\mathrm{loc}}(Q_2)$ norms of the $(f_k)_+$'s. Second, by integrating against a localisation function without squaring the unknown, we obtain a uniform bound on the local total variation of the $\mu_k$'s in $Q_2$.   
  \medskip

  \noindent \textbf{C. Limit by compactness.} Using standard compactness theorem, we have (still writing $f_k$ and $A_k$ and $\mu_k$ for the subsequences, in an abuse of notation)
  \begin{equation*}
    (f_k)_+ \stackrel{\ast}{\rightharpoonup} F  \text{ in } L^\infty_{\text{loc}} (Q_2)
  \end{equation*}
  and 
  \begin{equation}
    \label{conv:grad}
    \nabla_v (f_k)_+ \rightharpoonup \nabla_v F \quad \text{ and } \quad
      A_k \nabla_v (f_k)_+  \rightharpoonup H \quad 
    \text{ in } L^2_{\text{loc}} (Q_2) 
  \end{equation}
  for some weak limit $F$ and $H_1$, and
  \begin{equation*}
    \mu_k \rightharpoonup \mu
  \end{equation*}
  in the space of non-negative measure, thanks to the bound on the total variation.

  By weak limit, we deduce uniform bounds on the $L^2_tH^1_x(Q_1)$ norm of $F$, on the $L^2(Q_1)$ norm of $H$, and on the total variation of $\mu$, and we have the equation
  \begin{equation}
    \label{eq:F}
    \cT  F =   \nabla_v \cdot H - \mu.
  \end{equation}
  
  Thanks to the averaging lemma~\eqref{eq:golse}, we then improve the convergence of $(f_k)_+$ (taking if necessary a further subsequence), so that we have the strong convergence
  \begin{equation*}
    f_k^+ \to F \text{ in } L^p_{\text{loc}}(Q_2) \quad \text{
      for } 1 \le p < +\infty.
  \end{equation*}

  It implies the convergence in probability and thus the
  function $F$ satisfies 
  \begin{align}
    \label{eq:P1}
    & |\{ F \ge 1/2 \} \cap Q_r | \ge  \delta_1 |Q_r|\\
    \label{eq:P2}
    & |\{ F = 0 \} \cap \hat Q_r ^- |  \ge \delta_2 |Q_r ^-| \\
    \nonumber
    & \left|\{ 0 < F < 1/2\} \cap Q_1 \right|  = 0.
  \end{align}

  We now use the elliptic version of the intermediate value lemma (Lemma~\ref{lemma5}): for almost every $(t,x) \in (-1,0] \times B_1$,
  \begin{equation*}
    \left\{
      \begin{aligned}
        \text{either } \quad  &  \mbox{for almost every }  v \in B_1, 
        \quad F(t,x,v) \le 0, \\
        \text{ or } \quad \quad & \mbox{for almost every } v\in B_1, 
        \quad F(t,x,v) \ge \frac12.
      \end{aligned}
    \right.
  \end{equation*}

  In other words, $F (t,x,v) = F_0(t,x,v) \un_E (t,x)$ with $F_0 \ge 1/2$ and some measurable set $E \subset B_1 \times (-1,0)$. In view of \eqref{eq:P1} and
  \eqref{eq:P2}, this set $E$ satisfies
  \begin{equation}
    \label{eq:P3}
    \begin{cases}
      |E \cap (-r^2,0] \times B_{r^3}| >0 \\[2mm]
      |(-3r^2,-2r^2] \times B_{r^3}  \setminus E| >0.
    \end{cases}
  \end{equation}
  \medskip
  
  \noindent \textbf{D. Reaching a contradiction.} 
  We thus get from \eqref{eq:F} the following inequality
  \begin{equation*}
    \partial_t F + v \cdot \nabla_x F \le \nabla_v \cdot H
  \end{equation*}
  in $Q_{3/2}$. Then, choosing $v_0 \in B_1$, and averaging locally around $v_0$ with a smooth mollifier $\zeta$ we get (since $F$ only depends on $(t,x)$)
  \begin{equation}
    \label{eq:loc-up}
    \partial_t F + v_0 \cdot \nabla_x F 
    \le \int_{\R^d} \Big[ \left|H (x,v,t) \nabla_v \zeta(v-v_0)\right| 
    \Big] \dd v
  \end{equation}
  in $(t,x) \in B_1 \times (-1,0)$.  Since 
  $H \in L^2_{\text{loc}} (Q_2)$, we deduce, for 
  $v_0 \in B_1$ and $(t,x) \in B_1 \times (-1,0)$, that 
  \begin{equation}
    \label{eq:propagation}
    \partial_t F + v_0 \cdot \nabla_x F \le 0 \ \text{ as long as } F(t,x,\cdot) = 0.
  \end{equation}
  This is due to the fact that $F$ cannot jump from $0$ to $1/2$ without the left-hand side producing a positive Dirac mass, which is not allowed by the $L^2_{\mathrm{loc}}$ upper bound~\eqref{eq:loc-up}. 

  The rest of the argument is similar to the parabolic case: since $F(t,\cdot,\cdot)$ is almost everywhere non-positive for some times in the past cylinder (with non-zero measure), the previous inequality shows that it cannot become positive, which contradicts $F(t,\cdot,\cdot)$ being almost everywhere greater than $1/2$ for some times in the future cylinder (with non-zero measure). 
\end{proof}

We now turn to the constructive method. We will prove that the kinetic intermediate value lemma~\ref{lem:kin-ivl} is implied, with quantitative constants, by the following ``kinetic Poincar\'e inequality''. The setting of the kinetic Poincaré inequality is again given by Figure~\ref{fig:kin-m2p}.

The overall structure of the constructive proof of the kinetic second De Giorgi lemma is therefore as follows:
\begin{equation*}
  \boxed{
  \begin{array}{l}
     \text{Kinetic Poincaré inequality } \Longrightarrow \text{ Kinetic intermediate value lemma } \\[2mm]
    \Longrightarrow \text{ Kinetic decrease of the supremum } \Longrightarrow \text{ Kinetic reduction of oscillation}.
  \end{array}
  }
\end{equation*}
Note that the coercivity of the diffusion matrix $A$ is not required for the proof of the Poincaré inequality, but only for the later steps.

\begin{theorem}[Kinetic Poincaré inequality]
  \label{thm31}
  Given $\Lambda>0$, there is $C>0$, depending only on $\Lambda$ and the dimension $d$, and there is $r \in (0,1)$, so that the following holds:
  \smallskip
  
  Let $f \ge 0$ be a weak subsolution in $Q_1$ to
  \begin{equation*}
    \p_t f + v \cdot \nabla_x f \le \nabla_v \cdot \( A \nabla_v f \)
  \end{equation*}
  with $A$ symmetric measurable so that
  \begin{equation*}
    A(t,x,v) \le \Lambda \ \text{ for almost every } (t,x,v) \in Q_1.
  \end{equation*}

  Then $f$ satisfies the following inequality
  \begin{equation}
    \label{gmpoi}
    \int_{Q_r ^+} \(f- \langle f\rangle_{Q^-_r} \)_+ \dd z \le C \int_{Q_1} |\nabla_v f| \dd z,
  \end{equation}
   where $Q_r^-$ is defined as before, see~\eqref{eq:Qminus}, and the average $\langle f\rangle_{Q^-_r}$ is defined by
   \begin{equation*}
     \langle f\rangle_{Q^-_r} := \frac{1}{|Q^-_r|} \int_{Q^-_r} f \dd z.
   \end{equation*}
\end{theorem}

\begin{remarks}
  \begin{enumerate}
  \item This kinetic Poincaré inequality was developed along the series of works~\cite{guerand2022quantitative,niebel2023kinetic,anceschi2024poincare,dmnz}. The two key inspirations were first a weaker modified form of such inequality proposed before in~\cite{wang2009c}, albeit with a complicated corrector replacing the average $\langle f\rangle_{Q^-_r}$ (this corrector itself solving another hypoelliptic problem), and second the kinetic Poincaré inequality for constant coefficients obtained in~\cite{MR4776290}. 
  \item Intuitively this theorem measures how much (in $L^1$) the
    solution $f$ can grow above its past average over $Q_r^-$, in terms of a bound on $\nabla_v f$ in $L^1$ in a bigger cube \emph{in the
      present}, provided that $f$ is a subsolution on such a bigger
    cube. This is therefore a \textbf{time-oriented integral control on the oscillation}. For instance if $\nabla_v f \equiv 0$, then $f$ solves $\cT f \le 0$ and decays along the transport lines: it cannot increase its local average as time flows from past to future.
  \item This theorem only requires the bound from above on $A$, not the coercivity. However, in the sequel we will use the coercivity of $A$ in order to control the $L^2$ norm of $\nabla_v f$ in the right hand side of the previous display by the $L^2$ norm of $f$ via the energy estimate. 
    
  \item Notice the gap between the two small cubes. It is required for subsolutions, since the kinetic Poincaré inequality~\eqref{gmpoi} implies the kinetic intermediate value Lemma~\ref{lem:kin-ivl}, which fails for subsolutions when no gap is present.
    
  \item The characterisation of the optimal constant in~\eqref{gmpoi} is an interesting question, open to our knowledge.
\end{enumerate}
\end{remarks}

\begin{proof}[Proof of the intermediate value lemma~\ref{lem:kin-ivl} from the Poincaré inequality~\eqref{gmpoi}]
  \text{ } \\
  Given $f$ as in the statement of the intermediate value lemma~\ref{lem:kin-ivl}, its positive part $f_+$ is also a subsolution verifying exactly the same assumptions as $f$.

  We can therefore apply the Poincaré inequality~\eqref{gmpoi} to $f_+$:
  \begin{equation}
    \label{eq:appli-poincare}
    \int_{Q_r} \(f_+ - \langle f_+\rangle_{Q_r^-}\)_+ \dd z 
    \lesssim \int_{Q_1} |\nabla_vf_+| \dd z.
  \end{equation}

  We now study the average on $Q_r^-$ and use the first measure condition~\eqref{eq:kin-mc1}:
  \begin{equation*}
    \langle f_+\rangle_{Q_r^-} = \frac{1}{|Q^-_r|} \int_{Q^-_r} f_+ \dd z \le \frac{|\{f>0\} \cap Q_r^- |}{|Q_r^-|} \le 1-\delta_1.
  \end{equation*}

  We deduce the following lower bound on the left hand side of~\eqref{eq:appli-poincare}:
  \begin{align*}
    \int_{Q_r} \(f_+ - \langle f_+\rangle_{Q_{r}^-}\)_+ \dd z 
    & \ge \int_{Q_r} \(f_+ - (1-\delta_1)\)_+ \dd z \\
    & \ge \int_{Q_r\cap\{f>1-\theta\}} (\delta_1-\theta) \dd z \\
    & \ge |Q_r| \delta_2(\delta_1-\theta).
  \end{align*}
  
  Then, we bound from above the right-hand side of~\eqref{eq:appli-poincare}:
  \begin{align*}
    \int_{Q_1} |\nabla_vf_+| \dd z & = 
    \int_{Q_1 \cap\{f \le 0\}} |\nabla_vf_+| \dd z \\
    & \quad + \int_{Q_1 \cap\{1-\theta>f>0\}} |\nabla_vf_+| \dd z + \int_{Q_1 \cap\{f\ge 1-\theta\}} |\nabla_vf_+| \dd z \\ 
    & =: I_1 + I_2 + I_3.
  \end{align*}

  We now consider the three terms $I_1$, $I_2$ and $I_3$ successively. First, we have $I_1 = 0,$ since $\nabla_v f_+ =0$ almost everywhere over $\{f_+=0\}$, see again~\cite[Section~4.2.2, Theorem~4-(iii), p.~129]{zbMATH08010281}. Second, we have
   \begin{align*}
       I_2 & \lesssim |\{ 0 < f <1-\theta\} \cap Q_1|^{1/2} \left(  \int_{Q_1} |\nabla_vf_+|^2 \dd z \right)^{1/2} \\
       & \lesssim |\{ 0 < f <1-\theta\} \cap Q_1|^{1/2} \left(  \int_{Q_2} f_+^2 \dd z \right)^{1/2} \\
       & \lesssim |\{ 0 < f <1-\theta\} \cap Q_1|^{1/2},
   \end{align*}
   where we have used the energy estimate~\eqref{eq:caccio-kin} (the kinetic Caccioppoli inequality) between the first and second lines, and the bound $0 \le f_+ \le 1$ between the second and third lines.
   
   Third and finally, we have
   \begin{align*}
       I_3 & = \int_{Q_1} \left|\nabla_v \Big[\big(f-(1-\theta)\big)_+ + (1-\theta) \Big]\right| \dd z  \\
       & = \int_{Q_1} |\nabla_v (f-(1-\theta))_+ | \dd z \\ 
       & \lesssim \left( \int_{Q_1} |\nabla_v (f-(1-\theta))_+ |^2 \dd z \right)^{1/2} \\
       & \lesssim \left( \int_{Q_2} |(f-(1-\theta))_+ |^2 \dd z \right)^{1/2} \lesssim \theta,
   \end{align*}
   where we have used the energy estimate~\eqref{eq:caccio-kin} (the kinetic Caccioppoli inequality) between the third and fourth lines on the subsolution $(f-(1-\theta))_+$, and the bound $0 \le (f-(1-\theta))_+ \le \theta$ in the last line.
   
   Altogether, we obtain
   \begin{equation*}
     \delta_2(\delta_1-\theta) \lesssim |\{0<f<1-\theta\} \cap Q_1|^{1/2} + \theta.
   \end{equation*}
   Provided we choose $\theta \le \delta_1/2$, we deduce
   \begin{equation*}
     \delta_1\,\delta_2 \lesssim |\{0<f<1-\theta\} \cap Q_1|^{1/2} + \theta.
   \end{equation*}
   Provided we furthermore impose $\theta \lesssim \delta_1\delta_2$, we finally get
   \begin{equation*}
     \delta_1\,\delta_2 \lesssim |\{0<f<1-\theta\} \cap Q_1|^{1/2},
   \end{equation*}
   from which the conclusion follows. 
\end{proof}

\begin{proof}[Proof of Theorem~\ref{thm31} (the kinetic Poincaré inequality)]
  \text{ } \\
  We present the argument of \cite{anceschi2024poincare,dmnz}, inspired by \cite{guerand2022quantitative}, the subsequent work \cite{niebel2023kinetic} and control theory. The idea is to replace the straight lines of the constructive argument in the elliptic case by appropriate trajectories generated by the Hörmander vector fields $\p_t + v \cdot \nabla_x$ and $\nabla_v$, and use the \emph{upper bound} on $(\p_t + v \cdot \nabla_x)f \le \nabla_v \cdot (A \nabla_v f)$ provided by the assumption that $f$ is a subsolution; the right-hand side is then controlled after integrating by parts but the Jacobian of the parametrisation by the \emph{departing} cylinder must be controlled. We also introduce an intermediate cylinder $Q^0_r$ between $Q_r^-$ and $Q_r^+$ and a smooth localisation function in this intermediate cylinder: this  variable $z_0$ is used for the integration by parts (see Figure~\ref{fig:trajectories}).

  We recall the notation $z=(t,x,v)$.

    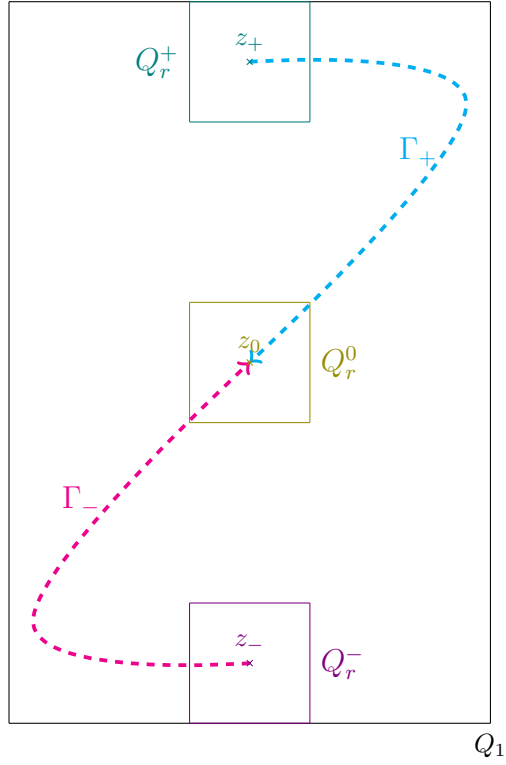
\begin{figure}{
      \centering
      \begin{tikzpicture}[scale =0.8]
        \draw [black] (-4,0) -- (4,0);
        \draw[black] (-4,-12) -- (4,-12);
        \draw [black](-4,0) -- (-4,-12);
        \draw [black](4,0) -- (4,-12) node[anchor=north, scale=1]
        {\footnotesize{$Q_1$}};
        \draw [teal](-1,0) -- (1, 0);
        \draw [teal](-1,-2) -- (1, -2);
        \draw [teal](-1,0) -- node[anchor=east, scale=1] {$Q^+_r$} (-1,
        -2);
        \draw [teal](1,0) -- (1, -2) ;
        \draw[olive] (-1,-5) -- (1, -5);
        \draw [olive](-1,-7) -- (1, -7);
        \draw [olive](-1,-5) -- (-1, -7);
        \draw [olive](1,-5) -- node[anchor=west, scale=1] {$Q^0_r$}(1,
        -7);
        
        \draw[violet] (-1,-10) -- (1, -10);
        \draw [violet](-1,-12) -- (1, -12);
        \draw [violet](-1,-10) -- (-1, -12);
        \draw [violet](1,-10)  -- node[anchor=west, scale=1] {$Q^-_r$}
        (1, -12);
        \draw [->, cyan, dashed, line width=0.5mm]  plot [smooth,
        tension = 0.7] coordinates {(0,-1)(3.6, -1.7)  (0,-6)};
        \node [cyan] at (2.8, -2.5) {$\Gamma_+$};
        \draw [->, magenta, dashed, line width=0.5mm] plot [smooth,
        tension = 0.7] coordinates {(0,-11)(-3.6, -10.3) (0,-6)};
        \node [magenta] at (-2.8, -8.3) {$\Gamma_-$};
        \coordinate[label={[teal]above:\footnotesize{$z_+$}}] (z_+) at
        (0, -1);
        \draw [teal] plot[only marks,mark=x,mark size=2pt]
        coordinates {(0, -1)};
        \coordinate[label={[olive]above:\footnotesize{$z_0$}}] (z_0)
        at (0, -6);
        \draw [olive] plot[only marks,mark=x,mark size=2pt]
        coordinates {(0, -6)};
        \coordinate[label={[violet]above:\footnotesize{$z_-$}}] (z_-)
        at (0, -11);
        \draw [violet] plot[only marks,mark=x,mark size=2pt]
        coordinates {(0, -11)};
      \end{tikzpicture}
    }
    \caption{Construction of the trajectories. The curve $\Gamma_+$
      connects any point $z_+ \in Q^+_r$ to some intermediate point
      $z_0 \in Q^0_r$, whereas the curve $\Gamma_-$ connects any point
      $z_- \in Q^-_r$ to some intermediate point
      $z_0 \in Q^0_r$.}\label{fig:trajectories}
  \end{figure}
  \medskip

  \noindent
  \textbf{A. Construction of the trajectories.} 
  Given three points
  $z_+ \in Q^+_r, z_0 \in Q^0_r$ and $z_- \in Q^-_r$, we want to
  construct two paths $s \rightarrow \Gamma_+(s)$ and
  $s \rightarrow \Gamma_-(s)$ for $s \in [0, 1]$ such that (see
  Figures~\ref{fig:trajectories} and~\ref{fig:trajectories-toy})
  \begin{equation*}
    \begin{aligned}
      & \Gamma_+(s) = \left(T_+(s),X_+(s), V_+(s)\right), \qquad \Gamma_+(0) = z_+, \qquad \Gamma_+(1) =
      z_0,\\
      & \Gamma_-(s) = \left(T_-(s), X_-(s), V_-(s)\right), \qquad \Gamma_-(0) = z_-, \qquad \Gamma_-(1) =
      z_0,
    \end{aligned}
  \end{equation*}
  where
  \begin{equation}
    \label{eq:control}
    \begin{aligned}
      \begin{cases}
        \ds V_{\pm}(s) 
        = \mathsf m_\pm ^{(0)}
        g_0''(s) +\mathsf m^{(1)}_\pm
        g_1 ''(s), \\[2mm]
        \ds X_\pm(s) = \delta_\pm V_\pm(s),\\[2mm]
        \ds T_\pm(s) = \delta_\pm,
      \end{cases}
    \end{aligned}
  \end{equation}
  for $\delta_\pm := t_0 - t_\pm$, some control functions
  $g_i \in C^2((0,1])$ with $g_i(0)=g_i'(0)=0$ for $i=0,1$, and
  constant vectors $\mathsf m^{(i)} _\pm \in \R^d$ for $i=0,1$. Note that the forcing term is independent of the position and velocity. This makes it easy to solve these differential equations.
  \begin{figure}{ \centering
      \begin{tikzpicture}[scale = 0.45,>=stealth]
        \draw[->] (xyz cs:x=-2) -- (xyz cs:x=10) node[above] {$v$};
        \draw[->] (xyz cs:y=-12) -- (xyz cs:y=2) node[right] {$t$};
        \draw[->] (xyz cs:z=-2) -- (xyz cs:z=26) node[left] {$x$};
        \node[fill, circle,inner
        sep=1.5pt,label={left:\footnotesize{$\Gamma_+(0)$}}] at
        (xyz cs:x=0,z=0, y = 0) {};
        \node[fill,circle,inner
        sep=1.5pt,label={right:\footnotesize{$\Gamma_+(1)$}}] at
        (xyz cs:x=4,z=4, y = -10) {};
        \draw[smooth, ->, red, line width=0.45mm] plot coordinates
        {(xyz cs:x=0,z=0, y = 0) (xyz cs:x=-5.56,z=0.732, y = -2.5)
          (xyz cs:x=-7.16,z=2.42, y = -5)  (xyz cs:x=-3.94,z=3.91,
          y = -7.5)  (xyz cs:x=4,z=4, y = -10)};
        \begin{scope}[on background layer]
          \draw[smooth, ->, blue] (xyz cs:x=-7.16,z=2.42, y = -5) --
          node[above] {\tiny{$\nabla_v$}}(xyz cs:x=-4.16, z=2.42,
          y = -5);
          \draw[smooth, ->, blue] (xyz cs:x=-7.16,z=2.42, y = -5) --
          node[left] {\tiny{$\cT$}}(xyz cs:x=-7.16, z=4.92, y =
          -7.5);
          \draw[smooth, ->, blue] (xyz cs:x=-7.16,z=2.42, y = -5) --
          (xyz cs:x=-4.16, z=4.92, y = -7.5)node[right] {\tiny{$\dot
              {\Gamma}_+(s)$}};
          \draw[dashed, blue](xyz cs:x=-7.16, z=4.92, y = -7.5) --
          (xyz cs:x=-4.16, z=4.92, y = -7.5);
          \draw[dashed, blue](xyz cs:x=-4.16, z=2.42, y = -5)--
          (xyz cs:x=-4.16, z=4.92, y = -7.5);
        \end{scope}
      \end{tikzpicture}
    }
    \caption{Construction of the trajectories. The curve $\Gamma_+$
      connects any point $\Gamma_+(0) = z_+ \in Q^+_r$ to some
      intermediate point $\Gamma_+(1) = z_0 \in Q^0_r$ along the
      vector fields $\cT$ and $\nabla_v$. The tangent vector $\dot \Gamma_+(s)$ is in the span of $\cT$ and $\nabla_v$ (the third axis, for the $x$ variable, is orthogonal to the plane of the paper in this schematic representation).} \label{fig:trajectories-toy}
  \end{figure}
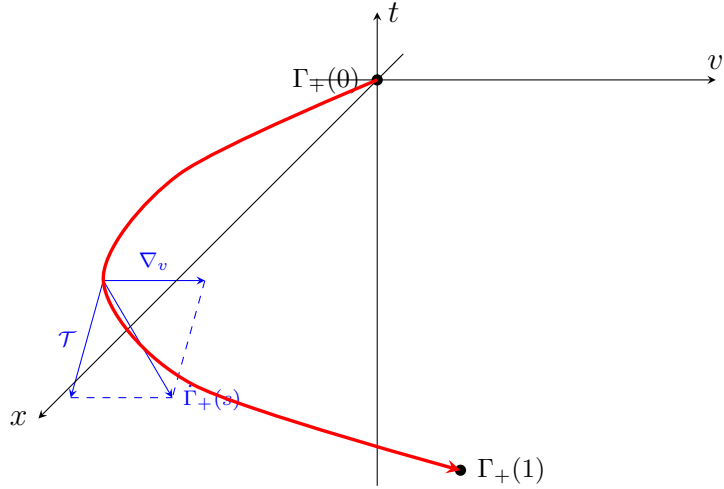

  Solving these differential equations then yields
  \begin{equation}
    \label{eq:control-solved}
    \begin{aligned}
      \begin{cases}
        V_{\pm}(s) = v_\pm + \mathsf m_\pm ^{(0)}
        g_0'(s) + \mathsf m^{(1)}_\pm
        g_1 '(s), \\[2mm]
        X_\pm(s) = x_\pm + s \delta_\pm v_\pm +
        \delta_\pm \left[ \mathsf m_\pm ^{(0)} g_0(s) + \mathsf
          m^{(1)}_\pm
          g_1 (s) \right] ,\\[2mm]
        T_\pm(s) = s t_0 + (1-s) t_\pm.
      \end{cases}
    \end{aligned}
  \end{equation}
  
  Let us denote $\U_d$ the $d$-vector of $1$'s, $\V_d$ the
  $d$-vector of $0$'s, $\I_d$ the $d\times d$ identity matrix,
  $\OO$ a zero block of arbitrary size, and
  \begin{align*}
    & \mathsf M_\pm :=
      \begin{pmatrix}
        \mathsf m^{(0)}_\pm \\
        \mathsf m^{(1)} _\pm
      \end{pmatrix}, \quad
    \sW(s) :=
    \begin{pmatrix}
      g_0'(s) \I_d & g_1 '(s) \I_d \\
      g_0(s) \I_d & g_1(s) \I_d
    \end{pmatrix}, \\
    & \sW_\pm ^\delta (s) :=
      \begin{pmatrix}
        \I_d & \OO \\
        \OO & \delta_\pm \I_d
      \end{pmatrix} \sW(s), \quad
              \mathsf Y_\pm =
              \begin{pmatrix}
                \mathsf y^{(0)} _\pm \\
                \mathsf y^{(1)} _\pm
              \end{pmatrix}
    :=
    \begin{pmatrix}
      v_0 - v_\pm \\
      x_0 - x_\pm - \delta_\pm v_\pm
    \end{pmatrix}.
  \end{align*}
  The boundary conditions $\Gamma_\pm(1)=z_0$ impose
  \begin{align}\label{eq:boundary-condition}
    \sW_\pm ^\delta(1) \mathsf M_\pm = \mathsf Y_\pm \quad \Longrightarrow
    \quad \mathsf M_\pm = \sW_\pm ^\delta(1)^{-1} \mathsf Y_\pm,
  \end{align}
  provided that the Wronskian matrix $\sW(1)$ is invertible at
  $s=1$. We deduce
  \begin{align*}
    \begin{pmatrix}
      V_\pm (s) \\
      X_\pm (s)
    \end{pmatrix}
    =   \sW_\pm ^\delta(s) \mathsf M_\pm +
    \begin{pmatrix}
      \I_d & \OO \\
      \delta_\pm s \I_d & \I_d
    \end{pmatrix}
    \begin{pmatrix}
      v_\pm\\
      x_\pm
    \end{pmatrix}.
  \end{align*}
  Using~\eqref{eq:boundary-condition} we obtain
  \begin{align}
    \nonumber
    \begin{pmatrix}
      V_\pm (s) \\
      X_\pm (s)
    \end{pmatrix}
    &= \sW_\pm ^\delta(s) \sW_\pm ^\delta(1)^{-1} \mathsf Y_\pm+
      \begin{pmatrix}
        \I_d & \OO \\
        \delta_\pm s \I_d & \I_d
      \end{pmatrix}
                            \begin{pmatrix}
                              v_\pm\\
                              x_\pm
                            \end{pmatrix}\\
    \nonumber
    &= W^\delta_\pm(s) \big[W^\delta_\pm(1)\big]^{-1}
      \begin{pmatrix}
        v_0\\
        x_0
      \end{pmatrix}\\
    \nonumber
    &\quad- W^\delta_\pm(s) \big[W^\delta_\pm(1)\big]^{-1}
      \begin{pmatrix}
        v_\pm \\ x_\pm + \delta_\pm v_\pm
      \end{pmatrix}
    +
    \begin{pmatrix}
      \I_d & \OO \\
      \delta_\pm s \I_d & \I_d
    \end{pmatrix}
                          \begin{pmatrix}
                            v_\pm\\
                            x_\pm
                          \end{pmatrix} \\
    \label{eq:mapping}
    & = \mathfrak A^s
      \begin{pmatrix}
        v_0 \\ 
        x_0
      \end{pmatrix} + \mathfrak B^s 
    =: \Phi^s_\pm (x_0,v_0)
  \end{align}
  is an affine function with matrix
  $\mathfrak A^s:=\sW_\pm ^\delta (s) [\sW_\pm ^\delta (1)]^{-1}$
  and a vector $\mathfrak B^s$ that depend only on $s$ and $(x_\pm,v_\pm)$.

  We still have to prove that the matrix $\mathfrak A^s$ is
  invertible for $s \in (0,1]$. If so, given $s \in (0,1]$, the
  derivative along the second variable $v_\pm(s)$ of the
  inverse of $\Phi^s$ is
  \begin{equation*}
    \nabla_{v_0} \left( \Phi_\pm ^s \right)^{-1} =
    \sW^\delta_\pm(1)\big[\sW^\delta_\pm(s)\big]^{-1}
    \left( \begin{matrix}
        \U_d \\ \V_d
      \end{matrix} \right).
  \end{equation*}

  We choose, following a new idea in the recent contribution~\cite{dmnz},
  \begin{equation*}
    \begin{cases}
      \displaystyle
      g_0(s) := s^{\frac32} \cos \ln s,\\[2mm]
      \displaystyle
      g_1(s) := s^{\frac32} \sin \ln s.
    \end{cases}
  \end{equation*}
  Then the Wronskian matrix is invertible for all $s \not =0$ with the simple formula
  \begin{equation}
      \label{toy-jacobian}
      \det \sW(s)
      = \left(g_0'(s) g_1(s)- g_0(s) g_1'(s)\right)^d = s^{2d}.
  \end{equation}

  Our choice of two linearly control functions $g_0$
  and $g_1$ is guided by: (1) obtaining the critical exponent $-1/2$ for the forcings $g_0''$ and $g_1''$ (scaling condition), and (2) ensuring that the Wronskian matrix is invertible (independence condition). We refer to~\cite{dmnz} for a more detailed discussion of these aspects.

  Moreover, we have 
  \begin{equation*}
    \sW_\pm ^\delta (s) :=
    \begin{pmatrix}
      s^{\frac12} \left(\frac32 \cos \ln s - \sin \ln s \right) \, \I_d & s^{\frac12} \left(\frac32 \sin \ln s + \cos \ln s \right) \, \I_d \\
      s^{\frac32} \cos \ln s \, \I_d &
      s^{\frac32} \sin \ln s \, \I_d 
    \end{pmatrix}
  \end{equation*}
  which yields
  \begin{equation}
    \label{eq:bound-jac-kin}
    \nabla_{v_0} \left( \Phi_\pm ^s \right)^{-1}
    = W^\delta_\pm(1)\big[W^\delta_\pm(s)\big]^{-1}
    \begin{pmatrix}
      \U_d \\ \V_d
    \end{pmatrix}
    \sim O\( s^{-\frac12} \).
  \end{equation}
  Note that $\nabla_{v_0} ( \Phi_\pm ^s )^{-1}$ remains integrable on $s \in [0,1]$ will be used in the next step.

  Observe also that our choice of control functions $g_0$ and $g_1$ implies that their derivatives up to order two are integrable on $s \in [0,1]$, which implies that the trajectories are bounded (with integrable tangent vector field). This allows to choose $r \in (0,1)$ so that the trajectories we have constructed between points of $Q_r^\pm$ and $Q^0_r$ stays in the larger base cylinder $Q_1$, provided that $r \in (0,1)$ is chosen small enough.

  Note finally that the scalings $\nabla_{v_0} ( \Phi_\pm ^s )^{-1} \sim O(s^{-1/2})$ and $g_i''(s) = O(s^{-1/2})$ are optimal, as discussed in~\cite{dmnz}. Trajectories achieving this scaling are called \textbf{critical kinetic trajectories} in~\cite{dmnz}.
  \medskip

  \noindent
  \textbf{B. Proof of the kinetic Poincaré inequality by following the trajectories.} We consider the three cylinders $Q^+,Q^-, Q^0 \subset \Omega$ as in Figure~\ref{fig:trajectories}, and our subsolution $f$. Let $\varphi \in C_c^\infty(\R^{2d})$ be a non-negative function of the variables $(x,v)$ (excluding the time component) with compact support in any time-slice of $Q^0$ and such that
  \begin{equation*}
    \frac{1}{|Q_0|} \int_{Q^0} \varphi \dd z = 1.
  \end{equation*}
  
  Then
  \begin{align*}
       \int_{Q^+} \Big(f(z_+) - \langle f\rangle_{Q^-}\Big)_+ \dd 
      z_+ 
      & \le \frac{1}{|Q^-|} \int_{Q^+} \int_{Q^-} \Big(f(z_+) - f(z_{-})
      \Big)_+ \dd z_{-} \dd z_+ \\
      & \le \frac{1}{|Q^-|} \int_{Q^+} \int_{Q^-} \Big(f(z_+) - \langle
      f\varphi\rangle_{Q^0}\Big)_+\dd z_{-} \dd z_+ \\
      &\quad + \frac{1}{|Q^-|} \int_{Q^+} \int_{Q^-} \Big( \langle
      f\varphi\rangle_{Q^0} - f(z_{-})\Big)_+ \dd z_{-}\dd z_+,
  \end{align*}
  which yields the following estimate
  \begin{align}
    \nonumber
      & \int_{Q^+} \Big(f(z_+) - \langle f\rangle_{Q^-}\Big)_+ \dd 
        z_+ \\
    \nonumber
      & \leq \frac{1}{|Q^0|} \int_{Q^+} \Bigg\{\underbrace{\int_{Q^0} \big(f(z_+)
        - f(z_0)\big)\varphi(x_0,v_0) \dd z_0}_{=: \cI^+}
        \Bigg\}_+ \dd z_+ \\
    \label{eq:poincare-aux}
      & \quad + \frac{|Q^+|}{|Q^0| |Q^-|} \int_{Q^-} \Bigg\{\underbrace{\int_{Q^0}
        \big(f(z_0) - f(z_{-})\big)\varphi(x_0,v_0) \dd z_0}_{=:
        \cI^-}\Bigg\}_+ \dd z_{-},
  \end{align}
  where $z_+ = (t_+x_+,v_+), z_- = (t_-,x_-,v_-)$ and
  $z_0 = (t_0,x_0,v_0)$.  Note that $t_{-} < t_0 < t_+$. We now use
  the previous trajectories to estimate the right hand side.
  
  Using the chain rule, we get
  \begin{align*}
    \cI^\pm & = \pm \int_{Q^0} \Big(f(z_\pm) -
      f(z_0)\Big) \, \varphi(x_0,v_0) \dd z_0 \\
      & = \mp \int_{Q^0} \int_0^1 \frac{\dd}{\dd
        s}f\big(\Gamma_\pm(s)\big) \, \varphi(x_0,v_0)\dd s\dd z_0 \\
      & = \mp\delta_\pm  \int_{Q^0} \int_0^1 \left( \mathcal T
        f \right)\big(\Gamma_\pm(s)\big) \, \varphi(x_0,v_0)\dd s\dd z_0 \\
      & \qquad \mp \int_{Q^0} \int_0^1  V'_\pm(s) \cdot \left(\nabla_v f\right)(\Gamma_\pm(s)) \, \varphi(x_0,v_0)\dd s \dd
      z_0,
  \end{align*}
  which yields the following estimate
  \begin{align}
    \nonumber
    \cI^\pm
    & \leq  \underbrace{\pm\delta_\pm \int_{Q^0} \int_0^1
      \left[ \nabla_v ( A \nabla_v  f ) \right]
      \big(\Gamma_\pm(s)\big) \, \varphi(x_0,v_0)\dd s\dd z_0}_{=:
      \cI^\pm_1} \\
    \label{eq:I1-I2}
    & \qquad  \underbrace{\mp\int_{Q^0} \int_0^1 \sum_{i =
      0}^1 g_i''(s) \mathsf m_\pm^{(i)}
      \cdot \left( \nabla_v f \right)(\Gamma_\pm(s)) \,
      \varphi(x_0,v_0)\dd s \dd z_0}_{=: \cI^\pm_2}.
  \end{align}
  Note that the only difference between the two terms $\cI^+$ and
  $\cI^-$ is the role of $z_0$: in the former case it is the past
  variable, in the latter it is the future variable.
  
  We now integrate by parts the terms $\cI^\pm _1$ after a
  change of variables $(x_0,v_0) \mapsto y:= \Phi^s_\pm(x_0,v_0)$ for
  $s,t_0$ fixed. We recall that $\Phi^s_\pm$ is the affine map defined in \eqref{eq:mapping}, and that it is invertible for $s \not =0$.
  \begin{equation*}
    \begin{aligned}
      & \cI_1^\pm  = \pm\delta_\pm \int_{Q^0} \int_0^1
      \left[  \nabla_v \cdot ( A \nabla_v  f ) \right]
      \big(\Gamma_\pm(s)\big) \, \varphi(x_0)\dd s\dd z_0 \\
      & = \pm\delta_\pm \int_{Q^0} \int_0^1
      \left[  \nabla_v \cdot ( A \nabla_v  f ) \right]
      \left( \Phi^s_\pm (x_0),st_0 + (1-s)t_\pm
      \right) \, \varphi(x_0)\dd s \dd t_0 \dd x_0 \dd v_0  \\
      & = \pm \delta_\pm \int_{(\Phi^s_\pm \otimes \I)(Q^0)} \int_0^1
      \left[  \nabla_v \cdot ( A \nabla_v  f ) \right]
      \left( y,st_0 + (1-s)t_\pm \right) \, \varphi
      \left( \left(\Phi^s_\pm\right)^{-1}(y)\right) \frac{\dd s \dd t_0
        \dd y }{\left| \det \mathfrak
          A^s \right|} \\
      & = \pm \delta_\pm \int_{(\Phi^s_\pm \otimes
        \I)(Q^0)} \int_0^1 \left[ A \nabla_v f\right]
      \left( y,st_0 + (1-s)t_\pm \right) \, \nabla_v \left[\varphi
        \left(\left( \Phi^s_\pm \right)^{-1} (y)\right)\right]
      \frac{\dd s \dd t_0 \dd y}{\left| \det \mathfrak A^s \right|} \\
      & = \pm \delta_\pm \int_{Q^0} \int_0^1
      \left[ A \nabla_v f\right]
      \left( \Gamma_\pm(s) \right) \, \left\{ \nabla_v \left[\varphi
          \left(\left( \Phi^s_\pm \right)^{-1} (y)\right)\right]
      \right\}_{|y=\Phi^s_\pm(x_0)} \dd s \dd t_0 \dd x_0 \dd v_0.
    \end{aligned}
  \end{equation*}
  
  We then use~\eqref{eq:bound-jac-kin} to bound
  \begin{align*}
    \left| \left( \nabla_v \left[\varphi
    \left(\left( \Phi^s_\pm \right)^{-1} (y)\right)\right]
    \right)_{|y=(\Phi^s_\pm)^{-1}(x_0)} \right|
    & \lesssim_\varphi
      \left\| \nabla_{v_0} \left( \Phi^s_\pm \right)^{-1}
      \right\|_\infty  \\
    & \lesssim_\varphi s^{-\frac12},
  \end{align*}
  and deduce finally (using the bound on the matrix $A$)
  \begin{equation*}
    \cI_1^\pm \lesssim |\delta_\pm| \int_{Q^0} \int_0^1
    \left| \left[ \nabla_v f\right]
      \left( \Gamma_\pm(s) \right) \right|  s^{-\frac12} \dd s \dd t_0 \dd x_0 \dd
    v_0.
  \end{equation*}

  Now let us turn to $\cI_2^\pm$ in \eqref{eq:I1-I2}. Using the bound $|g_i''(s)|\lesssim s^{-1/2}$, we deduce 
  \begin{align*}
      \cI_2^\pm
      & = \mp\int_0^1  \int_{Q^0}  \sum_{i = 0}^1
      g_i''(s) \mathsf m_\pm^{(i)} \cdot\nabla_v
      f(\Gamma_\pm(s)) \varphi(x_0,v_0) \dd s  \dd t_0  \dd x_0 \dd v_0 \\
      & \lesssim \int_0^1 \int_{Q^0} \left|
      \left(\nabla_v f\right)\left(
        \Gamma_\pm(s) \right) \right| s^{-\frac12} \dd s \dd t_0 \dd x_0 \dd v_0.
  \end{align*}
  
  It follows from \eqref{eq:poincare-aux} and \eqref{eq:I1-I2} that
  we are left with estimating
  \begin{equation*}
    \cJ := \int_{Q^\pm} \big\{\cI^\pm_1\big\}_+ \dd
    z_\pm  + \int_{Q^\pm}  \big\{\cI^\pm_2\big\}_+ \dd
    z_\pm.
  \end{equation*}
  The previous two estimates on $\cI^\pm_1$ and $\cI^\pm_2$ imply
  \begin{equation*}
    \cJ \lesssim \bar \cJ := \int_{Q^\pm} \int_{Q^0} \int_0^1 
     \left| \left( \nabla_v f \right)
        (\Gamma_\pm(s))\right| s^{-\frac12} \dd s \dd z_\pm \dd z_0.
  \end{equation*}
  Since we are now integrating in both $Q^0$ and $Q_\pm$, we are now in a position to use a change of variable not only with the intermediate variable $z_0$, but also with the future/past variables $z_\pm$. Note that it was not possible to use the integration
  in $Q_\pm$ before because of the positive value around the $Q^0$
  integral. 
  
  We split $\bar \cJ$ as follows, given $s_0 \in (0,1)$,
  \begin{equation*}
    \begin{aligned}
      \bar \cJ & = \int_{Q^\pm} \int_{Q^0} \int_0^1
       \Big( \cdots \Big) \dd s \dd z_0 \dd z_\pm
      \\
      & = \underbrace{\int_{Q^\pm} \int_{Q^0} \int_0^{s_0}
        \Big( \cdots \Big)  \dd s \dd z_0 \dd
        z_\pm}_{=: \bar \cJ^\pm_1} + \underbrace{\int_{Q^\pm}
        \int_{Q^0} \int_{s_0}^1 \Big(
        \cdots \Big) \dd s \dd z_0 \dd z_\pm}_{=: \bar \cJ^\pm_2}.
    \end{aligned}
  \end{equation*}
  The two changes of variables on each part are represented in
  Figure~\ref{fig:change-of-var}.
  
  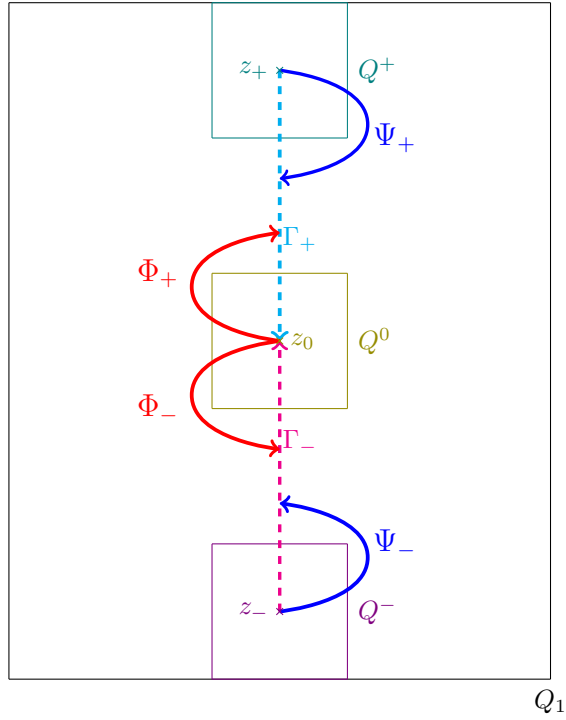
\begin{figure}{
      \centering
      \begin{tikzpicture}[scale =0.9]
        \draw [black] (-4,0) -- (4,0);
        \draw[black] (-4,-10) -- (4,-10);
        \draw [black](-4,0) -- (-4,-10);
        \draw [black](4,0) -- (4,-10) node[anchor=north, scale=1]
        {\footnotesize{$Q_1$}};
        \draw [teal](-1,0) -- (1, 0);
        \draw [teal](-1,-2) -- (1, -2);
        \draw [teal](-1,0) -- (-1, -2);
        \draw [teal](1,0) -- node[anchor=west, scale=1]
        {\footnotesize{$Q^+$}}(1, -2);
        \draw[olive] (-1,-4) -- (1, -4);
        \draw [olive](-1,-6) -- (1, -6);
        \draw [olive](-1,-4) -- (-1, -6);
        \draw [olive](1,-4) -- node[anchor=west, scale=1]
        {\footnotesize{$Q^0$}}(1, -6);
        \draw[violet] (-1,-8) -- (1, -8);
        \draw [violet](-1,-10) -- (1, -10);
        \draw [violet](-1,-8) -- (-1, -10);
        \draw [violet](1,-8)  -- node[anchor=west, scale=1]
        {\footnotesize{$Q^-$}} (1, -10);
        \draw [->, cyan, dashed, line width=0.5mm]  plot [smooth]
        coordinates {(0,-1)(0,-5)};
        \node [cyan] at (0.3, -3.5) {\footnotesize{$\Gamma_+$}};
        \draw [->, magenta, dashed, line width=0.45mm] plot
        [smooth] coordinates {(0,-9) (0,-5)};
        \node [magenta] at (0.3, -6.5) {\footnotesize{$\Gamma_-$}};
        \draw [->, blue, line width=0.45mm] plot [smooth, tension =
        1.5] coordinates {(0,-1) (1.3, -1.8)(0,-2.6)};
        \node [blue] at (1.7, -2) {$\Psi_+$};
        \draw [->, blue, line width=0.5mm] plot [smooth, tension =
        1.5] coordinates {(0,-9) (1.3, -8.2) (0,-7.4)};
        \node [blue] at (1.7, -8) {$\Psi_-$};
        \draw [->, red, line width=0.5mm] plot [smooth, tension =
        1.5] coordinates {(0,-5) (-1.3, -4.2)(0,-3.4)};
        \node [red] at (-1.8, -4) {$\Phi_+$};
        \draw [->, red, line width=0.5mm] plot [smooth, tension =
        1.5] coordinates {(0,-5) (-1.3, -5.8) (0,-6.6)};
        \node [red] at (-1.8, -6) {$\Phi_-$};
        \coordinate[label={[teal]left:\footnotesize{$z_+$}}] (z_+) at (0, -1);
        \draw [teal] plot[only marks,mark=x,mark size=2pt]  coordinates {(0, -1)};
        \coordinate[label={[olive]right:\footnotesize{$z_0$}}] (z_0) at (0, -5);
        \draw [olive] plot[only marks,mark=x,mark size=2pt]
        coordinates {(0, -5)};
        \coordinate[label={[violet]left:\footnotesize{$z_-$}}] (z_-) at (0, -9);
        \draw [violet] plot[only marks,mark=x,mark size=2pt]
        coordinates {(0, -9)};
      \end{tikzpicture}
    }
    \caption{The change of variables that we use in the proof for
      some fixed $s \in (0, 1)$. For $s \in (0,s_0)$ we use
      $\Psi_\pm$ that map $z_\pm$ onto $ \Gamma_\pm$, whereas
      for $s \in (s_0,1)$ we use $\Phi_\pm$ that map $z_0$ onto
      $\Gamma_\pm$.}\label{fig:change-of-var}
  \end{figure}
  
  The part $\bar \cJ^\pm_2$ is controlled by using the same
  change of variables $\Phi^s_\pm$ as before:
  \begin{align}
    \nonumber \bar \cJ^\pm_2
    & = \int_{Q^0} \int_{s_0}^1 s^{-\frac12}
      \left| \left(\nabla_v f\right)\left(
      y, st_0 +(1-s)t_\pm\right) \right| \frac{\dd s \dd t_0 \dd y}{\left| \det \mathfrak A^s \right|} \\
    \label{eq:J2}
    & \lesssim_{s_0} \int_{Q_1}  \left|
      \nabla_v f(z) \right| \dd z,
  \end{align}
  where we have used first that the integral in $s$ avoids the
  singularity at $s=0$, and therefore is finite.

  To control the part $\bar \cJ^\pm_1$ we parametrise
  $\Gamma_\pm(s)$ by the $z_\pm$ coordinates, for $s \in [0,s_0]$
  and $t_0, t_\pm$ all fixed:
  \begin{equation*}
    (t_\pm,x_\pm,v_\pm) \to \Gamma_\pm (s)
    =: (t_0 s +(1-s) t_\pm,\Psi^s_\pm(x_\pm)).
  \end{equation*}
  Since $s$ does not approach $1$, we can prove that this change of
  variables is not singular. The mappings $\Psi_\pm^s$ are
  determined by solving~\eqref{eq:mapping} for $x_\pm$ instead of
  $x_0$. It yields
  \begin{equation*}
    \Psi_\pm^s \( \begin{matrix} v_\pm \\ x_\pm \end{matrix} \) = \mathfrak a^s \( \begin{matrix} v_\pm \\ x_\pm \end{matrix} \) + \mathfrak b^s
  \end{equation*}
  with the matrix $\mathfrak a^s$ and vector $\mathfrak b^s$ given
  by
  \begin{align*}
    & \mathfrak a^s _\pm := \left( \Id -  \sW_\pm ^\delta (s)
      \big[\sW_\pm ^\delta (1)\big]^{-1} \right) \begin{pmatrix}
      \I_d & \OO \\
      \delta_\pm s \I_d & \I_d
    \end{pmatrix}, \\
    & \mathfrak b^s _\pm := \sW_\pm ^\delta (s) \big[\sW_\pm ^\delta
    (1)\big]^{-1} \( \begin{matrix} v_0 \\ x_0 \end{matrix} \),
  \end{align*}
  which depend only on $s$ and $(x_0,v_0)$.

  We have proved in that $\sW(s) = O(s^{2d})$, which readily implies $\sW^\delta_\pm(s) = O(s^{2d})$. Therefore, $\sW^\delta_\pm(s)$ goes to zero as $s \to 0$. Therefore, $\mathfrak a^s_\pm$ is the product of two matrices; the first one is close to identity for $s\in [0,1]$ close enough to zero and so invertible in this range, while the second matrix is a lower triangular matrix with unit diagonal, which is invertible for all $s \in [0,1]$. All in all, we deduce that $\mathfrak a^s_\pm$ is invertible with uniformly bounded inverse
  on $s \in [0,s_0]$ for $s_0>0$ small enough.

  We apply this change of variables to estimate $\bar \cJ_1^\pm$:
  \begin{equation}
    \label{eq:J1}
    \bar \cJ^\pm _1
    = \int_{Q^0} \int_0 ^{s_0} s^{-\frac12}
    \left| \left(\nabla_v f\right)\left(
        st_0 +(1-s)t_\pm,y\right) \right| \frac{\dd s \dd t_0 \dd y}{\left| \det \mathfrak a^s \right|}
    \lesssim \int_{Q_1}  \left|
      \nabla_v f(z) \right| \dd z.
  \end{equation}
  The combination of~\eqref{eq:J2} and~\eqref{eq:J1} yields
  \begin{equation*}
    \int_{Q^+} \Big(f(z_+) - \langle f\rangle_{Q^-}\Big)_+ \dd 
      z_+ \lesssim \cJ \lesssim \bar \cJ \lesssim \bar \cJ^+ _1 + \bar \cJ^- _1 + \bar \cJ^+ _2 + \bar \cJ^- _2 \lesssim \int_{Q_1}  \left| \nabla_v f(z) \right| \dd z.
  \end{equation*}
  This concludes the proof of~\eqref{gmpoi}.
\end{proof}

\subsection{Kinetic versions of the Moser and Kruzhkov methods}

A kinetic version of Moser’s approach had long seemed out of reach, but the critical kinetic trajectories constructed in~\cite{dmnz} have recently made it possible to finally extend this method to the kinetic setting. In this extension, the \emph{critical scaling} of the trajectories is essential in deriving universal bounds on the logarithm of the solution and apply the Bombieri-Giusti result, in the spirit of the paper of Moser~\cite{moser1971}.

Regarding a kinetic version of Kruzhkov's approach, we refer to~\cite{wang2009c,MR4653756} for an extension in the kinetic setting. Note, however, that these papers use a weaker form of the kinetic Poincaré inequality, with a complicated corrector, which adds layers of technical complexity to the proofs. These Kruzhkov-type proofs could be simplified by using the kinetic Poincaré inequality~\eqref{gmpoi}. 


\bibliography{references}

\begin{thebibliography}{10}

\bibitem{agrachev2019comprehensive}
{\sc A.~Agrachev, D.~Barilari, and U.~Boscain}, {\em A comprehensive
  introduction to sub-{R}iemannian geometry}, vol.~181, Cambridge University
  Press, 2019.

\bibitem{MR4776290}
{\sc D.~Albritton, S.~Armstrong, J.-C. Mourrat, and M.~Novack}, {\em
  Variational methods for the kinetic {F}okker-{P}lanck equation}, Anal. PDE,
  17 (2024), pp.~1953--2010.

\bibitem{MR164135}
{\sc A.~D. Aleksandrov}, {\em Uniqueness conditions and bounds for the solution
  of the {D}irichlet problem}, Vestnik Leningrad. Univ. Ser. Mat. Meh.
  Astronom., 18 (1963), pp.~5--29.

\bibitem{MR1765272}
{\sc R.~Alexandre, L.~Desvillettes, C.~Villani, and B.~Wennberg}, {\em Entropy
  dissipation and long-range interactions}, Arch. Ration. Mech. Anal., 152
  (2000), pp.~327--355.

\bibitem{MR3375485}
{\sc R.~Alexandre, J.~Liao, and C.~Lin}, {\em Some a priori estimates for the
  homogeneous {L}andau equation with soft potentials}, Kinet. Relat. Models, 8
  (2015), pp.~617--650.

\bibitem{MR2679369}
{\sc R.~Alexandre, Y.~Morimoto, S.~Ukai, C.-J. Xu, and T.~Yang}, {\em
  Regularizing effect and local existence for the non-cutoff {B}oltzmann
  equation}, Arch. Ration. Mech. Anal., 198 (2010), pp.~39--123.

\bibitem{anceschi2024poincare}
{\sc F.~Anceschi, H.~Dietert, J.~Guerand, A.~Loher, C.~Mouhot, and A.~Rebucci},
  {\em Poincar\'e inequality and quantitative {D}e {G}iorgi method for
  hypoelliptic operators}, 2024.
\newblock Preprint arXiv:2401.12194.

\bibitem{MR163025}
{\sc V.~I. Arnold}, {\em Proof of a theorem of {A}. {N}. {K}olmogorov on the
  preservation of conditionally periodic motions under a small perturbation of
  the {H}amiltonian}, Uspehi Mat. Nauk, 18 (1963), pp.~13--40.

\bibitem{MR244638}
{\sc D.~G. Aronson and J.~Serrin}, {\em Local behavior of solutions of
  quasilinear parabolic equations}, Arch. Rational Mech. Anal., 25 (1967),
  pp.~81--122.

\bibitem{MR152860}
{\sc J.-P. Aubin}, {\em Un th\'{e}or\`eme de compacit\'{e}}, C. R. Acad. Sci.
  Paris, 256 (1963), pp.~5042--5044.

\bibitem{auscher2024weaksolutionskolmogorovfokkerplanckequations}
{\sc P.~Auscher, C.~Imbert, and L.~Niebel}, {\em Weak solutions to
  {K}olmogorov-{F}okker-{P}lanck equations: regularity, existence and
  uniqueness}, 2024.
\newblock Preprint arXiv:2403.17464.

\bibitem{zbMATH03536702}
{\sc J.~Bergh and J.~L{\"o}fstr{\"o}m}, {\em Interpolation spaces. {An}
  introduction}, vol.~223 of Grundlehren Math. Wiss., Springer, Cham, 1976.

\bibitem{bernstein1904nature}
{\sc S.~Bernstein}, {\em Sur la nature analytique des solutions des \'equations
  aux d\'eriv\'ees partielles du second ordre}, Mathematische Annalen, 59
  (1904), pp.~20--76.

\bibitem{bobylev1976fourier}
{\sc A.~Bobylev}, {\em The {F}ourier transform method for the {B}oltzmann
  equation for {M}axwell molecules}, in Sov. Phys. Dokl, vol.~20, 1976,
  pp.~820--822.

\bibitem{boltzmann1872weitere}
{\sc L.~Boltzmann}, {\em Weitere studien {\"u}ber das w{\"a}rmegleichgewicht
  unter gasmolek{\"u}len}, Kaiserliche-königliche Hof-und Staatsdruckerei,
  1872.

\bibitem{bg1972}
{\sc E.~Bombieri and E.~Giusti}, {\em Harnack's inequality for elliptic
  differential equations on minimal surfaces}, Invent. Math., 15 (1972),
  pp.~24--46.

\bibitem{zbMATH01975967}
{\sc F.~Bouchut}, {\em Hypoelliptic regularity in kinetic equations.}, J. Math.
  Pures Appl. (9), 81 (2002), pp.~1135--1159.

\bibitem{bouin2019hypocoercivity}
{\sc E.~Bouin, J.~Dolbeault, L.~Lafleche, and C.~Schmeiser}, {\em
  Hypocoercivity and sub-exponential local equilibria}, Monatsh. Math., 194
  (2021), pp.~41--65.

\bibitem{MR46536}
{\sc R.~Caccioppoli}, {\em Limitazioni integrali per le soluzioni di
  un'equazione lineare ellitica a derivate parziali}, Giorn. Mat. Battaglini
  (4), 4(80) (1951), pp.~186--212.

\bibitem{zbMATH05913780}
{\sc L.~Caffarelli, C.~H. Chan, and A.~Vasseur}, {\em Regularity theory for
  parabolic nonlinear integral operators}, J. Am. Math. Soc., 24 (2011),
  pp.~849--869.

\bibitem{chiarenza1986harnack}
{\sc F.~Chiarenza, E.~Fabes, and N.~Garofalo}, {\em Harnack’s inequality for
  {S}chr{\"o}dinger operators and the continuity of solutions}, Proceedings of
  the American Mathematical Society, 98 (1986), pp.~415--425.

\bibitem{de1957sulla}
{\sc E.~De~Giorgi}, {\em Sulla differenziabilit\`a e l'analiticit\`a delle
  estremali degli integrali multipli regolari}, Mem. Accad. Sci. Torino. Cl.
  Sci. Fis. Mat. Nat. (3), 3 (1957), pp.~25--43.

\bibitem{de2016john}
{\sc C.~De~Lellis and L.~Sz{\'e}kelyhidi~Jr}, {\em John {N}ash’s nonlinear
  iteration}, Lecture notes,  (2016).

\bibitem{MR1737547}
{\sc L.~Desvillettes and C.~Villani}, {\em On the spatially homogeneous
  {L}andau equation for hard potentials. {I}. {E}xistence, uniqueness and
  smoothness}, Comm. Partial Differential Equations, 25 (2000), pp.~179--259.

\bibitem{MR1737548}
\leavevmode\vrule height 2pt depth -1.6pt width 23pt, {\em On the spatially
  homogeneous {L}andau equation for hard potentials. {II}. {$H$}-theorem and
  applications}, Comm. Partial Differential Equations, 25 (2000), pp.~261--298.

\bibitem{dmnz}
{\sc H.~Dietert, C.~Mouhot, L.~Niebel, and R.~Zacher}, {\em Critical
  trajectories in kinetic geometry}, 2025.
\newblock In progress.

\bibitem{MR2597943}
{\sc L.~C. Evans}, {\em Partial differential equations}, vol.~19 of Graduate
  Studies in Mathematics, American Mathematical Society, Providence, RI,
  second~ed., 2010.

\bibitem{zbMATH08010281}
{\sc L.~C. Evans and R.~F. Gariepy}, {\em Measure theory and fine properties of
  functions}, Textb. Math., Boca Raton, FL: CRC Press, 2nd edition~ed., 2025.

\bibitem{MR3032092}
{\sc D.~Faraco and K.~M. Rogers}, {\em The {S}obolev norm of characteristic
  functions with applications to the {C}alder\'{o}n inverse problem}, Q. J.
  Math., 64 (2013), pp.~133--147.

\bibitem{fournier2010uniqueness}
{\sc N.~Fournier}, {\em Uniqueness of bounded solutions for the homogeneous
  {L}andau equation with a {C}oulomb potential}, Communications in Mathematical
  Physics, 299 (2010), pp.~765--782.

\bibitem{MR2502525}
{\sc N.~Fournier and H.~Gu\'{e}rin}, {\em Well-posedness of the spatially
  homogeneous {L}andau equation for soft potentials}, J. Funct. Anal., 256
  (2009), pp.~2542--2560.

\bibitem{zbMATH07986618}
{\sc F.~Golse}, {\em Regularity of solutions of {Fokker}-{Planck} equations
  with rough coefficients}, Riv. Mat. Univ. Parma (N.S.), 15 (2024),
  pp.~143--173.

\bibitem{golse:hal-02145096}
{\sc F.~Golse, M.~P. Gualdani, C.~Imbert, and A.~Vasseur}, {\em Partial
  regularity in time for the space-homogeneous {Landau} equation with {Coulomb}
  potential}, Ann. Sci. {\'E}c. Norm. Sup{\'e}r. (4), 55 (2022),
  pp.~1575--1611.

\bibitem{golse2019harnack}
{\sc F.~Golse, C.~Imbert, C.~Mouhot, and A.~Vasseur}, {\em Harnack inequality
  for kinetic {Fokker}-{Planck} equations with rough coefficients and
  application to the {Landau} equation}, Ann. Sc. Norm. Super. Pisa, Cl. Sci.
  (5), 19 (2019), pp.~253--295.

\bibitem{MR923047}
{\sc F.~Golse, P.-L. Lions, B.~Perthame, and R.~Sentis}, {\em Regularity of the
  moments of the solution of a transport equation}, J. Funct. Anal., 76 (1988),
  pp.~110--125.

\bibitem{MR808622}
{\sc F.~Golse, B.~Perthame, and R.~Sentis}, {\em Un r\'{e}sultat de
  compacit\'{e} pour les \'{e}quations de transport et application au calcul de
  la limite de la valeur propre principale d'un op\'{e}rateur de transport}, C.
  R. Acad. Sci. Paris S\'{e}r. I Math., 301 (1985), pp.~341--344.

\bibitem{zbMATH06313565}
{\sc L.~Grafakos}, {\em Classical {Fourier} analysis}, vol.~249 of Grad. Texts
  Math., New York, NY: Springer, 3rd ed.~ed., 2014.

\bibitem{MR2784329}
{\sc P.~T. Gressman and R.~M. Strain}, {\em Global classical solutions of the
  {B}oltzmann equation without angular cut-off}, J. Amer. Math. Soc., 24
  (2011), pp.~771--847.

\bibitem{MR4398231}
{\sc J.~Guerand}, {\em Quantitative parabolic regularity \`a la {D}e {G}iorgi},
  in S\'{e}minaire {L}aurent {S}chwartz---\'{E}quations aux d\'{e}riv\'{e}es
  partielles et applications. {A}nn\'{e}e 2018--2019, Ed. \'{E}c. Polytech.,
  Palaiseau, [2018] \copyright 2018, pp.~Exp. No. VI, 21.

\bibitem{guerand2020quantitativeregularityparabolicgiorgi}
{\sc J.~Guerand}, {\em Quantitative regularity for parabolic {D}e {G}iorgi
  classes}, 2020.
\newblock Preprint arXiv:1903.07421.

\bibitem{MR4653756}
{\sc J.~Guerand and C.~Imbert}, {\em Log-transform and the weak {H}arnack
  inequality for kinetic {F}okker-{P}lanck equations}, J. Inst. Math. Jussieu,
  22 (2023), pp.~2749--2774.

\bibitem{guerand2022quantitative}
{\sc J.~Guerand and C.~Mouhot}, {\em Quantitative {D}e {G}iorgi methods in
  kinetic theory}, J. \'{E}c. polytech. Math., 9 (2022), pp.~1159--1181.

\bibitem{guillen2025landauequationdoesblow}
{\sc N.~Guillen and L.~Silvestre}, {\em The {L}andau equation does not blow
  up}, 2025.
\newblock Preprint arXiv:2311.09420.

\bibitem{MR2777537}
{\sc Q.~Han and F.~Lin}, {\em Elliptic partial differential equations}, vol.~1
  of Courant Lecture Notes in Mathematics, Courant Institute of Mathematical
  Sciences, New York; American Mathematical Society, Providence, RI,
  second~ed., 2011.

\bibitem{henderson2020c}
{\sc C.~Henderson and S.~Snelson}, {\em ${C}^\infty$ smoothing for weak
  solutions of the inhomogeneous {L}andau equation}, Archive for Rational
  Mechanics and Analysis, 236 (2020), pp.~113--143.

\bibitem{Hilbert1900}
{\sc D.~Hilbert}, {\em Mathematische {P}robleme}, Nachrichten von der
  Gesellschaft der Wissenschaften zu Göttingen, Mathematisch-Physikalische
  Klasse, 1900 (1900), pp.~253--297.

\bibitem{MR3930576}
{\sc H.~Holden and R.~Piene}, eds., {\em The {A}bel {P}rize 2013--2017},
  Springer, Cham, 2019.

\bibitem{hormander1967hypoelliptic}
{\sc L.~H{\"o}rmander}, {\em Hypoelliptic second order differential equations},
  Acta Mathematica, 119 (1967), pp.~147--171.

\bibitem{zbMATH01950198}
\leavevmode\vrule height 2pt depth -1.6pt width 23pt, {\em The analysis of
  linear partial differential operators. {I}: {Distribution} theory and
  {Fourier} analysis.}, Class. Math., Berlin: Springer, reprint of the 2nd
  edition 1990~ed., 2003.

\bibitem{imbert2021schauder}
{\sc C.~Imbert and C.~Mouhot}, {\em The {S}chauder estimate in kinetic theory
  with application to a toy nonlinear model}, Ann. H. Lebesgue, 4 (2021),
  pp.~369--405.

\bibitem{imbert2020decay}
{\sc C.~Imbert, C.~Mouhot, and L.~Silvestre}, {\em Decay estimates for large
  velocities in the {Boltzmann} equation without cutoff}, J. {\'E}c. Polytech.,
  Math., 7 (2020), pp.~143--184.

\bibitem{imbert2019weak}
{\sc C.~Imbert and L.~Silvestre}, {\em The weak {H}arnack inequality for the
  {B}oltzmann equation without cut-off}, J. Eur. Math. Soc. (JEMS), 22 (2020),
  pp.~507--592.

\bibitem{imbert2021schauderii}
\leavevmode\vrule height 2pt depth -1.6pt width 23pt, {\em The {S}chauder
  estimate for kinetic integral equations}, Anal. PDE, 14 (2021), pp.~171--204.

\bibitem{imbert2022global}
{\sc C.~Imbert and L.~E. Silvestre}, {\em Global regularity estimates for the
  {B}oltzmann equation without cut-off}, J. Amer. Math. Soc., 35 (2022),
  pp.~625--703.

\bibitem{john1961functions}
{\sc F.~John and L.~Nirenberg}, {\em On functions of bounded mean oscillation},
  Communications on pure and applied Mathematics, 14 (1961), pp.~415--426.

\bibitem{kolmogoroff1934zufallige}
{\sc A.~N. Kolmogorov}, {\em Zuf\"{a}llige {B}ewegungen (zur {T}heorie der
  {B}rownschen {B}ewegung)}, Ann. of Math. (2), 35 (1934), pp.~116--117.

\bibitem{MR68687}
\leavevmode\vrule height 2pt depth -1.6pt width 23pt, {\em On conservation of
  conditionally periodic motions for a small change in {H}amilton's function},
  Dokl. Akad. Nauk SSSR (N.S.), 98 (1954), pp.~527--530.

\bibitem{kruzhkov1963priori}
{\sc S.~N. Kruzhkov}, {\em A priori bounds for generalized solutions of
  second-order elliptic and parabolic equations}, in Doklady Akademii Nauk,
  vol.~150, Russian Academy of Sciences, 1963, pp.~748--751.

\bibitem{MR171086}
\leavevmode\vrule height 2pt depth -1.6pt width 23pt, {\em A priori bounds and
  some properties of solutions of elliptic and parabolic equations}, Mat. Sb.
  (N.S.), 65(107) (1964), pp.~522--570.
\newblock Translation into English by Zuckerman published by the AMS.

\bibitem{MR420016}
{\sc N.~V. Krylov}, {\em Sequences of convex functions, and estimates of the
  maximum of the solution of a parabolic equation}, Sibirsk. Mat. \v{Z}., 17
  (1976), pp.~290--303, 478.

\bibitem{zbMATH00947827}
\leavevmode\vrule height 2pt depth -1.6pt width 23pt, {\em Lectures on elliptic
  and parabolic equations in {H{\"o}lder} spaces}, vol.~12 of Grad. Stud.
  Math., Providence, RI: AMS, American Mathematical Society, 1996.

\bibitem{krylov1980certain}
{\sc N.~V. Krylov and M.~V. Safonov}, {\em A certain property of solutions of
  parabolic equations with measurable coefficients}, Izvestiya Rossiiskoi
  Akademii Nauk. Seriya Matematicheskaya, 44 (1980), pp.~161--175.

\bibitem{Landau1936}
{\sc L.~D. Landau}, {\em Kinetic equation for the case of coulomb interaction},
  Phys. Z. Sowjetunion,  (1936).

\bibitem{MR1487894}
{\sc E.~M. Landis}, {\em Second order equations of elliptic and parabolic
  type}, vol.~171 of Translations of Mathematical Monographs, American
  Mathematical Society, Providence, RI, 1998.
\newblock Translated from the 1971 Russian original by Tamara Rozhkovskaya.

\bibitem{MR259693}
{\sc J.-L. Lions}, {\em Quelques m\'{e}thodes de r\'{e}solution des probl\`emes
  aux limites non lin\'{e}aires}, Dunod, Paris; Gauthier-Villars, Paris, 1969.

\bibitem{MR4688651}
{\sc A.~Loher}, {\em Quantitative {D}e {G}iorgi methods in kinetic theory for
  non-local operators}, J. Funct. Anal., 286 (2024), pp.~Paper No. 110312, 67.

\bibitem{maxwell1867iv}
{\sc J.~C. Maxwell}, {\em On the dynamical theory of gases}, Philosophical
  transactions of the Royal Society of London,  (1867), pp.~49--88.

\bibitem{zbMATH07604914}
{\sc C.~Mooney}, {\em Hilbert's 19th problem revisited}, Boll. Unione Mat.
  Ital., 15 (2022), pp.~483--501.

\bibitem{morrey1938solutions}
{\sc C.~B. Morrey}, {\em On the solutions of quasi-linear elliptic partial
  differential equations}, Transactions of the American Mathematical Society,
  43 (1938), pp.~126--166.

\bibitem{moser1960new}
{\sc J.~Moser}, {\em A new proof of {D}e {G}iorgi's theorem concerning the
  regularity problem for elliptic differential equations}, Communications on
  Pure and Applied Mathematics, 13 (1960), pp.~457--468.

\bibitem{MR159138}
\leavevmode\vrule height 2pt depth -1.6pt width 23pt, {\em On {H}arnack's
  theorem for elliptic differential equations}, Comm. Pure Appl. Math., 14
  (1961), pp.~577--591.

\bibitem{MR147741}
\leavevmode\vrule height 2pt depth -1.6pt width 23pt, {\em On invariant curves
  of area-preserving mappings of an annulus}, Nachr. Akad. Wiss. G\"{o}ttingen
  Math.-Phys. Kl. II, 1962 (1962), pp.~1--20.

\bibitem{moser1964harnack}
\leavevmode\vrule height 2pt depth -1.6pt width 23pt, {\em A {H}arnack
  inequality for parabolic differential equations}, Communications on pure and
  applied mathematics, 17 (1964), pp.~101--134.

\bibitem{MR206461}
\leavevmode\vrule height 2pt depth -1.6pt width 23pt, {\em A rapidly convergent
  iteration method and non-linear differential equations. {II}}, Ann. Scuola
  Norm. Sup. Pisa Cl. Sci. (3), 20 (1966), pp.~499--535.

\bibitem{MR199523}
\leavevmode\vrule height 2pt depth -1.6pt width 23pt, {\em A rapidly convergent
  iteration method and non-linear partial differential equations. {I}}, Ann.
  Scuola Norm. Sup. Pisa Cl. Sci. (3), 20 (1966), pp.~265--315.

\bibitem{moser1967correction}
\leavevmode\vrule height 2pt depth -1.6pt width 23pt, {\em Correction to: "{A}
  {H}arnack inequality for parabolic differential equations"}, Communications
  on Pure and Applied Mathematics, 20 (1967), pp.~231--236.

\bibitem{moser1971}
\leavevmode\vrule height 2pt depth -1.6pt width 23pt, {\em On a pointwise
  estimate for parabolic differential equations}, Communications on Pure and
  Applied Mathematics, 24 (1971), pp.~727--740--236.

\bibitem{MR2322149}
{\sc C.~Mouhot and R.~M. Strain}, {\em Spectral gap and coercivity estimates
  for linearized {B}oltzmann collision operators without angular cutoff}, J.
  Math. Pures Appl. (9), 87 (2007), pp.~515--535.

\bibitem{MR75639}
{\sc J.~Nash}, {\em The imbedding problem for {R}iemannian manifolds}, Ann. of
  Math. (2), 63 (1956), pp.~20--63.

\bibitem{nash1958continuity}
\leavevmode\vrule height 2pt depth -1.6pt width 23pt, {\em Continuity of
  solutions of parabolic and elliptic equations}, American Journal of
  Mathematics, 80 (1958), pp.~931--954.

\bibitem{niebel2023kinetic}
{\sc L.~Niebel and R.~Zacher}, {\em On a kinetic {P}oincar\'{e} inequality and
  beyond}, J. Funct. Anal., 289 (2025), pp.~Paper No. 110899, 18.

\bibitem{nz_trajMBG_2025}
\leavevmode\vrule height 2pt depth -1.6pt width 23pt, {\em A trajectorial
  interpretation of {Moser}'s proof of the {Harnack} inequality}, To appear in
  Annali della Scuola Normale Superiore di Pisa. Classe di Scienze. Serie V,
  (2025).

\bibitem{nirenberg1963solvability}
{\sc L.~Nirenberg and F.~Treves}, {\em Solvability of a first order linear
  partial differential equation}, Communications on Pure and Applied
  Mathematics, 16 (1963), pp.~331--351.

\bibitem{pascucci2004moser}
{\sc A.~Pascucci and S.~Polidoro}, {\em The {M}oser's iterative method for a
  class of ultraparabolic equations}, Communications in Contemporary
  Mathematics, 6 (2004), pp.~395--417.

\bibitem{petrowsky1939analyticite}
{\sc I.~G. Petrovsky}, {\em Sur l'analyticit\'{e} des solutions des syst\`emes
  d'\'{e}quations diff\'{e}rentielles}, Rec. Math. N.S. [Mat. Sbornik], 5(47)
  (1939), pp.~3--70.

\bibitem{schauder1934lineare}
{\sc J.~Schauder}, {\em {\"U}ber lineare elliptische {D}ifferentialgleichungen
  zweiter {O}rdnung.}, Mathematische Zeitschrift, 38 (1934), pp.~257--282.

\bibitem{silvestre2017upper}
{\sc L.~Silvestre}, {\em Upper bounds for parabolic equations and the {L}andau
  equation}, Journal of Differential Equations, 262 (2017), pp.~3034--3055.

\bibitem{LANDAUPAPERS}
{\sc D.~Ter~Haar}, {\em Collected Papers of L.D. Landau}, Gordon and Breach,
  Science Publishers, Inc., 1965.

\bibitem{MR226168}
{\sc N.~S. Trudinger}, {\em Pointwise estimates and quasilinear parabolic
  equations}, Comm. Pure Appl. Math., 21 (1968), pp.~205--226.

\bibitem{vasseur2016giorgi}
{\sc A.~F. Vasseur}, {\em The {D}e {G}iorgi method for elliptic and parabolic
  equations and some applications}, Lectures on the analysis of nonlinear
  partial differential equations, 4 (2016).

\bibitem{peccot-villani}
{\sc C.~Villani}, {\em Propriétés qualitatives des solutions de l'équation
  de {B}oltzmann}.
\newblock Cours {P}eccot du {C}ollège de {F}rance (2003).

\bibitem{MR1646502}
{\sc C.~Villani}, {\em On the spatially homogeneous {L}andau equation for
  {M}axwellian molecules}, Math. Models Methods Appl. Sci., 8 (1998),
  pp.~957--983.

\bibitem{wang2009c}
{\sc W.~Wang and L.~Zhang}, {\em The $\mathrm{C}^\alpha$ regularity of a class
  of non-homogeneous ultraparabolic equations}, Science in China Series A:
  Mathematics, 52 (2009), pp.~1589--1606.

\bibitem{MR3158719}
{\sc K.-C. Wu}, {\em Global in time estimates for the spatially homogeneous
  {L}andau equation with soft potentials}, J. Funct. Anal., 266 (2014),
  pp.~3134--3155.

\end{thebibliography}
\bibliographystyle{siam}

\end{document}